\documentclass[11pt]{article}

\usepackage{geometry}
\usepackage{floatrow}            
\usepackage{amsthm}
\usepackage{enumitem}
\usepackage{graphicx}
\usepackage{amssymb}
\usepackage{epstopdf}
\usepackage{amsmath}
\usepackage{amsfonts}
\usepackage{placeins}
\usepackage{caption}
\usepackage{subcaption}
\usepackage[longnamesfirst]{natbib}
\bibpunct[ ]{(}{)}{,}{a}{}{,}

\usepackage{amsmath,xcolor}
\definecolor{ao(english)}{rgb}{0.0, 0.5, 0.0}
\usepackage{hyperref}
\hypersetup{
	colorlinks=true,
	linkcolor=ao(english),
	citecolor=ao(english),
	filecolor=magenta,      
	urlcolor=cyan,
	linktoc=page
}

\usepackage{floatrow}

\theoremstyle{plain}
\newtheorem{theorem}{Theorem}[section]
\newtheorem{corollary}[theorem]{Corollary}
\newtheorem{conj}[theorem]{Conjecture}
\newtheorem{problem}[theorem]{Problem}
\newtheorem{question}[theorem]{Question}
\newtheorem{prop}[theorem]{Proposition}
\newtheorem{lemma}[theorem]{Lemma}

\theoremstyle{definition}
\newtheorem{remark}{Remark}[section]
\newtheorem{definition}[theorem]{Definition}

\theoremstyle{plain}

\newcommand{\nat}{\mathbb{N}}

\newcommand{\N}{\mathbb{N}}

\newcommand{\prob}{\mathbb{P}}
\newcommand{\p}{\mathbb{P}}
\newcommand{\E}{\mathbb{E}}
\newcommand{\expt}{\mathbb{E}}
\newcommand{\indic}{\mathbf{1}}

\newcommand{\floor}[1]{{\left\lfloor #1 \right\rfloor}}
\newcommand{\ceil}[1]{{\left\lceil #1 \right\rceil}}

\newcommand{\sset}{\subset}

\newcommand{\al}{\alpha}
\newcommand{\Om}{\Omega}
\newcommand{\mathforall}{\text{ for all }}

\newcommand{\mathand}{\;\text{and}\;}

\newcommand{\mathor}{\;\text{or}\;}

\newcommand{\mathas}{\;\text{as}\;}

\newcommand{\ga}{\gamma}
\newcommand{\Ga}{\Gamma}
\newcommand{\ep}{\epsilon}
\newcommand{\om}{\omega}

\newcommand{\de}{\delta}

\newcommand{\sig}{\sigma}

\newcommand{\scrA}{\mathcal{A}}

\newcommand{\scrG}{\mathcal{G}}

\newcommand{\scrK}{\mathcal{K}}
\newcommand{\scrQ}{\mathcal{Q}}

\newcommand{\scrR}{\mathcal{R}}
\newcommand{\scrL}{\mathcal{L}}
\newcommand{\scrH}{\mathcal{H}}
\newcommand{\scrS}{\mathcal{S}}
\newcommand{\scrW}{\mathcal{W}}
\newcommand{\scrB}{\mathcal{B}}
\newcommand{\scrJ}{\mathcal{J}}

\newcommand{\card}[1]{\left\vert #1 \right\vert}

\newcommand{\close}[1]{\mkern 1.5mu\overline{\mkern-1.5mu#1\mkern-1.5mu}\mkern 1.5mu}

\newcommand{\Z}{\mathbb{Z}}

\newcommand{\Q}{\mathbb{Q}}
\newcommand{\R}{\mathbb{R}}

\newcommand{\eqd}{\stackrel{d}{=}}

\newcommand{\cvgd}{\stackrel{d}{\to}}

\newcommand{\X}{\times}

\newcommand{\cvgdown}{\downarrow}

\newcommand{\smin}{\setminus}
\newcommand{\lf}{\left}
\newcommand{\rg}{\right}

\newcommand{\LP}{\rightarrow}

\author{Duncan Dauvergne \and Janosch Ortmann \and B\'alint Vir\'ag }

\title{The directed landscape}

\begin{document}
\maketitle

\begin{abstract}
The conjectured limit of last passage percolation is a scale-invariant, independent, stationary increment process with respect to metric composition. We prove this for Brownian last passage percolation.  We construct the Airy sheet and characterize it in terms of the Airy line ensemble.  We also show that last passage geodesics converge to random functions with H\"older-$2/3^-$ continuous paths. This work completes the construction of the central object in the Kardar-Parisi-Zhang universality class, the directed landscape.
\end{abstract}

\tableofcontents

\section{Introduction}

\FloatBarrier
\begin{figure}
	\centering
\includegraphics[width=4in]{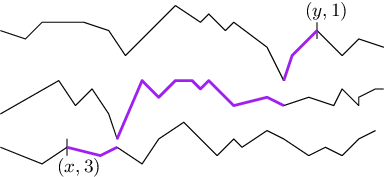}
\caption{An example of last passage across three functions. The purple path is the last passage path from $(x, 3)$ to $(y, 1)$. It can be viewed as either maximizing the sum of increments along the path, or minimizing the sum of gaps. We will always think of our sequences of functions as being labelled so that $f_{i}$ sits above $f_{i+1}$. Our notions of `right' and `left' in the paper are with respect to this picture.}
\label{fig:LP-basic}
\end{figure}
		
For a sequence of differentiable real-valued functions $f = (\dots, f_{-1}, f_0, f_1, \dots)$ with domain $\R$, and coordinates $x\le y$ and $n\le m$, define the last passage value as
\begin{equation}
  \label{E:lpintro}
 f[(x,m)\LP (y,n)]=\sup_\pi \int_x^y f'_{\pi(t)}(t) \,dt.
\end{equation}
Here the supremum is over nonincreasing functions $\pi:[x,y]\to \mathbb Z$ with $\pi(x)=m$ and $\pi(y)=n$. The integral is just a sum of increments of $f$ (see Figure \ref{fig:LP-basic}), so the same can be defined for continuous $f$, in particular for a sequence $B$ of independent two-sided Brownian motions. This model, Brownian last passage percolation, is a representative of a class of models that have been the focus of intense research in recent years.
By continuity, optimizing paths exist in \eqref{E:lpintro}; let $\pi_n$ denote one for  $B[(0,n)\LP(1,1)]$. As one of the highlights of this paper, we show that $\pi_n$ has a scaling limit.

\begin{theorem}\label{T:directed-geodesic-simple}
There exists a random continuous function $\Pi:[0, 1] \to \R$ and a new coupling of all the paths $\pi_n$ such that
$$
\sup_{s \in [0, 1]} \lf|\frac{\pi_{n}(s)-n(1-s)}{2n^{2/3}} - \Pi(s) \rg| \to 0 \qquad \text{ almost surely. }
$$
\end{theorem}
The limit $\Pi$ is the directed geodesic, a random H\"older-$2/3^-$ continuous function defined in terms of a new limiting object, the directed landscape. The directed landscape is the full four-parameter scaling limit of Brownian last passage percolation, see Definition \ref{D:directed-landscape-i}. To describe it completely, we must first discuss the parabolic Airy line ensemble.

\medskip
The parabolic Airy line ensemble is a random sequence of ordered functions $\scrA_1 > \scrA_2 > \dots$. The shifted process $\{\scrA_i(t) + t^2: i \in \N\}$, constructed by \cite{prahofer2002scale} via a determinantal formula, is stationary. The process $\scrA_1(t) + t^2$ is known as the Airy process (sometimes Airy$_2$) and describes the large-$n$ scaling limit of the function $y\mapsto B[(0,n) \LP (y,1)]$. The remaining lines have interpretations in terms of last passage percolation with multiple disjoint paths. The process $\scrA$ satisfies an important Brownian Gibbs property which allows for a probabilistic understanding of the object. This was shown by
\cite{CH}. They used this property to rigorously show that Airy lines are nonintersecting and locally absolutely continuous with respect to Brownian motion with variance $2$.
For brevity, in the remainder of the paper we often omit the word parabolic and simply refer to $\scrA$ as the Airy line ensemble.

\medskip

The Airy line ensemble doubles as the limiting eigenvalue process of Brownian motion on Hermitian matrices. Construction of the Airy sheet, the scaling limit of the two-parameter function $(x,y)\mapsto B[(x,n) \LP (y,1)]$
(conjectured in \citet*{corwin2015renormalization})
does not follow from the integrable methods that give convergence to the Airy line ensemble. This is partly because the Airy sheet seems to be fundamentally different from random matrix limits. As the first step in our construction of the directed landscape, we show that the Airy sheet law can be described in terms of last passage percolation across the Airy line ensemble.

\begin{definition}
\label{D:airy-sheet-i} The \textbf{Airy sheet} is a random continuous function $\scrS:\mathbb R^2\to \mathbb R$ so that
\begin{enumerate}[label=(\roman*)]
\item
$\scrS$ has the same law as  $\scrS(\cdot+t, \cdot+t)$ for all $t\in \mathbb R$.
\item $\scrS$ can be coupled with a parabolic Airy line ensemble so that $\scrS(0,\cdot)=\scrA_1(\cdot)$ and for all $(x, y, z) \in \Q^+ \X \Q^2$ there exists a random integer $K_{x,y,z}$ such that for all $k\ge K_{x,y,z}$ almost surely
\begin{equation}
\label{E:Airy-sheet-i}
\scrA[(-\sqrt{k/(2x)}, k) \LP (z,1) ] -  \scrA[(-\sqrt{k/(2x)}, k) \LP (y,1) ]=\scrS(x, z) - \scrS(x, y).
\end{equation}

\end{enumerate}
\end{definition}

\begin{theorem}
	\label{T:intro-sheet}
	The Airy sheet exists and is unique in law. Moreover, for every $n$, there exists a coupling so that
	$$
	B[(2x/n^{1/3},n) \LP (1+2y/n^{1/3},1)] = 2\sqrt{n} + 2(y - x)n^{1/6} + n^{-1/6} (\scrS+o_n)(x,y),
	$$
	where $o_n$ are random functions asymptotically small in the sense that on every compact set $K\subset \R^2$ there exists $a>1$ with
	$\E a^{\sup_K |o_n|^{3/2}}\to 1$.
\end{theorem}

\begin{remark} \ \\ \vspace{-1.2em}
\begin{enumerate}
  \item We prove in Proposition \ref{P:uniq-sheet} that Definition \ref{D:airy-sheet-i} uniquely determines a probability measure on random continuous functions. In other words, the Airy sheet, as defined here, exists and is unique in law. The Airy sheet exists because it is the limit of Brownian last passage percolation, see Theorem \ref{T:airy-sheet}. Theorem \ref{T:intro-sheet} is the combination of these two results.
 
  \item Equation \eqref{E:Airy-sheet-i} can be loosely interpreted as saying that the Airy sheet value $\scrS(x, y)$ is the renormalized limit, as $k \to \infty$, of the last passage value in $\scrA$ from $(-\sqrt{k/(2x)}, k)$ to $(y, 1)$. While we were not able to prove that such a limit exists, we believe that such a description is possible, see Conjecture \ref{P:det-correction}. Instead, we make sense of this picture by analyzing differences of last passage values. See Remark \ref{R:differences} for more discussion about why it is easier to look at differences.
  \item We show in the preprint \cite{dauvergne2021scaling} that the Airy sheet is also the limit of classical integrable models of last passage percolation: geometric, exponential, and Poisson models. We expect it to be a universal limit object in the Kardar-Parisi-Zhang (KPZ) universality class, see \cite{corwin2016kardar} for an informal description. We prove convergence of Brownian last passage percolation to the Airy sheet in this paper. However, one consequence of the work \cite{dauvergne2021scaling} is that any limiting formula established in one of these models will apply to all others as well.
  \item Definition \ref{D:airy-sheet-i} is not in terms of explicit formulas for distributions, which has traditionally been the main approach to KPZ limits. A celebrated example of such a definition in a different context is that of the Schramm-Loewner evolution. As is the case there, what may be more desirable than exact formulas is a set of tools to work directly with this limiting object. In this paper we concentrate on establishing the Airy sheet and the directed landscape as a limit; in upcoming work we use the present description of the Airy sheet and the directed landscape to understand the limiting geometry.
  \item Tightness for the Airy sheet limit for certain models is known, e.g. see \cite{pimentel2017local}; a short proof for Brownian last passage percolation is provided in Lemma \ref{L:Sn-tight}. Uniqueness of the limit is much harder, and is one of the main results of this paper.  Theorem \ref{T:intro-sheet} says that Brownian last passage percolation looks the same on all scales in a precise sense.
  \item The process $\scrS(x,y)+(x-y)^2$ is stationary in both variables, see Lemma \ref{L:airy-facts}, and $\scrS(0,0)$ has GUE Tracy-Widom distribution. The GUE Tracy-Widom limit for $\scrS(0,0)$ is well known; for Brownian last passage percolation it was independently shown by \cite{baryshnikov2001gues} and \cite{gravner2001limit}.
  \item By monotonicity, \eqref{E:Airy-sheet-i} is equivalent to the following Busemann function definition. For every triple $x, y, z \in \R^+ \X \R^2$, the left hand side of \eqref{E:Airy-sheet-i} converges to the right hand side as $k \to \infty$, see Remark \ref{R:busemann-def}. Busemann functions have been used previously in last passage percolation to study problems around infinite geodesics, e.g. see \cite{cator2012busemann}, \cite{georgiou2017stationary}.
\end{enumerate}
\end{remark}


A direct consequence of Theorem \ref{T:intro-sheet} is the celebrated 1-2-3 (or KPZ) scaling for Airy sheets. Define the \textbf{Airy sheet of scale $s$} by
$$
\scrS_s(x, y) = s \,\scrS(x/s^{2}, y/s^{2}).
$$
Let $\scrS_s,\, \scrS_t$ be independent Airy sheets of scale $s$ and $t$. Then the \textbf{metric composition} is an Airy sheet of
scale $r$:
$$
\scrS_r(x,z)= \max_{y \in \R} \scrS_s(x,y)+\scrS_t(y,z), \qquad \text{ with  } r^3=s^3+t^3.
$$
The metric composition law is a semigroup property for the max-plus algebra. The Airy sheet is an analogue of the Gaussian distribution in this semigroup, inspiring a definition of the analogue of Brownian motion there. It is natural to have a parameter space, directed $\R^4$,
$$
\mathbb R^4_\uparrow=\{(x,s;y,t)\in \R^4:s<t\}
$$
representing increments from time $s$ to time $t$. We will think of $\R^4_\uparrow$ as representing ordered pairs of points in spacetime with a one-dimensional space. The coordinates $x$ and $y$ are spatial and the coordinates $s$ and $t$ are temporal.

\begin{definition}\label{D:directed-landscape-i}The {\bf directed landscape} is a random continuous function $\scrL:\mathbb R^4_\uparrow \to \mathbb R$ satisfying the  metric composition law
\begin{equation}
\label{E:land-metric}
\scrL(x,r;y,t)=\max_{z\in \mathbb R} \scrL(x,r;z,s)+\scrL(z,s;y,t) \qquad \text{ for all }\; (x,r;y,t) \in  \mathbb R^4_\uparrow,\; s\in (r,t),
\end{equation}
and with the property  that $\scrL(\cdot,t_i;\cdot,t_i+s_i^3)$
are independent Airy sheets of scale $s_i$ for any set of disjoint time intervals $(t_i,t_i+s_i^3)$.
\end{definition}

\FloatBarrier

\begin{figure}
	\centering
		\includegraphics[width=4in]{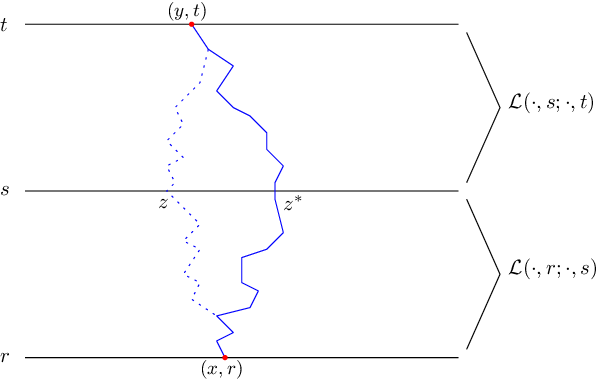}
		\caption{An illustration of the metric composition law for the directed landscape. The geodesic from $(x, r)$ to $(y, t)$ passes through a point $(z^*, s)$ for some $z^* \in \R$, giving that the right hand side of \eqref{E:land-metric} is greater than or equal the left. See \eqref{E:geodesic-definition} and surrounding discussion for a precise definition of a geodesic in $\scrL$. Also, for any $z \in \R$, a concatenation of the geodesic from $(x, r)$ to $(z, s)$ with the geodesic from $(z, s)$ to $(y, t)$ gives a candidate for a maximizing path from $(x, r)$ to $(y, t)$, yielding the opposite inequality. The original Brownian lines that give rise to $\scrL(\cdot, r; \cdot, s)$ and $\scrL(\cdot, s; \cdot, t)$ are independent. This allows us to build up $\scrL$ at any finite set of times using metric composition with independent increments, and hence construct the directed landscape from independent Airy sheets analogously to the construction of Brownian motion.}
		\label{fig:LP-metric}
\end{figure}

\begin{remark} \ \\ \vspace{-1.2em}
\begin{enumerate}
  \item We will prove in Section \ref{S:landscape} that Definition \ref{D:directed-landscape-i} uniquely determines a probability measure on random continuous functions from $\mathbb R^4_\uparrow$ to $\mathbb R$. See Figure \ref{fig:LP-metric} for an illustration of the construction.
  \item We show in the preprint \cite{dauvergne2021scaling} that the directed landscape is also the limit of classical integrable models of last passage percolation: geometric, exponential, and Poisson models. We expect it to be a universal limit object in the KPZ universality class, capturing all important limiting information.
  We prove the Brownian last passage case in this paper. One consequence of the work \cite{dauvergne2021scaling} is that any limiting formula established in one of these models will apply to all others as well.
  \item Independent increment processes on semigroups are more complicated than those on groups; for example, for Brownian motion, $B(s)-B(0)$ and $B(t)-B(0)$ clearly determine the increment $B(s)-B(t)$. In semigroups, the increments cannot be computed in this way and have to be specified for all pairs of times $s<t$. This what the directed landscape does for the metric composition semigroup.
  \item The two-time formula of \cite{johansson2017two} gives the joint distribution of any pair of directed landscape values of the form $\big(\scrL(x,r;y_1,s_1),\scrL(x,r;y_2,s_2)\big)$. \cite{johansson2019multi} and \cite{liu2019multi} independently extended this to multiple endpoints $(y_i,s_i)$.
  \item The KPZ fixed point of \citet*{matetski2016kpz} is a Markov process in $t$ which can be written in terms of the directed landscape and its initial condition $h_0$ as
  \begin{equation}
  \label{E:hty}
 h_t(y) = \sup_{x\in \mathbb R} h_0(x)+\scrL(x,0;y,t).
  \end{equation}
  \citet*{matetski2016kpz} first established the KPZ fixed point $h$ as the limit of TASEP i.e.\ essentially exponential last passage percolation. They showed that the finite dimensional distributions of $h$ can be expressed in terms of a tractable Fredholm determinant formula involving Brownian hitting probabilities. Recently, \cite{nica2020one} showed that Brownian last passage percolation also converges to the KPZ fixed point, thus rigorously showing \eqref{E:hty}.
 \item The directed landscape contains more information than the KPZ fixed point, namely the joint distribution of the coupled evolution for all initial conditions. This allows for a full description of the law of limiting geodesics as in Theorem \ref{T:directed-geodesic-simple}. As will be shown in upcoming work, this also allows for a description of the scaling limit of TASEP second class particle trajectories.
 \item Since convergence to the directed landscape is uniform on compact sets and the directed landscape is continuous, this immediately implies that last passage values along any space-like curve converge uniformly to an Airy process. Previous approaches to such results include finding explicit determinantal formulas for space-like curves, see \cite{borodin2006stochastic} and \cite{borodin2008large}, and geometric analysis of slow decorrelations, see \cite{ferrari2008slow} and \cite{corwin2012universality}.

 \item The negative of the directed landscape can be thought of as a ``signed directed metric'' on $\mathbb R^2$; it satisfies the triangle inequality for points in the right time order. Signed directed metrics occur naturally in fields such as geometry, e.g. Perelman's $\scrL$-distance, see \cite{perelman2002entropy}.
  \end{enumerate}
\end{remark}

Letting $(x,s)_n=(s+2x/n^{1/3},-\lfloor sn\rfloor )$, the translation between limiting and pre-limiting locations, we have the following theorem.
\begin{theorem}[Full scaling limit of Brownian last passage percolation]
\label{T:lp-limit}
There exists a coupling of Brownian last passage percolation and the directed landscape $\scrL$ so that
$$
B_n[(x,s)_n\LP (y,t)_n ]\;\; = \;\;2(t-s)\sqrt{n} + 2(y - x)n^{1/6} + n^{-1/6} (\scrL+o_n)(x,s;y,t).
$$
Here each $B_n$ is a sequence of independent Brownian motions. Each $o_n$ is a random function asymptotically small in the sense that on every compact set $K\subset \R^4_{\uparrow}$ there exists $a>1$ with
$\E a^{\sup_K |o_n|^{3/4}}\to 1$.
\end{theorem}

Theorem \ref{T:lp-limit} is restated and proven in the body as Theorem \ref{T:landscape-limit}.
We have strong control over the modulus of continuity of the directed landscape. In this next proposition and throughout the paper, by a \textbf{random constant} we simply mean an almost surely finite random variable.
\begin{prop}
\label{P:mod-land-i}
Let $\scrR(x,t;y,t+s)=\scrL(x,t;y,t+s) + (x-y)^2/s$ denote the stationary version of the directed landscape. Let $K\subset \R^4_\uparrow$ be a compact set. Then
$$
|\scrR(u) - \scrR(v)| \le C \lf(\tau^{1/3}\log^{2/3}(\tau^{-1}+1) + \xi^{1/2}\log^{1/2}(\xi^{-1}+1) \rg).
$$
for all $\xi,\tau>0$ and points $u,v\in K$ with spatial and temporal coordinates of distance at most $\xi,\tau$ respectively. Here
$C$ is a random constant depending on $K$ with $\E a^{C^{3/2}} < \infty$  for some $a > 1$.
\end{prop}

See Proposition \ref{P:mod-land} for a version of Proposition \ref{P:mod-land-i} that keeps track of the dependence on the compact set.
The spatial fluctuations of $\scrL$ have been known to be locally Brownian since \cite{CH}. The temporal modulus of continuity has also been previously obtained in the context of Poisson last passage percolation, see \cite{hammond2018modulus}, building on related work from \cite{basu2014last} and \cite{basu2017coal}. Rather than using such results as a starting point for the proof of Proposition \ref{P:mod-land-i}, we deduce the proposition from explicit probability bounds on two-point differences (see Lemma \ref{L:levy-est} and Lemma \ref{L:mod-sk}).

\medskip
The directed landscape is a rich object containing all asymptotic information about last passage percolation in this scaling. In particular, as advertised above, we can take limits of last passage paths.
For a continuous path $h:[t, s] \to \R$, define the \textbf{length} of $h$ by
\begin{equation}
\label{E:geodesic-definition}
\int d \scrL \circ h = \inf_{k \in \N} \;\;\inf_{t = t_0 < t_1 < \dots < t_k = s} \;\;\sum_{i=1}^k \scrL(h(t_{i-1}), t_{i-1}; h(t_i), t_i).
\end{equation}
This is the analogue of defining the length of a curve in Euclidean space by piecewise linear approximation. We call $h$ a \textbf{directed geodesic} if equality holds for all subdivisions before taking any infima. We show that with probability one, directed geodesics exist between every pair of endpoint, see Lemma \ref{L:rig-exist}.
Moreover, the directed geodesic between any fixed pair of endpoints is almost surely unique and H\"older-$2/3^-$ continuous. Note that uniqueness may fail for some exceptional pairs. In particular, the directed geodesic is more regular than a Brownian path!
\begin{theorem}[Continuity of directed geodesics]
	\label{T:cty-geod-i}
Fix $u = (x, t; y, s) \in \R^4_\uparrow$. Then almost surely, there is a unique directed geodesic $\Pi_u$ from $(x, t)$ to $(y, s)$. Its distribution
only depends on $u$ through scaling: as random continuous functions from $[0, 1] \to \R$, we have
$$
\Pi_{(x, t; y, s)}(s + (t - s) r) \eqd \Pi_{(0,0; 0, 1)}(r) + x + (y - x) r.
$$
Moreover, for $u=(0,0;0,1)$ we have
$$
|\Pi_u(t+s) - \Pi_u(t)| \le Cs^{2/3}\log^{1/3} (2/s) $$
for all $s>0$ with $t,t+s\in [0,1]$. The random constant satisfies $\expt a^{C^3} < \infty$ for some $a > 1$.
\end{theorem}

Theorem \ref{T:cty-geod-i} is proven in Section \ref{S:directed-geodesics}.
For $u = (x, s; y, t) \in \R^4_\uparrow$ and $n \in \N$ let $\pi_{n,u}$ be a path from $(x,s)_n$ to $(y,t)_n$ that maximizes \eqref{E:lpintro}. We also show that the joint limit of last passage paths is given by the joint distribution of directed geodesics. Since each last passage path $\pi_{n,u}$ has domain $[s + 2xn^{-1/3}, t + 2yn^{-1/3}]$, we need to first compose $\pi_{n, u}$ with the affine shift $h_{n,u}$ that maps the interval $[s, t]$ to $[s + 2xn^{-1/3}, t + 2yn^{-1/3}]$ to talk about convergence to $\Pi_{u}$. Note that this shift is not necessary when $x = y = 0,$ as in Theorem \ref{T:directed-geodesic-simple}.

\begin{theorem}[Convergence of last passage paths]
	\label{T:intro-cvg-paths} In the coupling of Theorem \ref{T:lp-limit} there exists an event $A$ of probability 1 such that the following holds. For $u\in \R^4_\uparrow$, let $C_u$ be the set where the directed geodesic $\Pi_u$ is unique in $\scrL$. Then for any $u \in \R^4_\uparrow$, we have that
$$
\frac{\pi_{n,u}\circ h_{n, u}+nh_{n, u}}{2n^{2/3}} \to \Pi_u \quad \text{ uniformly, on the event $A \cap C_u$}.
$$
\end{theorem}

Theorem \ref{T:intro-cvg-paths} is restated and proven in Section \ref{S:joint-limits-lp-paths} as Theorem \ref{T:blppgeod}.

\medskip

Our approach to the proofs is probabilistic. It is based on understanding the geometry of last passage percolation using a continuous version of the Robinson-Schensted-Knuth (RSK) correspondence. Our reliance on formulas is minimal -- we only
use estimates about the Airy line ensemble from \cite{DV}, which rely solely on the determinantal nature of the Airy point process and the Brownian Gibbs property of the Airy line ensemble. On the other hand, our results imply convergence of formulas. For example, the super-exponential control of the error term in Theorem \ref{T:intro-sheet} guarantees uniform convergence of moment generating functions.

\medskip

The starting point for our proofs is a combinatorial fact about the continuous RSK correspondence. This correspondence and its application to Brownian paths was developed in \cite{o2002representation} and \cite{o2003path}. The continuous RSK correspondence maps an $n$-tuple of continuous functions $f:[0,1]\times \{1,\ldots, n\}\to \R, (x, i) \mapsto f_i(x)$ to an $n$-tuple $Wf$ in the same space, which we call the melon of $f$. The functions in the melon are ordered decreasingly: $Wf_1 \ge Wf_2 \ge \dots \ge Wf_n$. The main property used in the last passage literature is that the melon satisfies
\begin{equation}
\label{E:Wf-1}
(Wf)_1(t)=f[(0,n)\LP (t,1)].
\end{equation}
We use the stronger fact that for any $s<t$, remarkably
\begin{equation}
\label{E:Wfff}
Wf[(s,n)\LP(t,1)]=f[(s,n)\LP (t,1)].
\end{equation}
The identity \eqref{E:Wf-1} is equivalent to \eqref{E:Wfff} when $s = 0$ by the ordering of $Wf$. A generalization of \eqref{E:Wfff} to $k$ disjoint paths with arbitrary start and end points $(s_i, n), (t_i, 1)$ for $i = 1, \dots, k$ is proven as Proposition \ref{P:wm-equivalent}. This proposition is the crucial deterministic input in our construction of the Airy sheet, and we use it throughout the first part of the paper.

\medskip

Long after completing the first version of the paper, we learned that Proposition \ref{P:wm-equivalent} was proven in \cite{biane2005littelmann}, Lemma 4.8, for the case of disjoint paths with all equal start and end points. \cite{noumi2002tropical} also obtained close relatives of this result in a discrete setting, see Theorem 1.7.
A version of this formula in the special case when $s_i=0$ for all $i$ was previously obtained and studied in the planar positive temperature setting by \cite{o2016multi}, see the discussion directly below equation (20) and Theorem 3.4 from that paper.

\medskip

With $B$ restricted to the first $n$ lines, $WB$ has the law of Brownian motions conditioned not to intersect; these converge to the Airy line ensemble in the right scaling. So one might hope that the properly rescaled last passage problem in $WB$ converges to a last passage problem in the Airy line ensemble. This seems incredible at first because a large part of the last passage percolation is taking place in parts of the melon that disappear in the limit. Indeed, the limit only sees the top few lines and a time window of order $n^{-1/3}$.

\medskip

The technical part of the proof is to show that the last passage path from $(2xn^{-1/3},n)$ to $(1,1)$ in $WB$ follows a parabola of the form \begin{equation}
\label{E:kmapsto}
k\mapsto (1-2\sqrt{k/(2x)}n^{-1/3},k)
\end{equation}
These asymptotics follow from a multi-step path transformation lemma (Lemma \ref{L:rev-wm-fact}), a detailed analysis of a last passage problem on the Airy line ensemble (Section \ref{S:short-lp}), and an analysis of an optimization problem (Lemma \ref{L:zk-unif-bound}). Analyzing the last passage problem across the Airy line ensemble is the most technical part of the paper. It requires delicate structural results about the Airy line ensemble from \cite{DV}. In this paper, we take these structural results as black box inputs for the proof.

\medskip

For nearby starting points $x$ and $x'$ the parabolas in \eqref{E:kmapsto} diverge as $k\to\infty$, while the last passage values should be close. This suggests that there is not much information contained in these paths away from the top corner. We turn this intuition into a proof of Theorem \ref{T:intro-sheet} in Sections \ref{S:parabola} and \ref{S:airy-sheet}. In order to facilitate the proof, one key idea is too look at differences of last passage values. It is easier to directly show that these differences only depend on the top corner, and then extract the result for the last passage values themselves by averaging, see Remark \ref{R:differences}.

\medskip

The directed landscape can be patched together from Airy sheets. For the convergence of last passage percolation to the directed landscape, there is a technical tightness issue that we handle in Section \ref{S:cvg-land}.

\medskip

The focus of this paper is to construct the limiting objects, and we do not explore their properties in detail here. However, our description makes several natural questions about the directed landscape accessible. We will analyze the geometry of this object in future work.

\medskip

Our results complete the construction of the central object in the Kardar-Parisi-Zhang universality class. Although the full construction is new, several aspects of the directed landscape have been studied previously. We only mention a few results most directly related to the present work. For a gentle introduction suitable for a newcomer to the area, see \cite{romik2015surprising}. Review articles and books focusing on more recent developments include \cite{corwin2016kardar, ferrari2010random, quastel2011introduction, takeuchi2018appetizer} and \cite{weiss2017reflected}.

\medskip

The \cite{baik1999distribution} proof for the length of the longest increasing subsequence was the first to give the single point distribution of $\scrL$ as GUE Tracy-Widom in a slightly different model, see also \cite{johansson2000shape}. \cite{baryshnikov2001gues} and  \cite*{gravner2001limit} showed this convergence for Brownian last passage percolation by showing that $B[(0, n) \to (1, 1)]$ is equal in law to the top eigenvalue of the Gaussian Unitary Ensemble. This connection was 	extended to all eigenvalues at the level of a last passage process by \cite{o2002representation}.

\medskip

\cite{prahofer2002scale} proved convergence of last passage values to $\mathcal L(0, 0; y, 1)$ jointly for different values of $y$. This the top line of the Airy line ensemble.  \cite{corwin2013continuum} extended the analysis to continuum statistics of functions of $y$. \cite{CH} showed the Brownian Gibbs property of the Airy line ensemble, making it more amenable to probabilistic analysis. \citet*{corwin2015renormalization} predicted many of the results of the present paper.

\medskip

After predictions by \cite{dotsenko2013two}, \cite{johansson2017two, johansson2018two} gave the joint distribution of $\scrL(a,b)$ and $\scrL(a,c)$ for fixed $a,b,c$. \citet*{matetski2016kpz} derived a formula for the distribution of $g\scrS$, the metric composition of a fixed function $g$ and the Airy sheet  $\scrS$. \cite{baik2017multi} found formulas for the joint distribution of $\{\scrL'(a, b_i) : i \in \{1, \dots, k\}\}$ for any fixed $a, b_1, \dots, b_k$ for a related limiting object $\scrL'$ that in our language would be the directed landscape on the cylinder. We note the conjectured limit $\scrL'$ can be described by wrapping the directed landscape $\scrL$ around the cylinder and redefining path lengths locally. Recently, \cite{johansson2019multi} and \cite{liu2019multi} proved formulas analogous to those of Baik and Liu with $\scrL$ in place of $\scrL'$.

\medskip

Probabilistic and geometric methods have been used previously to prove qualitative statements about last passage percolation. As an early example, \cite{johansson2000transversal} studied transversal fluctuations. More recently, \cite{pimentel2017local} showed tightness of the Airy sheet in a different model, and proved that the Airy sheet locally looks like a sum of independent Brownian motions. \cite{ferrari2019time}
analyzed the covariance of last passage values at two different times. For other probabilistic and geometric approaches, see \citet*{basu2014last}, \citet*{basu2017coal}, and \cite{hammond2018modulus} discussed above.

\medskip

\cite{hammond2016brownian} used a probabilistic approach to prove Radon-Nikodym derivative and other regularity bounds for the Airy line ensemble with respect to Brownian bridges. Subsequent papers (see \cite{hammond2017exponents, hammond2017modulus, hammond2017patchwork}) combined this work with geometric reasoning to understand problems about the geometry of last passage paths in Brownian last passage percolation and the roughness of limiting growth profiles in that model.

\subsection{Brief outline of the text}

The first part of the paper is deterministic. Section \ref{S:prelim} and Section \ref{S:geometry} contain preliminaries and straightforward facts. Section \ref{S:melon} proves the key identity \eqref{E:Wfff} and its generalization, Proposition \ref{P:wm-equivalent}, and Section \ref{S:properties-melon} contains an important consequence of this proposition for last passage percolation in melons. 
The probabilistic part of the paper begins in Section \ref{S:short-lp}. In Sections \ref{S:short-lp}-\ref{S:airy-sheet} we construct the Airy sheet. The remaining sections build the directed landscape and directed geodesics from the Airy sheet (Sections \ref{S:properties-Airy}, \ref{S:directed-landscape}, and \ref{S:directed-geodesics}), prove convergence to the directed landscape (Section \ref{S:cvg-land}), and prove convergence of last passage paths (Section \ref{S:joint-limits-lp-paths}).

\section{Preliminaries}
\label{S:prelim}
\subsection{Last passage across general functions}
\label{SS:lpp}
For an interval $I \sset \Z$, let $C^I$ be the space of all continuous functions $$
f:\R \X I\to \R, \qquad (x,i)\mapsto f_i(x).
$$
We will often think of $f$ as a sequence of functions $\{f_i : i \in I\}$.
When $I = \{1, \dots, n\}$, we will simply write $C^n$. We call a nonincreasing function $\pi:[x,y]\to I$ which is cadlag on $[x,y]$ and satisfies $\pi(x)\ge \ell$ and $\pi(y)=m$ a \textbf{path} from  $(x,\ell)$ to $(y,m)$. Unfortunately the left endpoint $(x,\ell)$ is not specified by the function $\pi$ and has to be given separately. We will define the left limit of $\pi$ at $x$ to be $\ell$, and the right limit at $y$ to be $m$. 
Our paths are nonincreasing instead of nondecreasing to accommodate the natural indexing of the Airy line ensemble.

\medskip

We define the \textbf{length} of $\pi$ with respect to a coordinatewise differentiable function $f \in C^I$ by
$$
\int df \circ \pi := \int_x^y f'_{\pi(t)}(t) dt.
$$
For each $\pi$, this is just a sum of increments of $f$, so this definition extends to all continuous $f$. Note that for many of the cases we are interested in, the functions $f_i$ are ordered so that $f_i \ge f_{i+1}$. Hence when visualizing nondecreasing path length with respect to a set of such functions, it is natural to draw nondecreasing paths as rising physically, see Figure \ref{fig:LP-basic}.

\medskip

For $x\le  y \in \R$ and $m \le \ell \in I$ define the \textbf{last passage value} of $f$ from $(x, \ell)$ to $(y, m)$ by
$$
f[(x, \ell) \to (y, m)] = \sup_{\pi} \int df \circ \pi,
$$
where the supremum is taken over all paths $\pi$ from $(x, \ell)$ to $(y, m)$. 
See Figure \ref{fig:LP-basic} for an illustration of this definition. We say that a point $(x, t)$ \textbf{lies along} a path $\pi:[s, r] \to \Z$ if $t \in [s, r]$ and if
$$
\lim_{q \to t^-} \pi(q) \ge x \ge \lim_{q \to t^+} \pi(q).
$$
In other words, if the graph of $\pi$ is connected at its jumps by vertical lines, then $(x, t)$ will lie on this connected version of the graph.
\medskip

For $f \in C^I$, define the \textbf{gap process} $g = g(f)$ by $g_i = f_{i} - f_{i + 1}$. We can alternately define path length in terms of the gap process. For a nonincreasing path $\pi$ from $(x, \ell)$ to $(y, m)$, we have
\begin{equation}
\label{E:lp-gaps}
\int df \circ \pi = f_m(y) - f_\ell(x) - \sum_{i=m}^{\ell-1} g_i(t_i),
\end{equation}
where
\begin{equation}
\label{E:jump}
t_i\in[x,y]\quad  \text{ is the unique time so that  }\quad \pi(s)\begin{cases}\ge i+1&\text{ for }s < t_i,\\\le i&\text{ for }s \ge t_i.
\end{cases}
\end{equation}
We call  $t_{m+1},\ldots, t_{\ell}$  the \textbf{jump times} of $\pi$.
Thus the last passage value can be thought of as a difference of endpoints minus a minimal sum of gaps. This definition will be useful when we deal with nonintersecting sets of lines $f \in C^I$ whose gap processes are nonnegative. By \eqref{E:lp-gaps} we have 
\begin{equation}
\label{E:lp-gaps-cont}
f[(x,\ell)\to (y,m)] = f_m(y) - f_\ell(x) - \inf_{x\le t_{\ell-1}\le\dots\le t_{m}\le y}\sum_{i=m}^{\ell-1} g_i(t_i),
\end{equation}
which implies that the last passage value is continuous in $(x,y)$.

\medskip

We now extend the definition of last passage to disjoint collections of paths. Let
$$
U = \{(x_i, \ell_i) : i \in \{1, \dots k \} \}\qquad \mathand \qquad V = \{(y_i, m_i) : i \in \{1, \dots k \} \}
$$
 be two sequences of ordered pairs in $\R \X I$ with $x_i \le y_i$ and $m_i \le \ell_i$ for all $i$. The points $(x_i, \ell_i)$ and $(y_i, m_i)$ will be endpoints of disjoint paths $\pi_i$. Define $\scrQ(U, V)$ to be the {\bf set of disjoint paths} $\mathbf{\pi} = (\pi_1, \dots, \pi_k)$ from $U$ to $V$. More precisely, $\pi \in \scrQ(U, V)$ if
\smallskip
\begin{enumerate}[nosep, label=(\roman*)]
\item For all $i \in \{1, \dots, k\}$, the function $\pi_i$ is a path from $(x_i, \ell_i)$ to $(y_i, m_i)$.
\item For all $i \in \{1, \dots, k - 1\}$, we have that $\pi_i(t) < \pi_{i+1}(t)$ for all $t \in (x_i, y_i) \cap (x_{i+1}, y_{i+1})$. In order to ensure existence of disjoint paths with repeated endpoints, we do not enforce a disjointness condition at the endpoints.
\end{enumerate}

\smallskip
\noindent
For a path $\pi \in \scrQ(U, V)$, we define the \textbf{length} of $\pi$ with respect to $f$ by
$$
\int df \circ \pi = \sum_{i=1}^k \; \int df \circ \pi_i.
$$
With the above definition of $\scrQ(U, V)$, we say that $(U, V)$ is an \textbf{endpoint pair} if the following conditions hold. For this definition $U = \{(x_i, \ell_i)\}_{i \in \{1, \dots, k\}}$ and $V = \{(y_i, m_i)\}_{i \in \{1, \dots, k\}}$.

 \smallskip

 \begin{enumerate}[nosep, label=(\roman*)]
 \item For all $i \in \{1, \dots, k\}$, we have that $x_i \le y_i$ and $\ell_i \ge m_i$.
 \item For all $i \in \{1, \dots k -1\}$, we have that $x_i \le x_{i + 1}$ and $y_i \le y_{i + 1}$.
 \item The set of paths $\scrQ(U, V)$ is nonempty.
 \end{enumerate}

 \smallskip

 \noindent
 For an endpoint pair $(U, V)$ and a function $f \in C^I$, we define the \textbf{last passage value} of $f$ across $(U, V)$ by
$$
f[U \LP V] = \sup_{\pi \in \scrQ(U, V)} \; \int d f \circ \pi.
$$
In the case when $k = 1$, we recover the previous definition of the last passage path.

\medskip

We also define the set of \textbf{last passage paths} between $U = \{(x_i, \ell_i)\}_{i \in \{1, \dots, k\}}$ and $V = \{(y_i, m_i)\}_{i \in \{1, \dots, k\}}$ by
\begin{equation}\label{E:lp-paths}
P_f[U, V] = \lf\{ \pi \in \scrQ(U, V) : \int d f \circ \pi = \sup_{\sig \in \scrQ(U, V)} \; \int df \circ \sig \rg\}.
\end{equation}
We will omit the subscript $f$ above when the function $f$ is clear from context or does not change throughout a proof.

\subsection{Melons}
\label{SS:melons}
Let $C^n_+$ be the space of continuous functions
$$
f:[0, \infty) \X \{1, \dots, n\} \to \R, \qquad (x, i) \mapsto f_i(x).
$$
For $f \in C^n_+$, we can define a function $Wf \in C^n_+$ by the formula
$$
\sum_{i=1}^k (Wf)_i(t) = f[(0, n)^k  \LP (t, 1)^k], \qquad \qquad \mathforall k \in \{1, \dots, n\}, t \in [0, \infty).
$$
Here $(s, i)^k$ is the sequence with $k$ copies of the point $(s, i)$. The function $Wf$ can be thought of as the recording tableau of a continuous version of the Robinson-Schensted-Knuth bijection, see Section 6 in \cite{o2003conditioned}. We call $Wf$ the \textbf{melon} of $f$. We will explore the process of constructing melons more in Section \ref{S:melon}. Paths in the melon $Wf$ are ordered so that $(Wf)_1 \ge (Wf)_2 \ge \dots \ge (Wf)_n$ (see the discussion at the beginning of Section \ref{S:melon}). This is where the term melon comes from: since paths in $Wf$ avoid each other and all start from $0$, they look like stripes on a watermelon. Note that in physics literature, the term watermelon is often used for ensembles of nonintersecting random walks or Brownian motions which fit into this context.

\subsection{Brownian melons and the Airy line ensemble}
\label{SS:brownian}

We now introduce the main object of study in this paper, Brownian last passage percolation. See \citet*{weiss2017reflected} for background on the integrable aspects of this model. Let $B \in C^\Z$ be a sequence of independent two-sided Brownian motions. Let $B^n$ be $B$ restricted to $\R^+\times\{1,\ldots \, n\}$. We are concerned with finding the scaling limit of last passage values across the sequence $B$. By a result in Section \ref{S:melon}, we will be able to relate these last passage values to last passage values across the \textbf{Brownian $n$-melon} $WB^n$.

\medskip

There are many remarkable descriptions of the Brownian $n$-melon $WB^n$. The description that will be most useful to us here is that $WB^n$ can be described as the distributional limit as $\ep \to 0$ of a sequence of $n$ independent Brownian motions $B^i_\ep:[0, 1/\ep] \to \R$ with $B^i_\ep(0) = i \ep$  conditioned so that
$$
B^1_\ep (t) > B^2_\ep(t) > \dots > B^n_\ep(t).
$$
This was first proven as Theorem 7 in  \cite{o2002representation}, see also \citet*{biane2005littelmann}. The top lines of the Brownian $n$-melon have a scaling limit known as the Airy line ensemble.
This next theorem was proven in many parts, see \cite{prahofer2002scale, johansson2003discrete, adler2005pdes, CH}. We note that the final version of this theorem in \cite{CH} proves it for nonintersecting Brownian bridges, rather than nonintersecting Brownian motions. The two convergence statements are equivalent via the scaling relationship
$$
P(s) = \frac{t-s}{\sqrt{t}} B\lf(\frac{s}{t-s}\rg)
$$
relating a system of nonintersecting Brownian bridges $P$ on $[0,t]$ to a system of nonintersecting Brownian motions $B$ on $[0, \infty]$.

\begin{theorem}
\label{T:airy-line}
Let $WB^n$  be a Brownian $n$-melon. Define the rescaled melon
$
A^n = (A^n_1, \dots, A^n_n)
$
by
$$
A^n_i(y) = n^{1/6} \lf((WB^n)_i(1 + 2yn^{-1/3}) - 2\sqrt{n} - 2yn^{1/6} \rg).
$$
Then $A^n$ converges to a random sequence of functions $\scrA = (\scrA_1, \scrA_2, \dots) \in C^\N$ in law with respect to product of uniform-on-compact topology on $C^\N$. For every $y \in \R$ and $i < j$, we have that $\scrA_i(y) > \scrA_j(y)$. The function $\scrA$ is called the \textbf{(parabolic) Airy line ensemble.}
\end{theorem}

\begin{figure}
	\floatbox[{\capbeside\thisfloatsetup{capbesideposition={right,center},capbesidewidth=9cm}}]{figure}[\FBwidth]
	{\caption{A sketch of a `Brownian melon' from time $0$ to just after time $1$. If we zoom in around the location $(1, 2\sqrt{n})$ on a $O(n^{-1/3})\times O(n^{-1/6})$ parallelogram with slope $\sqrt{n}$, then we get the Airy line ensemble.}\label{fig:melon}}
	{\includegraphics[width=6cm]{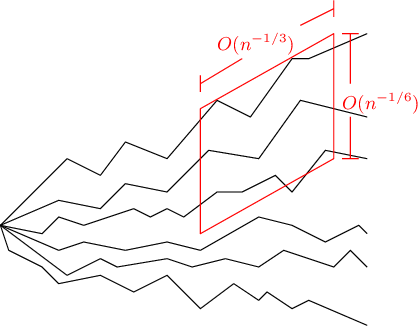}}
\end{figure}

The shifted line ensemble $\scrA(y) + y^2$ is stationary in time. We will refer to this object as the \textbf{stationary Airy line ensemble}. However, for our purposes, the parabolic Airy line ensemble is the object of interest. 
\medskip

The function $\mathcal A_1(y)+y^2$ is known as the \textbf{Airy process} (sometimes Airy$_2$). We now collect a few key facts about the Airy line ensemble and the Airy process. For this proposition and throughout the paper, we say that a Brownian motion (or bridge, or melon) has \textbf{variance $v$} if its quadratic variation in an interval $[s,t]$ is proportional to $v(t-s)$.

\begin{prop}[\cite{CH}, Proposition 4.1]

\label{P:brownian-airy}
Fix an interval $[a, a + b] \sset \R$ and $k \in \N$, and define $B_i(t) = \scrA_i(a + t) - \scrA_i(a)$ for $i \in \{1, \dots, k\}$. Then on the interval $[0, b]$ the sequence $(B_1, \dots, B_k)$ is absolutely continuous with respect to the law of $k$ independent Brownian motions with variance $2$.
\end{prop}




The one-point distributions of the Airy process follow a GUE Tracy-Widom distribution. This well-known result goes back to \cite*{baik1999distribution}. As this will be used throughout the paper, we state it here as a theorem. The tail bounds on GUE Tracy-Widom random variables that we use go back to \cite{tracy1994level}, see also \cite{ramirez2011beta} for short proofs.

\begin{theorem}[\cite*{tracy1994level, baik1999distribution}]
\label{T:TW-airy}
For every $t \in \R$, the random variable $\scrA_1(t) + t^2$ has GUE Tracy-Widom distribution. In particular, it satisfies the tail bounds
$$
\p(\scrA_1(t) + t^2 > m) \le ce^{-dm^{3/2}} \quad \mathand \quad \p(\scrA_1(t) + t^2 < -m) \le ce^{-dm^{3}}
$$
for universal constants $c$ and $d$.
\end{theorem}

We will also use the following bound on two-point distributions of the Airy process.

\begin{lemma} [\cite{DV}, Lemma 6.1]
\label{L:airy-tails}
There are constants $c, d > 0$ such that for every $t \in \R, s, a > 0,$ we have
$$
\p(|\scrA_1(t) + t^2 - \scrA_1(t + s) - (t+s)^2| > a\sqrt{s}) \le ce^{-da^2}.
$$
\end{lemma}

Finally, we will also need the finite versions of Theorem \ref{T:TW-airy} and Lemma \ref{L:airy-tails}, as well as a proposition bounding the entire Brownian $n$-melon below a particular function.

\begin{theorem}
\label{T:top-bd}
Let $W^n_1$ be the top line of a Brownian $n$-melon. There exist constants $c$ and $d$ such that for all $m > 0$ and all $n \ge 1$ we have
$$
\prob(|W^n_1(1) - 2 \sqrt{n} | \ge m n^{-1/6} ) \le c e^{-d m^{3/2}}.
$$
\end{theorem}

Theorem \ref{T:top-bd} is proven in \cite{ledoux2010small} for $m < n^{2/3}$. For greater values of $m$, the result is more classical and follows from the large deviation theory of the Gaussian unitary ensemble, see for example \cite{ledoux2007deviation}, equation (2.7).

\begin{prop}
\label{P:dyson-tails}
Fix $a > 0$. There exist constants $c, d > 0$ such that for every $n \in \N$, $t > 0, \;s \in [0, atn^{-1/3}]$, and $m > 0$ we have
$$
\p\bigg(\lf| W_1^n(t) - W_1^n(t +s) - s\sqrt{n/t}\rg| > m \sqrt{s} \bigg) \le c e^{-d m^{3/2}}.
$$
\end{prop}

Proposition \ref{P:dyson-tails} follows from \cite{hammond2016brownian}, Theorem 2.14, see also \cite{DV}, Proposition 4.1 for an alternate proof.

\begin{prop}[\cite{DV}, Proposition 4.3]
\label{P:cross-prob}
There exist positive constants $b, c,$ and $d$ such that for all $m > 0$ and $n \ge 1$, with probabity at least $1- c e^{-d m^{3/2}}$ we have
$$
W^n_1(t) \le 2 \sqrt{nt} + \sqrt{t} n^{-1/6}[m + b\log^{2/3}(n^{1/3} \log(t \;\vee \;t^{-1}) + 1)] \;\; \mathforall t \in [0, \infty).
$$
Here $a \vee b$ is the maximum of $a$ and $b$.
\end{prop}

Frequently in the paper, we use Theorem \ref{T:top-bd}, Proposition \ref{P:dyson-tails}, and Proposition \ref{P:cross-prob} to bound last passage values either between two fixed points, or to a single line. The connection is via the formula for $Wf$ in Section \ref{SS:melons}. Note that all three of these bounds give optimal or near-optimal results even in the limiting scaling.

\subsection{Modulus of Continuity}

In order to construct many of the objects in the paper, we will need a way of translating tail bounds on two-point differences into modulus of continuity bounds. For this, we use a generalized version of L\'evy's modulus of continuity for Brownian motion.

\begin{lemma}[\cite{DV}, Lemma 3.3]
	\label{L:levy-est}
	Let $T=I_1\times \dots \times I_d$ be a product of bounded real intervals of length $b_1, \dots, b_d$. Let $c, a>0$.
	Let $\scrH$ be a random continuous function from $T$ taking values in a vector space $V$ with norm $|\cdot |$. Assume that for every $i > 0$, there exists $\al_i \in (0,1), \beta_i, r_i > 0$ such that
	\begin{equation}
	\label{E:tail-bd}
	\p(|\scrH(t+u e_i) - \scrH(t)| \ge m u^{{\alpha_i}}) \le c e^{-a{m^{{\beta_i}}}}
	\end{equation}
	for every coordinate vector $e_i$, $m>0$, and points $t,t+u e_i\in T$ with $u < r_i$. Set $\beta = \min_i \beta_i, \al = \max_i \al_i$, and $r = \max_i r_i^{\al_i}$. Then with probability one we have
	\begin{equation}
	|\scrH(t + s) - \scrH(t)| \le C \lf(\sum_{i=1}^d |s_i|^{\al_i} \log^{1/\beta_i} \lf(\frac{2 r^{1/\al_i}}{|s_i|} \rg) \rg),
	\end{equation}
	for every $t,t+s\in T$ with $|s_i| \le r_i$ for all $i$ (here $s = (s_1, \dots, s_d)$).
	Here $C$ is random constant satisfying
	$$
	\p(C > m) \le \lf[\prod_{i=1}^d \frac{b_i}{r_i} \rg] c c_0 e^{-c_1 m^{\beta}},
	$$
	where $c_0$ and $c_1$ are constants that depend on $\al_1, \dots \al_d, \beta_1, \dots, \beta_d, d$ and $a$. Notably, they do not depend on $b_1, \dots, b_d, c$ or $r_1, \dots, r_d$.
\end{lemma}

\subsection{Notation}
\label{SS:notation}

We now introduce notation for last passage values across Brownian motions, Brownian melons, and the Airy line ensemble. This notation will be used throughout the paper starting in Section \ref{S:short-lp}.
Letting $B \in C^\Z$ be a sequence of independent two-sided Brownian motions, we first define
\begin{align*}
[x \LP y]_n &:= B[(x, n) \LP (y, 1)],
\end{align*}
We will often omit the subscript $n$ from the brackets and simply write $[x \LP y]$ when the value of $n$ is clear from context. We will also use the mixed notation
$$
[(x, k) \LP y]_n := B[(x, k) \LP (y, 1)]  \quad \mathand \quad [x \LP (y, k)]_n := B[(x, n) \LP (y, k)].
$$
For last passage values in the melon $WB^n$, we will use all the same notation with curly brackets $\{ \;\}$ in place of square ones $[\;]$. We will also use angled brackets $\langle \; \rangle$ for last passage values across the parabolic Airy line ensemble:
$$
\langle (x, k) \LP y \rangle := \scrA[(x, k) \LP (y, 1)].
$$
We will write $\pi\{x, y\}$ for the rightmost last passage path in $WB^n$ from $(x, n)$ to $(y, 1)$. Here a last passage path $\pi$ is `rightmost' if $\pi \ge \tau$ for any other last passage path $\tau$, see Lemma \ref{L:rightmost}. We similarly write $\pi[x, y]$ be the rightmost last passage path in $B^n$ from $(x, n)$ to $(y, 1)$.

\medskip

To avoid carrying around $2n^{-1/3}$ spatial terms, we will often use the notation
$$
\bar{x} = 2xn^{-1/3} \quad \mathand \quad \hat{y} = 1 + 2yn^{-1/3}
$$
when the value of $n$ is clear from context.

\section{The geometry of last passage paths}
\label{S:geometry}

Last passage paths can be thought of as geodesics in a metric space. This is a guiding principle for many of the proofs in the paper. With this intuition in mind, we devote this section to stating and proving some basic facts about the geometry of last passage paths. Many of these facts are well-known and well-used in the context of last passage percolation.

\medskip

For the rest of this section, let $f\in C^\Z$. The first lemma states that last passage paths have the geodesic property that they maximize length between any two points on the path. Its proof is straightforward and hence omitted.

\begin{lemma}[Geodesic property]\label{L:geod-btw}
Let $((x, \ell), (y, m))$ be an endpoint pair of single points. Then for any $\pi$ in the set of last passage paths $P_f[(x, \ell), (y, m)]$, and any times $s < t \in [x, y]$, we have that
$$
\int_s^t df \circ \pi = \max \lf\{ \int_s^t df \circ \tau : \;\; \tau \text{ is a path from } (s,\pi(s)) \text{ to } (t,\pi(t)) \rg\}.
$$
\end{lemma}

The next fact is a straightforward consequence of Lemma \ref{L:geod-btw}. Its proof is again omitted.

\begin{lemma}[Metric composition law]
\label{L:metric}
Let $((x, \ell), (y, m))$ be an endpoint pair of single points. Then for any $k \in \{m, \dots, \ell\}$, we have that
\begin{align*}
f[(x, \ell) \LP (y, m)] &= \sup_{z \in [x, y]} f[(x, \ell) \LP (z, k)]  + f[(z, k) \LP (y, m)]
\end{align*}
and for any $k \in \{m + 1, \dots, \ell\}$, we have
\begin{align*}
f[(x, \ell) \LP (y, m)]&= \sup_{z \in [x, y]} f[(x, \ell) \LP (z, k)]  + f[(z, k - 1) \LP (y, m)].
\end{align*}
\end{lemma}
Lemma \ref{L:metric} implies a \textbf{triangle inequality} for last passage values. For any $x \le z \le y$ and $m\le k \le \ell$ we have
\begin{equation}
\label{E:triangle-ineq}
f[(x, \ell) \LP (y, m)] \ge f[(x, \ell) \LP (z, k)]  + f[(z, k) \LP (y, m)].
\end{equation}
Note that in this equation, the inequality is reversed compared to the triangle inequality for metric spaces. It will also be useful to understand the right hand side above as a function of $z$.
\begin{lemma}
\label{L:double-mono}
Let $((x, \ell), (y, m))$ be an endpoint pair and fix $k \in \{m, \dots, \ell\}$. For $z \in [x, y]$, define
$$
h_1(z) = f[(x, \ell) \LP (z, k)]  - f_k(z) \quad \mathand \quad h_2(z) = f[(z, k) \LP (y, m)] + f_k(z).
$$
Then
$$
f[(x, \ell) \LP (y, m)] = \sup_{z \in [x, y]} h_1(z) + h_2(z),
$$
the function $h_1$ is nondecreasing and the function $h_2$ is nonincreasing.
\end{lemma}

\begin{proof}
The representation of $f[(x, \ell) \LP (y, m)]$ as a supremum over $h_1 + h_2$ follows immediately from Lemma \ref{L:metric}. By the sum of gaps representation \eqref{E:lp-gaps},
$$
h_1(z) = - f_\ell(x) - \inf \sum_{i=k+1}^{\ell} f_{i-1}(t_i) - f_{i}(t_i),
$$
where the infimum is over all sequences $t_{\ell} \le \dots \le t_{k+1} \in [x, z]$. As we increase $z$, this infimum can only get smaller, so $h_1$ is nondecreasing. The proof that $h_2$ is nonincreasing is similar.
\end{proof}



\begin{lemma}
\label{L:closed-set}
For any endpoint pair $((x,\ell),(y,m))$, the set $P[(x,\ell),(y,m)]$ is closed with respect to the topology of convergence of jump times \eqref{E:jump}.
\end{lemma}

Lemma \ref{L:closed-set} follows from the continuity of $f$. The next lemma shows that  we can pick out rightmost and leftmost paths in the set $P[(x, \ell), (y, m)]$.

\begin{lemma}
\label{L:rightmost}
Let $((x, \ell), (y, m))$ be an endpoint pair. There exist paths $\pi^-, \pi^+ \in P[(x, \ell), (y, m)]$ such that for any $\pi \in P[(x, \ell), (y, m)]$ and any $z \in [x, y]$ we have
$$
\pi^-(z) \le \pi(z) \le \pi^+(z).
$$
We refer to $\pi^-$ as the \textbf{leftmost last passage path} and $\pi^+$ as the \textbf{rightmost last passage path}.
\end{lemma}

\begin{proof}
The set of last passage paths is closed by Lemma \ref{L:closed-set}. Thus it suffices to show that for any paths $\pi_1, \pi_2 \in P[(x, \ell), (y, m)]$, that there exist paths $\tau_1, \tau_2 \in P[(x, \ell), (y, m)]$ such that
$$
\tau_1(z) \le \pi_i(z) \le \tau_2(z) \quad \text{ for } i = 1,2.
$$
Define paths   $\tau_1=\pi_1\wedge\pi_2$ and $\tau_2=\pi_1\vee\pi_2$ from $(x, \ell)$ to $(y, m)$. Then,
$$
\int d f \circ \tau_1 + \int d f \circ \tau_2 = \int d f \circ \pi_1 + \int d f \circ \pi_2,
$$
since the paths $\tau_1$ and $\tau_2$ cover the same parts of lines in $f$ as $\pi_1$ and $\pi_2$. Since the paths $\pi_1$ and $\pi_2$ maximize length, each of the paths $\tau_i$ must maximize length as well.
\end{proof}


\begin{lemma}
[Monotonicity and continuity of last passage paths]
\label{L:mono-path}For $x \le y$, let $\pi^+[x, y]$ denote the rightmost last passage path in $P[(x, n), (y, 1)]$. Then $\pi^+[x, y]$ is a nondecreasing, right continuous functions of both $x$ and $y$ in the topology of convergence of jump times. Similarly the leftmost path $\pi^-[x, y]$ is a nondecreasing, left continuous function of both $x$ and $y$.
\end{lemma}

\begin{proof} We just prove the statements for rightmost paths.
Let $x_1 \le x_2$, $y_1 \le y_2$, and $x_2 \le y_1$.
On the interval  $[x_2, y_1]$, define 
$$
\tau_1 = \pi^+[x_1, y_1]\wedge \pi^+[x_2, y_2],\quad \tau_2 = \pi^+[x_1, y_1]\vee \pi^+[x_2, y_2].
$$
We can extend $\tau_1$ to the interval $[x_1, x_2]$ by defining it to be equal to $\pi^+[x_1, y_1]$ there and similarly extend $\tau_2$ to $[y_1, y_2]$ by setting it to be equal to $\pi^+[x_2, y_2]$. For $i =1, 2$, $\tau_i$ is a path from $(x_i, n)$ to $(y_i, 1)$.
We have
$$
\int d f \circ \tau_1 + \int d f \circ \tau_2 = \int d f \circ \pi^+[x_1, y_1] + \int d f \circ \pi^+[x_2, y_2].
$$
Since $\pi^+[x_1, y_1]$ and $\pi^+[x_2, y_2]$ maximize length, we have that $\tau_i \in P[(x_i, n), (y_i, 1)]$ for $i =1, 2$. Moreover, $\tau_2 \ge \pi^+[x_2, y_2]$ by construction. Since $\pi^+[x_2, y_2]$ is a rightmost last passage path, this is in fact equality. Hence $\tau_1 = \pi^+[x_1, y_1]$ as well, showing monotonicity.

\medskip

Now we prove right continuity in $x$. The proof of right continuity in $y$ is similar. Fix $y \in \R$ and let $x_m \cvgdown x \in \R$. By monotonicity, the sequence of paths $\pi^+[x_m, y]$ has a limit $\pi$ in the topology of convergence of jump times. Since path length is a continuous function in this topology, we have
$$
\int df \circ \pi = \lim_{m \to \infty} \int df \circ \pi^+[x_m, y] = \lim_{m \to \infty} f[(x_m, n) \to (y, 1)] = f[(x, n) \to (y, 1)].
$$
The final equality follows from the continuity of last passage values in $x$, see \eqref{E:lp-gaps-cont}.
Therefore $\pi$ is a last passage path from $(x, n)$ to $(y, 1)$. Moreover, by monotonicity $\pi^+[x, y] \le \pi^+[x_m, y]$ on the interval $[x_m, y]$ for all $m$. Therefore $\pi^+[x, y] \le \pi$ as well. Since $\pi^+[x, y]$ is the rightmost last passage path, $\pi = \pi^+[x, y]$ as desired.
\end{proof}

We now show that paths exhibit a tree structure.

\begin{prop}
\label{P:tree-structure}
Let $x_1 \le x_2 < y_1 \le y_2$ be points in $\R$. Let $\pi^+[x_i, y_i]$ denote the rightmost last passage path in $P[(x_i, n), (y_i, 1)]$. Then there is a (possibly empty) interval $[a, b] \sset [x_2, y_1]$ such that the following holds.
\begin{enumerate}[label=(\roman*)]
\item $\pi^+[x_1, y_1](s) = \pi^+[x_2, y_2](s) \mathforall s \in (a, b).$
\item $\pi^+[x_1, y_1](s) < \pi^+[x_2, y_2](s) \mathforall s \in [x_2, y_1] \smin [a, b].$
\end{enumerate}
In particular, if $x_1 = x_2$, then the last passage paths to $y_1$ and $y_2$ follow the same path up to time $b$, and are entirely disjoint afterwards (so they form two branches in a tree).
The same tree structure holds for $\pi^-$ in place of $\pi^+$.
\end{prop}

The discrepancy at the endpoints $a,b$ is due to the fact that the paths are not continuous.

\begin{proof}
Let
$$
I = \{t \in [x_2, y_1]: \pi^+[x_1, y_1](t) = \pi^+[x_2, y_2](t) \}.
$$
For any two points $t_1 < t_2 \in I$, by the geodesic property of last passage paths, Lemma \ref{L:geod-btw}, the paths $\pi^+[x_i, y_i]|_{[t_1, t_2]}$ are both last passage paths between the points
$$
(t_1, \pi^+[x_1, y_1](t_1)) \quad \mathand \quad (t_2, \pi^+[x_1, y_1](t_2)).
$$
Moreover, both of these paths are rightmost last passage paths on this interval, so they must be equal. Hence $[t_1,t_2]\subset I$. Since $t_1,t_2$ are arbitrary points in $I$, it follows that $I$ is an interval, as desired.

Part (ii) of the lemma follows from monotonicity of last passage paths, Lemma \ref{L:mono-path}.
\end{proof}

We end this section with a monotonicity result for sums of last passage values.
\begin{prop}
\label{P:mono-inc} Let $x=(x_1,n), y=(y_1,1)$ and define $x'$,$y'$ similarly.
Assume  $x_1\le x_1' < y_1\le y_1' \in \R$. Then
$$
f[x\LP y]+f[x' \LP y']\ge f[x \LP y']+f[x' \LP y].
$$\end{prop}

\begin{proof}
By the ordering of the points, there must exist a point $z=(t, k)$ that lies along the rightmost last passage paths both  from $x$ to $y'$ and $x'$ to $y$. Because of the geodesic property
\begin{align*}
f[x \LP y']&= f[x\LP z]+f[z \LP y'] \quad \mathand \quad f[x' \LP y]= f[x'\LP z]+f[z \LP y].
\end{align*}
The result is then the sum of the triangle inequalities
\[
f[x \LP y]\ge f[x\LP z]+f[z \LP y] \quad \mathand 
\quad f[x' \LP y']\ge f[x'\LP z]+f[z \LP y'].  \qedhere
\]
\end{proof}

\section{Melons}
\label{S:melon}
Recall that $C^n_+$ is the space of $n$-tuples of continuous functions from $[0, \infty)$ to $\R$. Recall also the melon map $f \mapsto Wf$, introduced in Section \ref{SS:melons}, defined so that
\begin{equation}
\label{E:Wf-one}
\sum_{i=1}^k (Wf)_i(t) = f[(0, n)^k  \LP (t, 1)^k]
\end{equation}
for all $k \in \{1, \dots, n\}$ and $t \in [0, \infty)$.
In this section, we show that certain last passage values are preserved by the map $f \mapsto Wf$.

\medskip

To do this, we first approach the construction of $Wf$ by successively sorting pairs of functions. This approach is taken in \citet*{o2002representation, biane2005littelmann}. Let $f_1, f_2:[0, \infty) \to \R$ be two continuous functions. For $x< y \in [0, \infty)$, define the minimal gap size
$$
G(f_1, f_2)(x, y) = \min_{s \in [x, y]} [f_1(s) - f_2(s)].
$$
Then the last passage values satisfy
\begin{equation}
\label{E:W-Gdef}
\begin{split}
W(f_1,f_2)_1(t) &= f_1(t) - f_2(0) - G(f_1, f_2)(0, t), \\
W(f_1,f_2)_2(t) &= f_2(t) - f_1(0) + G(f_1, f_2)(0, t).
\end{split}
\end{equation}
To see the above formula for $W(f_1,f_2)_2$ given the formula for $W(f_1,f_2)_1$, simply note that 
$$
f[(0, n)^n  \LP (t, 1)^n] = \sum_{i=1}^n f_i(t) - f_i(0)
$$
for any $f \in C^n_+, t \ge 0$. The formula for $W(f_1,f_2)_2$ uses the case $n = 2$.
Now, for $f\in C^n_+$ and $i \in \{1, \dots, n -1 \}$,  define
$$
\sig_i(f) = (f_1, f_2, \dots, f_{i-1}, W(f_{i}, f_{i + 1})_1, W(f_{i}, f_{i + 1})_2, f_{i + 2} \dots, f_n).
$$
Now let $(i_1, \dots, i_{n \choose 2})$ be any sequence of numbers in $\{1, \dots, n-1\}$ such that $\tau_{i_1} \cdots \tau_{i_{n \choose 2}}$ is the reverse permutation $n (n-1) \cdots 1$, where $\tau_i=(i,i+1)$ is an adjacent transposition. Then we can alternately define the melon of $f$ by
\begin{equation}
\label{E:red-decomp}
\mathcal W f :=\sig_{i_1} \dots \sig_{i_{n \choose 2}} (f).
\end{equation}
By the discussion immediately preceding Proposition 2.8 in \citet*{biane2005littelmann}, the above function is independent of the choice of reduced decomposition of $n \cdots 1$. Moreover, by Corollary 2.9 there,
$$
\sig_i \scrW f =  \scrW f
$$
for any $\sig_i$.
This implies that $(\scrW f)_1 \ge (\scrW f)_2 \ge \dots \ge (\scrW f)_n$. Combining this with the fact that $(\scrW f)_i(0) = 0$ for all $i$, we get that for any $t > 0$, there exists a last passage path from $(0, n)^k$ to $(t, 1)^k$ that only uses the top $k$ paths. This implies that
\begin{equation}
\label{E:sWf-melon}
\sum_{i=1}^k (\scrW f)_i(t) = \scrW f[(0, n)^k \to (t, 1)^k].
\end{equation}
The fact that $Wf=\scrW f$ then follows from our Proposition \ref{P:wm-equivalent}.

\medskip
\citet*{biane2005littelmann} in Section 4.5 note that the transformation \eqref{E:red-decomp} yields a nonintersecting walk representation of the recording tableaux given by the RSK bijection.  From this the equivalent formula \eqref{E:Wf-one} follows from Greene's theorem, see \cite{sagan2013symmetric}. A proof of this connection with RSK can be found in \cite{o2003path}. 
After posting a previous version of this paper, we learned from Neil O'Connell that Proposition \ref{P:wm-equivalent} for $k$ identical starting points and $k$ identical endpoints follows from Lemma 4.8.\ of \citet*{biane2005littelmann}.

\begin{prop}
\label{P:wm-equivalent}
Let $n, k \in \nat$, and let $U = \{(x_i, n)\}_ {i \in \{1, \dots k\}}$ and $V = \{(y_i, 1)\}_{i \in \{1, \dots k\}}$ be an endpoint pair with $x_i \ge 0$ for all $i$. Then for any $f \in C^n_+$, we have that
$$
f[U \LP V] = \scrW f[U \LP V].
$$
\end{prop}

We will prove this proposition in three steps. We first deal with the case $n = 2$ and when $U$ and $V$ each have one element.

\begin{lemma}
\label{L:wm-lem}
Let $f = (f_1, f_2)$. For every $0 \le x \le y$, we have that
$$
f\big[(x, 2) \LP (y, 1) \big] = Wf\big[(x, 2) \LP (y, 1) \big].
$$
\end{lemma}

\begin{proof}
Last passage values are left unchanged by shifting functions up or down by a constant. Hence we may assume that $f_1(0) = f_2(0) = 0$. To avoid carrying bulky notation, we also set
$$
s(x, y) = \max_{s \in [x, y]} [f_2(s) - f_1(s)].
$$
We have that 
\[
f\big[(x, 2) \LP (y, 1) \big] = f_1(y) - f_2(x) + s(x, y)
\]
and by \eqref{E:W-Gdef}, that
\[Wf\big[(x, 2) \LP (y, 1) \big] = f_1(y) + s(0, y) - f_2(x) + s(0, x) + \sup_{t \in [x, y]} [f_2(t) - f_1(t) - 2s(0, t)].
\]
Therefore to prove the lemma, it is enough to show that
\begin{equation}
\label{E:sup-equal}
s(x, y) - s(0, y) - s(0, x) = \sup_{t \in [x, y]} [f_2(t) - f_1(t) - 2s(0, t)].
\end{equation}
To prove \eqref{E:sup-equal}, we divide into cases. First suppose that $s(0, x) = s(0, y)$. In this case, since $s(0, \cdot)$ is an nondecreasing function, we have that $s(0, t) = s(0, x) = s(0, y)$ for all $t \in [x, y]$. Therefore
\begin{align*}
\sup_{t \in [x, y]} [f_2(t) - f_1(t) - 2s(0, t)] &= \sup_{t \in [x, y]} [f_2(t) - f_1(t)] - s(0, y) - s(0, x) \\
&= s(x, y) - s(0, y)-s(0, x).
\end{align*}
For the case when $s(0, x) < s(0, y)$, observe that
\begin{equation}
\label{E:ineq-sup}
\sup_{t \in [x, y]} [f_2(t) - f_1(t) - 2s(0, t)] \le - \inf_{t \in [x, y]} s(0, t) = - s(0, x).
\end{equation}
Moreover, since $s(0, \cdot)$ is nonconstant on the interval $[x, y]$, continuity of the functions $f_1$ and $f_2$ implies that there exist times $t_1$ and $t_2$ in $[x, y]$ such that
$$
s(0, t_1) = f_2(t_1) - f_1(t_1) = s(0, x) \qquad \mathand \qquad f_2(t_2) - f_1(t_2) = s(0, y).
$$
The first equation above implies that the inequality in \eqref{E:ineq-sup} is in fact equality, and the second equation implies that $s(x, y) = s(0, y)$. Combining these facts proves \eqref{E:sup-equal}.
\end{proof}

Next, we extend the $n=2$ case to deal with an arbitrary number of paths.

\begin{lemma}
\label{L:n2-general}
Let $f = (f_1, f_2).$ For every endpoint pair of the form $U = \{(x_i, 2)\}_{i \in \{1, \dots k\}}$ and $V = \{(y_i, 1)\}_{i \in \{1, \dots k\}}$, we have that
$$
f\big[U \LP V \big] = Wf[U \LP V].
$$
\end{lemma}

\begin{proof}
Since there are only two lines, there can only be disjoint paths from $U$ to $V$ if $y_i \le x_{i + 2}$ for all $i \in \{1, \dots, k - 2\}$.
For the same reason, whenever $x_{i + 1} < y_i$, if $\pi$ is a $k$-tuple of disjoint paths from $U$ to $V$ then
\begin{equation}
\label{E:j-conds}
\pi_i(t) = 1 \mathand \pi_{i+1}(t) = 2 \qquad \mathforall t \in [x_{i+1}, y_i).
\end{equation}
Therefore we can write, recalling that $\scrQ(U,V)$ denotes the set of all $k$-tuples of disjoint paths from $U$ to $V$
\begin{align*}
&\sup_{\pi \in \scrQ(U, V)} \sum_{i=1}^k \int df \circ \pi_i = \\
&\sup_{\pi \in \scrQ(U, V)} \lf( \sum_{i=1}^k\int_{x_i \vee y_{i-1}}^{y_i \wedge x_{i+1}} df \circ \pi_i + \sum_{i=1}^{k -1} \indic(x_{i+1} < y_i)[f_1(y_i) + f_2(y_i) - f_1(x_{i+1}) - f_2(x_{i+1})] \rg).
\end{align*}
In the above formula, we treat $y_0$ as $0$ and $x_{k+1}$ as $\infty$.
Now, since $(Wf)_1 + (Wf)_2 = f_1 + f_2$, the second term under the right hand supremum above is preserved by mapping $f \mapsto Wf$. We need to check that the same is true of the first term. Since the intervals $(x_i \vee y_{i-1},  y_i \wedge x_{i+1})$ are all disjoint from each other, the only condition that forces interactions between the coordinates of a path $\pi \in \scrQ(U, V)$ is condition \eqref{E:j-conds}. Therefore
$$
\sup_{\pi \in \scrQ(U, V)}  \sum_{i=1}^k\int_{x_i \vee y_{i-1}}^{y_i \wedge x_{i+1}} df \circ \pi_i = \sum_{i=1}^k \;\; \sup_{\pi_i \in \scrQ((x_i \vee y_{i-1}, 2), (y_i \wedge x_{i+1}, 1))} \;\; \int df \circ \pi_i.
$$
By Lemma \ref{L:wm-lem}, each supremum term in the sum on the right hand side is preserved by the map $f \mapsto Wf$.
\end{proof}

Before proving Proposition \ref{P:wm-equivalent}, we record an extension of the metric composition law in Lemma \ref{L:metric}.  First, for a set of $k$-tuples $Z = \{(w_i, m)\}_{i \in \{1, \dots, k\}}$, define
$Z^- = \{(w_i, m - 1)\}_{i \in \{1, \dots, k\}}$.

\begin{lemma}
\label{L:split-path}
Let $f \in C^\Z$. Let $(U, V)$ be an endpoint pair of the form $U = \{(x_i, n)\}_{i \in \{1, \dots, k\}}$, $V = \{(y_i, 1)\}_{i \in \{1, \dots, k\}}$, and let $m \in \{1, \dots, n - 1\}$. Then
$$
f[U \LP V] = \sup_Z f[U \LP Z] +  f[Z^- \LP V],
$$
where the supremum is taken over all $k$-tuples $Z = \{(z_i, m)\}_{i \in \{1, \dots, k\}},$ such that both $(U, Z)$ and $(Z^-, V)$ are endpoint pairs.
\end{lemma}

The proof of Lemma \ref{L:split-path} is straightforward and hence we omit it.
\begin{proof}[Proof of Proposition \ref{P:wm-equivalent}.]
It is enough to show that for any sequence of continuous paths $f = (f_1, \dots, f_n)$ with $f_i(0) = 0$ for all $i$ and any $m \in \{1, \dots, n - 1\}$,
$$
\sig_mf[U \LP V] = f[U \LP V].
$$
For any $f$, by Lemma \ref{L:split-path} applied twice, we can write
\begin{equation}
\label{E:split-at-m}
f[U \LP V] =  \sup_{T, Z} \Big(f[U \LP T] + f[T^- \LP Z] +  f[Z^- \LP V] \Big).
\end{equation}
Here the supremum is over all $T = \{(t_i, m+2)\}_{i \in \{1, \dots, k\}},$ $Z = \{(z_i, m)\}_{i \in \{1, \dots, k\}}$ such that  $(U, T),$ $(T^-, Z)$, and $(Z^-, V)$ are all endpoint pairs. For a fixed $T, Z$, when we apply $\sig_m$ to $f$, the first and third terms under the supremum in \eqref{E:split-at-m} do not change since the relevant components $f_i$ do not change. The middle term is preserved under the transformation $\sig_m$ by Lemma \ref{L:n2-general}. Hence the right hand side of \eqref{E:split-at-m} is also preserved under $\sig_m$.
\end{proof}

We finish this section with a straightforward consequence of Proposition \ref{P:wm-equivalent}.
As in  Section \ref{SS:lpp} define two paths $\pi_1$ and $\pi_2$ to be disjoint if either $\pi_1 > \pi_2$ or $\pi_1 < \pi_2$ on the entire interior of both functions' domains. As before, we define $\pi_f^+[x, y]$ as the rightmost last passage path in the set $P_f[(x, n), (y, 1)]$. We define the leftmost path $\pi_f^-[x, y]$ similarly. 

\begin{lemma}
\label{L:zk-tilde}
Let $U = ((x_1, n), (x_2, n))$ and $V = ((y_1, 1), (y_2, 1))$ be an endpoint pair with $0 \le x_1$. The paths $\pi_f^-[x_1, y_1]$ and $\pi_f^+[x_2, y_2]$ are disjoint if and only if the paths $\pi_{Wf}^-[x_1, y_1]$ and $\pi_{Wf}^+[x_2, y_2]$ are disjoint.
\end{lemma}

\begin{proof}
We first show that the paths $\pi_f^-[x_1, y_1]$ and $\pi_f^+[x_2, y_2]$ are disjoint precisely when
\begin{equation}
\label{E:f-dj-cond}
f[(x_1, n) \to (y_1, 1)] + f[(x_2, n) \to (y_2, 1)] = f[U \to V].
\end{equation}
To see this, note that 
if the paths $\pi_f^-[x_1, y_1]$ and $\pi_f^+[x_2, y_2]$ are disjoint, then $(\pi_f^-[x_1, y_1], \pi_f^+[x_2, y_2])$ forms a last passage path from $U$ to $V$. Moreover, if equality holds in \eqref{E:f-dj-cond}, then there must exist $\pi \in P_f[(x_1, n), (y_1, 1)]$ and $\tau \in P_f[(x_2, n), (y_2, n)]$ that are disjoint. Pushing these paths further left and right, respectively, preserves disjointness.

\medskip

Now, Proposition \ref{P:wm-equivalent} implies that \eqref{E:f-dj-cond} holds if and only if it holds with $Wf$ in place of $f$. By the same reasoning as above, this is true if and only if the paths $\pi_{Wf}^-[x_1, y_1]$ and $\pi_{Wf}^+[x_2, y_2]$ are disjoint.
\end{proof}

\section{Properties of melon paths}
\label{S:properties-melon}

In this section, we collect some key deterministic facts about last passage paths in melons.
The first is about the location of last passage paths.  We state the following lemma for functions $f \in C^n_+$ that start at $0$ and stay ordered, as any melon $Wf$ has this property.

\begin{lemma}
\label{L:high-paths}
Let $f \in C^n_+$ be such that $f_i(0) = 0$ and $f_i \ge f_{i+1}$ for all $i$. Fix $j  \le k \le n \in \N$. Let $U = \{(x_i, n)\}_{i \in \{1, \dots, k\}}, V = \{(y_i, 1)\}_{i \in \{1, \dots, k\}}$ be an endpoint pair with $x_i = 0$ for all $i \in \{1, \dots, j\}$. Then there exists a last passage path $\pi$
from $U$ to $V$, see \eqref{E:lp-paths},
such that $\pi_i(t) = i$ for all $t \in [0, y_1), i \in \{1, \dots, j\}$.
\end{lemma}

For the proof, it may help the reader to recall the ordering constraints required on $U, V$ and any set of disjoint paths from $U$ to $V$, see Section \ref{SS:lpp}.

\begin{proof}
Let $g$ be the gap process of $f$ defined by $g_i = f_{i} - f_{i+1}$. By the identity \eqref{E:lp-gaps}, we can write
\begin{equation}
\label{L:sum-of-gaps}
\begin{split}
f[U \LP V] &= \sup_{\tau} \sum_{i=1}^k \lf(f_1(y_i) - f_n(x_i) - \sum_{r=1}^{n-1} g_r(t_{i,r})\rg) \\
&=\sum_{i=1}^k \lf(f_1(y_i) - f_n(x_i)  - \inf_{\tau} \sum_{r=1}^{n - 1} g_r(t_{i,r}) \rg).
\end{split}
\end{equation}
Here the supremum is over all disjoint $k$-tuples of paths from $U$ to $V$, and $t_{i,r}$ is the jump time of the path $\tau_i$ to line $r$, see \eqref{E:jump}. Hence, any last passage path $\pi$ across $f$ minimizes the sum of the gaps. In particular, this implies that if $\pi$ is a last passage path, then for $i \in [1, j]$ when $x_i = 0$, we can replace each function $\pi_i$ with a function that immediately jumps up to line $i$ to take advantage of the zero-sized gaps there: since $g_i(t) \ge 0$ for all $i, t$, this process cannot decrease the length of the path. By the ordering constraints on the functions $\{\pi_i\}_{i \in \{1, \dots, k\}}$, we must also have have that $\pi_i(t) \ge i$ for all $t < y_1$.
\end{proof}

\subsection*{Melons opened up at other times}
We will introduce melons starting at points other than $0$. Let $f \in C^{n}$.
We define the \textbf{melon at $z$} by
$$
W_zf = W[f(z + \cdot)].
$$
We also define the reversing maps $R_z: C^n \to C^n$ by
$$
R_zf_i(t) = -f_{n+1-i}(z - t).
$$
We can now define the \textbf{reverse melon opened up at $z$} by
\begin{equation}
\label{E:Wrzf}
W^*_z f = W R_z f
\end{equation}
We record the following fact about reverse melon last passage values. 

\begin{lemma}
\label{L:rev-wm-fact} Let $f \in C^n$ and let $U = \{(x_i, n)\}_{i \in \{1, \dots, k\}}, V = \{(y_i, 1)\}_{i \in \{1, \dots, k\}}$ be an endpoint pair. For any $z \ge y_k$, we have that
$$
f[U \LP V] = R_z f[V_z \LP U_z] = W^*_z f[V_z \LP U_z].
$$
Here $U_z = \{(z - x_{k + 1 -i}, 1)\}_{i \in \{1, \dots, k\}}$ and $V_z = \{(z - y_{k + 1 - i}, n)\}_{i \in \{1, \dots, k\}}$.
\end{lemma}

\begin{proof}
The first equality follows from the symmetry of the definition of last passage paths under reversal. The second equality follows from Proposition \ref{P:wm-equivalent}.
\end{proof}

In the remainder of this section, we will build a path transformation lemma which represents $Wf[(x, n) \LP (y, k)]$ in terms of a first passage value across a reverse melon. For $f \in C^n$ and $k < n$ define the \textbf{ backwards first passage value}
\begin{equation}
\label{E:mathfrakf}
f[(x, 1) \LP_\mathfrak{f} (y, k)] = \inf_{\pi} \int df \circ \pi,
\end{equation}
where the infimum is taken over all nondecreasing cadlag functions $\pi:[x, y] \to \{1, \dots, k\}$.
\begin{lemma}
\label{L:k-line-prior}
Let $f \in C^n$.  For every $0 \le x < z $, we have that
\begin{align*}
Wf[(x, n) \LP (z, k)] = (Wf)_k(z) - W^*_zf[(z - x, 1) \to_{\mathfrak{f}} (z, k)].
\end{align*}
\end{lemma}

The proof of this lemma is illustrated in Figure \ref{fig:melontrans}.

\begin{figure}%
	\centering
	\begin{subfigure}[t]{4cm}
		\includegraphics[width=4cm]{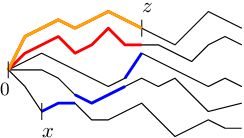}
		\caption{}
	\end{subfigure}
	\qquad
	\begin{subfigure}[t]{5cm}
		\includegraphics[width=5cm]{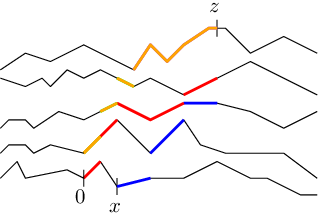}
		\caption{}
	\end{subfigure}
	\qquad
\begin{subfigure}[t]{4cm}
	\includegraphics[width=4cm]{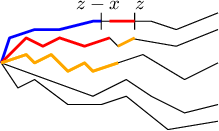}
	\caption{}
\end{subfigure}
	\caption{An illustration of the proof of Lemma \ref{L:k-line-prior}. We want to understand the last passage value in the melon from $(x, n)$ to $(z, k)$; here $n=5$ and $k =3$. To do this, note that this is the bottom path in a collection of $k-1$ disjoint last passage paths from $(0, n)$ to $(z, 1)$ and one path from $(x, n)$ to $(z, 1)$: this is Figure \ref{fig:melontrans} (a). We can then `demelonize' (Figure \ref{fig:melontrans} (b)) and then `remelonize' at $z$ after reversing the original lines (Figure \ref{fig:melontrans} (c)). These transformations leave the appropriate last passage values (i.e the sum of the lengths of the coloured paths) unchanged. In Figure \ref{fig:melontrans}  (c), the last passage paths use the top 3 lines up to time $z$, except for the graph of one increasing path from $(z-x, 1)$ to $(z, 3)$: this path gives the first passage value that appears in the statement of the lemma. This is only a sketch, i.e. Figures \ref{fig:melontrans} (a) and (c) are not truly the melon and reverse melon of Figure \ref{fig:melontrans} (b).} %
	\label{fig:melontrans}%
\end{figure}

\begin{proof}
To simplify notation in the proof, we write
$
\{z^{k-1}, w\}
$
for the $k$-tuple of points consisting of $k-1$ copies of $z$ and one copy of $w$. Similarly, we write $\{w, z^{k-1}\}$ if the points are in the opposite order. First by Lemma \ref{L:high-paths}, there is a last passage path $\pi$ in $Wf$ from $\{(0, n)^{k-1}, (x, n)\}$ to  $(z, 1)^{k}$
with $\pi_i (t) = i$ for all $t \in [0, z]$ and $i \le k - 1$. The only constraint on the remaining path $\pi_k$ is that it avoids the first $k-1$ lines on $[0, z)$. In other words, $\pi_k$ must be a last passage path from $(x, n)$ to $(z, k)$ in $Wf$. More formally, $\pi_k$ becomes a last passage path from $(x,n)$ to  $(z,k)$ if we redefine a single value $\pi_k(z):=k$.
Moreover,
$$
\sum_{i=1}^{k-1} \int d (Wf) \circ \pi_i = Wf [(0, n)^{k-1} \to (z, 1)^{k-1}],
$$
again because of Lemma \ref{L:high-paths}. Putting these facts together gives that
\begin{equation}
\label{E:split-1}
\begin{split}
Wf &[(x, n) \to (z, k)]\\
&= Wf [\{(0, n)^{k-1}, (x, n)\} \to (z, 1)^{k}] - Wf [(0, n)^{k-1} \to (z, 1)^{k-1}].
\end{split}
\end{equation}
We now analyze the first term on the right hand side above.
By Proposition \ref{P:wm-equivalent}, this term does not change when we replace $Wf$ by the original function $f$. Then we  use Lemma \ref{L:rev-wm-fact} to rewrite it in terms of the last passage values in the reverse melon $W^*_z f$ opened up at $z$. This gives that
\begin{equation}
\label{E:revv}
\begin{split}
Wf [\{(0, n)^{k-1}, (x, n)\} \to (z, 1)^{k}] = W^*_z f [(0, n)^k \to \{(z-x, 1), (z, 1)^{k-1} \}].
\end{split}
\end{equation}
By Lemma \ref{L:high-paths}, there exists a last passage path $\pi$  from $(0, n)^k$ to $\{(z-x, 1), (z, 1)^{k-1} \}$ 
in $W^*_z f$
such that $\pi_i(t) \le i$ for all $t > 0$ and $i \in \{1, \dots, k\}$. Hence $\pi|_{[0, z -x]}$ will be a last passage path from $(0, n)^k$ to $(z-x, 1)^k$ in $W^*_z f$. More formally, $\pi|_{[0, z -x]}$ will be a last passage path between these points if the endpoint values are redefined to be $1$ at $z-x$.
In particular, this implies that
\begin{align}
\nonumber
W^*_z f [(0, n)^k &\to \{(z-x, 1), (z, 1)^{k-1} \}]\\
\label{E:almost-fp} &= W^*_z f [(0, n)^k \to (z-x, 1)^k] + \sup_{\tau} \sum_{i=2}^k \int_{z-x}^z d W_z^* f \circ \tau_i,
\end{align}
where the supremum above is over all sequences of nonincreasing cadlag paths $\tau = (\tau_2, \dots, \tau_k)$ 
with $\tau_i(t) \le i$ and $\tau_i < \tau_{i+1}$ on the interval $(z-x, z)$. This restriction on the paths $\tau_i$ implies that for all $i$, $\tau_i(t) \in \{i-1, i\}$ for $t \in [z - x, z)$. Moreover, letting 
$$
t_i = z \wedge \inf \{ s \in [z- x, z] : \tau_i(s) = i-1 \}
$$ 
we have $z -x \le t_2 \dots \le t_k \le z$, and the complement of the union of the graphs of $\tau_2, \dots, \tau_k$ in $[z -x, z] \X \{1, \dots, k\}$ is given by
$$
[z-x, t_2) \X \{1\} \cup [t_2, t_3) \X \{2\} \cup \dots \cup [t_k, z] \X \{k\}.
$$
This is the graph of a nondecreasing cadlag path $\tau_*:[z-x,z]\to \{1, \dots, k\}$. 
Together, the paths $\tau_2, \dots, \tau_k, \tau_*$ cover all lines $\{1,\ldots, k\}$ on the interval $(z-x,z)$, so we have
\begin{align*}
\sum_{i=2}^k \int_{z-x}^z d W_z^* f \circ \tau_i + \int_{z-x}^z d W_z^* f \circ \tau_* &= \sum_{i=1}^k W^*_z f(z) - W^*_z f(z-x)   \\
&=W^*_z f [(0, n)^k \to (z, 1)^k]  - W^*_z f [(0, n)^k \to (z-x, 1)^k].
\end{align*}
Here the final equality follows from \eqref{E:sWf-melon}.
Hence the right hand side of \eqref{E:almost-fp} is equal to
$$
W^*_z f [(0, n)^k \to (z, 1)^k] - \inf_{\tau^*} \int_{z-x}^z d W_z^* f \circ \tau_*,
$$
where the infimum is taken over all nondecreasing paths $\tau_*:[z-x, z] \to \{1, \dots, k\}$. This infimum is simply the backwards first passage value $W^*_zf[(z - x, 1) \to_\mathfrak{f} (z, k)]$. Finally, combining this representation with \eqref{E:split-1} and \eqref{E:revv} implies that
\begin{equation}
\label{E:W-many}
\begin{split}
Wf [(x, n) \to (z, k)]  &= W^*_z f [(0, n)^k \to (z, 1)^k] \\
&- W f [(0, n)^{k-1} \to (z, 1)^{k-1}] - W^*_zf[(z - x, 1) \to_\mathfrak{f} (z, k)].
\end{split}
\end{equation}
We can rewrite the first term on the right hand side of \eqref{E:W-many} in terms of a last passage time across $f$ by using Lemma \ref{L:rev-wm-fact}. This gives that
$$
W^*_z f [(0, n)^k \to (z, 1)^k] = f [(0, n)^k \to (z, 1)^k] = W f [(0, n)^k \to (z, 1)^k],
$$
where the second equality uses Proposition \ref{P:wm-equivalent}.
Furthermore, by Lemma \ref{L:high-paths}, the above last passage is equal to $(Wf)_1(z) + \dots + (Wf)_k(z)$. The second term on the right hand side of \eqref{E:W-many} can similarly be written as $(Wf)_1(z) + \dots + (Wf)_{k-1}(z)$, completing the proof of the lemma.
\end{proof}

\section{The Airy line ensemble last passage problem}
\label{S:short-lp}

In this section, we study a last passage problem in the Airy line ensemble that arises naturally when we try to understand the limit of the right side of the equality in Proposition \ref{P:wm-equivalent}. To motivate the study of this problem, we first see how it arises in the study of Brownian last passage percolation via the melon identity in Lemma \ref{L:k-line-prior}. 

\medskip

Recall from Section \ref{SS:notation} the scaling operations $\bar{x} = 2xn^{-1/3}$ and $\hat{y} = 1 + 2yn^{-1/3}$, and the bracket notation $[]$, $\{\}$, and $\langle \rangle$ for last passage across Brownian motions, a Brownian melon, and the Airy line ensemble. By Proposition \ref{P:wm-equivalent}, the Brownian last passage value from $\bar{x}$ to $\hat{y}$ equals a last passage value across Brownian melon. By Lemma \ref{L:metric}, this analysis can be broken down into an analysis of a last passage problem across the top right corner of the melon and a last passage problem up to that corner. By continuity, the last passage problem across the top right corner will translate to a last passage problem in the Airy line ensemble.

\medskip

The harder part is the last passage up to the top right corner; these are values of the form $\{\bar x \to (\hat z,k)\}_n$. To tackle this problem, we approximate a variant,
$$
\{\bar x \to (\hat{z},k)\}_n-W_k^n(\hat z),
$$
which has the added advantage of monotonicity in $z$ (see Lemma \ref{L:double-mono}). Here and throughout we use the notation $W^n_k :=(WB^n)_k$ for the $k$th line in a Brownian $n$-melon.
\newcommand{\oo}{{\mathfrak o}}
In the sequel, for a random array $\{R_{n, k} :n, k \in \N\}$ we will write
\begin{equation}\label{E:onotation}
R_{n,k} = \oo(r_{k})  \qquad \text{ if for all }\epsilon>0 \qquad  \;\; \sum_{k=1}^\infty \limsup_{n\to\infty} \prob(|R_{n,k}/r_{k}|>\epsilon) < \infty.
\end{equation}
The idea behind this notation is that if we can pass to a limit in $n$ to get a sequence $R_k$, then by the Borel-Cantelli lemma, $R_k/r_k \to 0$ almost surely.

\begin{prop}\label{P:longbound} For each $n$, let $W^n = WB^n$ be a Brownian $n$-melon. Let $x > 0$ and let $z_k$ be an arbitrary sequence of real numbers. Let
$$
F^n_k(z)=n^{1/6}\left[\{\bar x \to (\hat{z},k)\}_n -W_k^n(\hat z) +2xn^{1/6}\right].
$$
Then
$$
F^n_k(z_k)= 2\sqrt{2kx}+2z_kx + \oo(\sqrt{k}).
$$
\end{prop}

The analysis in Proposition \ref{P:longbound} will boil down to understanding a last passage problem across the Airy line ensemble. As mentioned at the beginning of the section, this Airy line ensemble last passage problem arises out of an application of Lemma \ref{L:k-line-prior}. This application is dealt with by the following lemma. 
\begin{lemma}\label{L:reverseW} Fix $k \in \N$. Then in the setup of Proposition \ref{P:longbound}, we have
\begin{equation}\label{E:Lknlim}
  F^n_k(z_k) - 2z_kx \cvgd \langle (0, k) \LP (x, 1) \rangle \qquad \mathas n \to \infty.
\end{equation}
\end{lemma}

\begin{proof}
  By Lemma \ref{L:k-line-prior}, we have
  \begin{equation}
  \label{E:FnzkW}
  F^n_k(z_k) = 2xn^{1/3} - n^{1/6}W^{*}_{\hat z_k}B^n[(\hat z_k - \bar x, 1) \to_{\mathfrak{f}} (\hat z_k, k)]
  \end{equation}
where the notation $\mathfrak{f}$ refers to the backwards first passage value, see \eqref{E:mathfrakf}, and $W^{*}_{\hat z_k}$ is the reverse melon opened up at $\hat z_k$, see \eqref{E:Wrzf}. By time-reversal symmetry of Brownian motion, $W^*_{\hat z_k} B_n$ is equal to $WB_n$ in distribution. Therefore by Theorem \ref{T:airy-line}, the top corner of $W^*_{\hat z_k} B_n$ converges in distribution after proper rescaling to the Airy line ensemble.

\medskip

Now, for $f, \hat g = (g, \dots, g) \in C^n,$ and $\al \in \R$, backwards first passage values across $f$ and $\al f + \hat g$ are simply related by scaling by $\al$ and translation by an increment of $g$. Using this, we can rewrite the right side of \eqref{E:FnzkW} as a backwards first passage value across a sequence of functions that converge uniformly on compact sets to the Airy line ensemble. Since backwards first passage values are continuous with respect to uniform convergence on compact sets, this implies that the right hand side of \eqref{E:FnzkW} converges to
\begin{equation}
\label{E:Azkx}
-\mathcal A[ (z_k-x,1) \LP_\mathfrak{f} (z_k,k)].
\end{equation}
By flip symmetry of Airy line ensemble,
\eqref{E:Azkx} is equal in distribution to
\begin{equation}
\label{E:Azkx2}
\mathcal A[ (-z_k,k) \to (x - z_k,1)].
\end{equation}
This is equal in distribution to the right hand side of \eqref{E:Lknlim} plus $2z_kx$ since $t \mapsto \scrA(t) + t^2$ is stationary.
\end{proof}

We can now state the key theorem about last passage values across the Airy line ensemble. Together with Lemma \ref{L:reverseW} this implies Proposition \ref{P:longbound}.

\begin{theorem}
\label{T:airy-lp} Fix $x>0$, and recall that $\langle (0, k) \LP x \rangle$ is the last passage value across the Airy line ensemble $\scrA$ from line $k$ at time $0$ to line $1$ at time $x$. Then there exists a constant $d \in \N$ such that for every $\ep > 0$, we have
$$
\sum_{k \in \N} \p \lf(\frac{\langle (0, k) \LP x \rangle-{2\sqrt{2kx}}}{k^{3/7}\log^dk} > \ep\rg) < \infty.
$$
\end{theorem}

The intuition behind Theorem \ref{T:airy-lp} is as follows. Since the Airy line ensemble arises as the scaling limit at the edge of Dyson's Brownian motion, we can loosely interpret it as an infinite sequence of Brownian motions conditioned never to intersect. This intuition is made rigorous by the Brownian Gibbs property for the Airy line ensemble. This property states that conditioned on the values of $\scrA$ on the boundary of a region, inside that region $\scrA$ consists of independent Brownian bridges, conditioned so that the whole process remains nonintersecting and continuous. The typical spacing between the $k$th and $(k+1)$st Airy lines is $O(k^{-1/3})$, so this picture, along with Brownian scaling, suggests that on $o(k^{-2/3})$ time scales, $\scrA_k$ behaves like an independent Brownian motion.

\medskip

It is reasonable to expect that the last passage path across $\scrA$ from $(0, k)$ to $(x, 1)$ spends roughly the same amount of time -- $O(k^{-1})$ -- on each Airy line. By the above heuristic, Airy lines behave like Brownian motions on this scale, suggesting that $\langle (0, k) \LP x \rangle$ should be close to the corresponding Brownian last passage value of $2 \sqrt{2kx}$. Theorem \ref{T:airy-lp} proves this.

\medskip

We believe that $\langle (0, k) \LP x \rangle$ still behaves like a Brownian last passage value at a finer precision. In particular, we expect that the true fluctuation of $\langle (0, k) \LP x \rangle$ around $2 \sqrt{2kx}$ should be $O(k^{-1/6})$ as in Brownian last passage percolation.
 While the $o(k^{3/7}\log^d k)$ error we get could be improved by a more careful application of our methods, even our most optimistic heuristic proofs of the above theorem did not yield this error. It would be of interest to improve the above result to get a bound that is $o(1)$ as $k \to \infty$, as this would yield a slightly nicer representation of the Airy sheet in the limit; see Problem \ref{P:det-correction}.

\medskip

To prove Theorem \ref{T:airy-lp} we require structural results about the Airy line ensemble from \cite{DV}. The Brownian Gibbs property suggests that if we sample the points in $\scrA$ on a fine grid, then what lies in between is simply independent Brownian motions, conditioned not to intersect only when the grid points are close together. \cite{DV} make this picture rigorous. To state the results of that paper, we need two definitions. These definitions are illustrated in Figure \ref{fig:Bridge}.

\begin{definition}
\label{D:bridge-graph}
Fix $t > 0$. For a fixed $\ell > 0$, define $s_r = rt/\ell$ for all $r \in \{0, 1, \dots, \ell\}$. For $k, \de > 0$, we define a random graph $G_k(t, \ell, \de)$ on the set
$$
S_k(\ell) = \{1, \dots, k\} \X \{1, \dots, \ell \},
$$
where the points $(i, r)$ and $(i + 1, r)$ are connected if the two lines $\scrA_i$ and $\scrA_{i+1}$ are close at one of the two endpoints $s_{r-1}$ or $s_r$. That is,
$$
|\scrA_i(s_{r-1}) - \scrA_{i+1}(s_{r-1})| \le \de \quad \mathor \quad |\scrA_i(s_r) - \scrA_{i+1}(s_r)| \le \de.
$$
\end{definition}

\begin{definition}
\label{D:bridge-rep} The \textbf{bridge representation} $\mathcal{B}^k(t, \ell, \de)$ of the Airy line ensemble $\scrA$ is a sequence $(\scrB_1, \dots, \scrB_{2k})$ of functions from $[0, t]$ to $\R$ constructed as follows. For every line $i \in \{1, \dots, 2k\}$ and every $r \in \{1, \dots, \ell\}$, sample a Brownian bridge $B_{i, r}:[s_{r-1}, s_{r}] \to \R$ of variance $2$ with
$$
B_{i, r} (s_{r-1}) = \scrA_{i}(s_{r-1}) \quad \mathand \quad B_{i, r} (s_r) = \scrA_{i}(s_r).
$$
The bridges $B_{i, r}$ and $B_{i', r}$ are conditioned not to intersect if $(i, r)$ and $(i',r)$ are in the same component of $G_{2k}(t, \ell, \delta)$. We then define the $i$th line $\scrB_i$ of the line ensemble $\mathcal{B}^k(t, \ell, \de)$ by concatenating the bridges $B_{i, r}$. That is, $\scrB_i|_{[s_{r-1}, s_r]} = B_{i, r}$ for all $r \in \{1, \dots, \ell\}$.
\end{definition}

We now state the main structural result about the Airy line ensemble from \cite{DV}.

\begin{theorem}[\cite{DV}, Theorem 7.2]
\label{T:bridge-rep} There exist constants $c, d > 0$ such that the following holds for all $k \ge 3$, $t > 0$, $\ga \in (c \log(\log k)/\log k, 2]$ and $\ell \ge tk^{2/3 + \ga}$. The total variation distance between the laws of $\scrB^k(t, \ell, k^{-1/3 - \ga/4})|_{\{1, \dots, k\} \X [0, t]}$ and $\scrA|_{\{1, \dots, k\} \X [0, t]}$ is bounded above by
$$
\ell e^{-d \ga k^{\ga/12}}.
$$
\end{theorem}

\begin{figure}%
	\centering
	\begin{subfigure}[t]{6cm}
		\includegraphics[width=6cm]{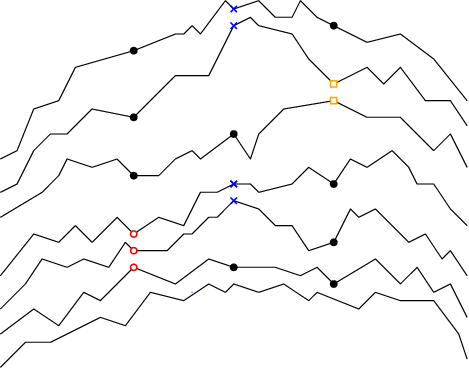}
		\caption{}
	\end{subfigure}
	\qquad
	\begin{subfigure}[t]{6cm}
		\includegraphics[width=6cm]{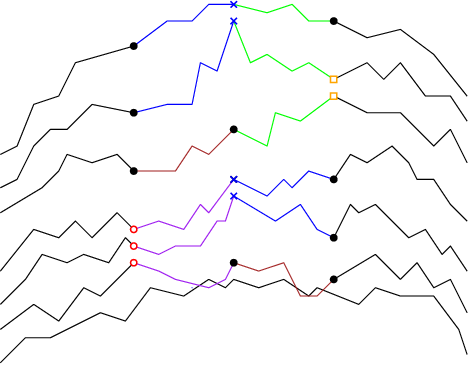}
		\caption{}
	\end{subfigure}
	\caption{An illustration of the bridge representation $\scrB$ of the Airy line ensemble. Figure \ref{fig:Bridge}(a) is the Airy line ensemble, with points at three times identified. Points with the same time coordinate are grouped together if they are close. To sample the bridge representation on this grid, we erase all lines between the specified points and resample independent Brownian bridges that are conditioned not to intersect each other if either of their endpoints are close (i.e. have the same colour). The result is Figure \ref{fig:Bridge}(b). If we only look at the top half of the lines in the bridge representation, then with high probability they do not intersect each other and resemble the Airy line ensemble.}%
	\label{fig:Bridge}%
\end{figure}

Theorem \ref{T:bridge-rep} shows that at the right scale, the Airy line ensemble can be represented as sequences of concatenated Brownian bridges. It will allow us to relate last passage across the Airy line ensemble to last passage across concatenated Brownian bridges, and then in turn to Brownian last passage percolation. The precise scale is essentially optimal given that the typical distance between the $k$th and $(k+1)$st Airy lines is $O(k^{-1/3})$. See \cite{DV} for more discussion of the scaling parameters in Theorem \ref{T:bridge-rep}. Note that in Theorem \ref{T:bridge-rep} we are required to sample bridges for $2k$ Airy lines, even though our comparison only concerns the top $k$ lines. The bridges with indices from $k+1$ to $2k$ take the place of the lower boundary condition in the usual Gibbs resampling. An index smaller than $2k$ is also possible, but for practical purposes this does not improve any of our estimates. 

\medskip

We will also need a structural result showing that edges in the graph $G_{2k}(t, \ell, \delta)$ are rare, and a modulus of continuity result for the Airy line ensemble.

\begin{prop}[Proposition 7.4, \cite{DV}]
\label{P:edge-spread}
	Fix $\ga \in (0, 2]$, and let $k \in \nat, t \in (0, \infty), \ell \in \N$. Let the graph $G := G_{2k}(t, \ell, k^{-1/3 - \ga/4})$ be as in Definition~\ref{D:bridge-graph}. For each $r \in \{1, \dots, \ell\}$, let
$$
V_r = \{i \in \{1, \dots, 2k\} :  \deg_G(i, r) \ge 1\}.
$$
In other words, $V_r$ is the set of vertices in $G$ with second coordinate $r$ that are connected to at least one other vertex. Then for any $\al \in (0, 1]$ there exist constants $c_{\al}, d_{\al}$ such that for all  $k \in \N, i \in \{\floor{k^{\al}} + 1, \dots, 2k\},$ and $m \le k^{\al/2}$, we have that
\begin{equation}
\label{E:Vj}
\p(|V_r \cap \{i - \floor{k^{\al}}, \dots, i \}| > mk^{\al - 3\ga/4}) \le c_{\al} e^{-d_{\al} m}.
\end{equation}

\end{prop}

\begin{theorem}[Theorem 8.1, \cite{DV}]
\label{T:mod-cont}
There exists a constant $d > 0$ such that for any $t > 0$, we have that
\begin{equation}
\label{E:k-summa}
\sum_{k \in \N} \p\lf( \sup_{s, s + r \in [0, t]} \frac{|\scrA_k(s) - \scrA_k(s + r)|}{\sqrt{r} \log^{1/2}(1 + r^{-1}) \log^d k} > 1 \rg) < \infty.
\end{equation}
\end{theorem}

The proof of Theorem \ref{T:airy-lp} relies on rewriting each line in the bridge representation of the Airy line ensemble as a Brownian motion plus error terms. Understanding last passage in this case can be handled by the following subadditivity lemma. The proof is straightforward, and so we omit it.

\begin{lemma}
\label{L:lp-subadd}
Let $f = (f_1, \dots, f_n)$ and
$$
L(f) = f[(0, n) \LP (t, 1)]
$$
be the last passage value across $f$ from time $0$ to time $t$, and let $F(f) = -L(-f)$ be the first passage value across $f$, again from the point $(0, n)$ to $(t, 1)$. Then for any $f, g \in C^n$ we have that
$$
L(f) + F(g) \le L(f + g) \le L(f) + L(g).
$$
\end{lemma}

\begin{proof}[Proof of Theorem \ref{T:airy-lp}]
We set $x=1$ for notational simplicity as the value of $x$ plays no important role. Let
$
\scrB^k =  \scrB^k(1, \ceil{k^{2/3 + \ga}}, k^{-1/3 - \ga/4})
$
be the bridge representation induced the division of time $\{s_r : r \in \{1, \dots, \ceil{k^{2/3 + \ga}}\}\}$ and the graph
$$
G_{2k} =G_{2k}(1, \ceil{k^{2/3 + \ga}}, k^{-1/3 - \ga/4}).
$$
Here $\ga \in (0, 1/3)$ is a parameter that we will optimize over later in the proof. By Theorem \ref{T:bridge-rep}, we can couple all the representations $\scrB^k$ with the Airy line ensemble $\scrA$ so that
\begin{equation}
\label{E:scrBkcouple}
\sum_{k \in \N} \p\lf(\scrB^k|_{\{1, \dots, k\} \X [0, 1]} \ne \scrA|_{\{1, \dots, k\} \X [0, 1]} \rg) < \infty.
\end{equation}
Hence it suffices to analyze the last passage time $L(\scrB^k)$ from $(0, k)$ to $(x, 1)$.

\medskip

\textbf{Step 1: Splitting up the paths.} By representing each of the Brownian bridges used to create $\scrB^k = (\scrB_{k, 1}, \dots, \scrB_{k, k})$ as a Brownian motion minus a random linear term, we can write
$$
\scrB_{k, i} = H_{k, i} + R_{k, i} + X_{k, i}
$$
Here the $k$-tuple $H_k = (H_{k, 1}, \dots, H_{k, k})$ consists of $k$ independent Brownian motions of variance $2$ on $[0, 1]$. The functions $R_{k, i}$ are piecewise linear with pieces defined on the time intervals $[s_{r-1}, s_{r}]$ for $r \in \{0, \dots, \ceil{k^{2/3 + \ga}}\}$, and the error term $X_{k, i}$ is equal to zero except for on intervals $[s_{r-1}, s_{r}]$ where the vertex $(i, r)$ is in a component of size greater than one in the graph $G_{2k}$. On such intervals, $X_{k, i}$ is the difference between a Brownian bridge from $0$ to $0$ and a Brownian bridge conditioned to avoid $U_{i, r} -1$ other Brownian bridges with certain start and endpoints. Here $U_{i, r}$ is the size of the component of $(i, r)$ in $G_{2k}$ and the two Brownian bridges used in the definition of $X_{k, i}$ are independent.

\medskip

By Lemma \ref{L:lp-subadd} applied twice, we have that
\begin{equation}
\label{E:subadd}
L(H_k)+F(R_k)+F(X_k) \le L(\scrB^k)\le L(H_k)+L(R_k)+L(X_k).
\end{equation}
By Theorem \ref{T:top-bd}, the main term
\begin{equation}
\label{E:Wk-2}
L(H_k)=2\sqrt{2k}+Y_kk^{-1/6},
\end{equation}
where $\{Y_k\}_{k \in \N}$ is a sequence of random variables satisfying a tail bound
$$
\p(|Y_k| > m) \le ce^{-dm^{3/2}}
$$
for  $c, d$ not depending on $m$ and $k$. To translate Theorem \ref{T:top-bd} to a bound on last passage values, we have used the identity \eqref{E:Wf-one}.

\medskip

\noindent \textbf{Step 2: Bounding the piecewise linear term.}  First, we have the bound
$$
|L(R_k)|, |F(R_k)| \le M_k,
$$
where $M_k$ is the maximum absolute slope of any of the piecewise linear segments in $R_k$. The slopes in $R_k$ come from increments in the Airy line ensemble minus the increments of the Brownian motions $H_k$ on the grid points. Recalling that $S_k(\ell) = \{1, \dots, k\} \X \{1, \dots, \ell\}$, we have the following upper bound for $M_k$:
\begin{equation*}
\ceil{k^{2/3 + \ga}}\lf[\max_{(i, r) \in S_k(\ceil{k^{2/3 + \ga}})} |H_{k, i}(s_{r}) - H_{k, i}(s_{r-1})| + \max_{(i, r) \in S_k(\ceil{k^{2/3 + \ga}})} |\scrA_i(s_{r}) - \scrA_i(s_{r-1})|\rg].
\end{equation*}
By a standard Gaussian bound on the first term and Theorem \ref{T:mod-cont} for the second term, for some $d \in \N$ we have that
\begin{equation}
\label{E:Mk}
\sum_{k \in \N} \p \lf(M_k \ge k^{1/3 + \ga/2} \log^d k \rg) < \infty.
\end{equation}
\medskip

\noindent \textbf{Step 3: Bounding the large component error.}
To bound $L(X_k)$ and $F(X_k)$, we divide $\{1,\dots,k\}$ into $n= \ceil{k^{2/3+\gamma}}$ intervals
$$
I_{k, i} = \lf\{ \floor{\frac{(i-1)k}n} + 1, \dots, \floor{\frac{ik}n}\rg\}, \qquad i \in \{1, \dots, n\}.
$$
This, and the division of time into the intervals $[s_{r-1}, s_r]$ for $r \in \{1, \dots, n\}$ breaks the line ensemble $X_k$ into $n^2$ boxes. Each last passage path can meet at most $2n-1$ of these boxes. So we have that
\begin{equation}
\label{E:Xk}
L(X_k) \le (2n-1)Z_k,
\end{equation}
where $Z_k$ is the maximal last passage value among all values that start and end in the same box (including the boundary). Specifically,
$$
Z_k = \max_{(i, r) \in [1, n]^2} \max \lf\{ X_k[(\ell_1, t_1) \LP (\ell_2, t_2)] : \ell_1, \ell_2 \in I_{k, i}, \; t_1, t_2 \in [s_{r-1}, s_r]\rg\}.
$$
We have that $Z_k \le N_kD_k$, where
\begin{align*}
N_k &= \max_{(i, r) \in [1, n]^2} \card{\lf\{ \ell \in I_{k, i} : X_{k,\ell}|_{[s_{r-1}, s_{r}]} \ne 0\rg\}} \quad \mathand \\
 D_k &= \max \bigg\{ |X_{k, \ell}(t) - X_{k, \ell}(m)| : \ell \in [1, k], t, m \in [s_{r-1}, s_r] \text{ for some } r \in \{1, \dots, n\} \bigg\}.
\end{align*}
That is, $N_k$ is the maximum number of nonzero line segments in any box, and $D_k$ is the maximum increment over any line segment in a box. Since $X_{k, \ell} = \scrB_{k, \ell} - H_{k, \ell} - R_{k, \ell}$, we can bound $D_k$ in terms of the deviations of the other paths. To bound the deviation of $R_{k, \ell}$, we use the bound on $M_k$ above. The deviation of $H_{k, \ell}$ can be bounded with standard bounds on Gaussian random variables. On the event where $\scrB^k|_{\{1, \dots, k\} \X [0, 1]} = \scrA|_{\{1, \dots, k\} \X [0, 1]}$, we can bound the deviation of $\scrB_{k, \ell}$ using Theorem \ref{T:mod-cont}. Therefore for some constant $d \in \N$, we have
\begin{equation}
\label{E:Dka}
\sum_{k=1}^\infty \p\lf(D_k > k^{-1/3 - \gamma/2}\log^d k, \; \scrB^k = \scrA|_{\{1, \dots, k\} \X [0, 1]} \rg) < \infty.
\end{equation}
Combining equations \eqref{E:Dka} and \eqref{E:scrBkcouple} gives
\begin{equation}
\label{E:Dk}
\sum_{k=1}^\infty \p\lf(D_k > k^{-1/3 - \gamma/2}\log^d k \rg) < \infty.
	\end{equation}
The quantity $N_k$ is equal to the maximum number of edges in the graph $G_k$ in a region of the form $I_{k, i} \X \{r\}$ for some $r \in \{1, \dots, n\}$. This can be bounded by using Proposition \ref{P:edge-spread} and a union bound, which yields
$$
\sum_{k \in \N} \p \lf(N_k > k^{1/3 - \ga} k^{- 3 \ga/4} \log^2 k \rg) < \infty.
$$
Combining this with the bound in \eqref{E:Xk} and \eqref{E:Dk} implies that for some constant $d > 0$, that
\begin{equation}
\label{E:Mk-2}
\sum_{k \in \N} \p\lf(L(X_k) > k^{2/3 - 5\ga/4} \log^d k \rg) < \infty.
\end{equation}
We can symmetrically bound $F(X_k)$.

\medskip

\noindent \textbf{Step 4: Putting it all together.}
By combining the inequalities \eqref{E:subadd}, \eqref{E:Wk-2}, \eqref{E:Mk} and \eqref{E:Mk-2}, we get that for some $d \in \N$,
$$
\sum_{k \in \N} \p\lf( |L(\scrB^k) - 2\sqrt{2k}| >  k^{2/3 - 5\ga/4} \log^d k + k^{1/3 + \ga /2} \log^d k \rg) < \infty.
$$
Taking $\ga = 4/21$ and increasing the power of $\log k$ from $d$ to $d + 1$ completes the proof.
\end{proof}

\section{Melon paths are parabolas}
\label{S:parabola}
In this section we use the results of Section \ref{S:short-lp} to establish bounds on the location of melon last passage paths. We will also establish that last passage paths that start or end close together meet with high probability. These facts will allow us to construct the Airy sheet in Section \ref{S:airy-sheet}. Throughout the section we write $W = WB^n$ for the melon, and use the last passage and scaling notation $\bar x = 2x n^{-1/3}$ and $\hat y = 1 + 2y n^{-1/3}$ introduced in Section \ref{SS:notation}.

\medskip
Let $\hat Z^n_k(x,y) = 1 + n^{-1/3} Z^n_k(x,y)$ denote the jump time from line $k+1$ to $k$ on the rightmost last passage path from $\bar{x}$ to $\hat y$ in the melon, see \eqref{E:jump}. Observe that $Z^n_k(x, y)$ is nonincreasing in $k$, and $Z^n_k(x,y)$ is a nondecreasing function in both $x$ and $y$ by monotonicity of last passage paths, Lemma \ref{L:mono-path}. The next lemma gives asymptotics for $Z^n_k$.

\begin{lemma}
\label{L:zk-unif-bound}
Let $K$ be a compact subset of $(0, \infty) \X \R$. Then
\begin{equation}
\label{E:xyKxyK}
\sup_{(x,y)\in K} \lf|Z^n_k(x, y) +\sqrt{\frac{k}{2x}}\rg| =\oo(\sqrt{k})
\end{equation}
and $Z^n_k(x, y)$ is tight as a function of $n$ for each fixed $k \in \N, (x, y) \in (0, \infty) \X \R$.
\end{lemma}

\begin{proof}
We first fix $x,y\in K$, rescale by $n^{1/6}$ and center so that the triangle inequality
\begin{equation*}
\{\bar x \LP (\hat z,k)\} + \{(\hat z,k)\LP \hat y\} \le   \{\bar x\LP \hat y \}
\end{equation*}
reads 
\begin{equation}\label{E:triangle2}
F^n_k(x, z)+G^n_k(z, y)\le H_n(x, y) 
\end{equation}
with
\begin{eqnarray*}
H_n(x, y)&=&n^{1/6}\{\bar x\LP \hat y \}-2n^{2/3}-2(y-x)n^{1/3}, \\
F^n_k(x, z) &=& n^{1/6}(\{\bar x \LP (\hat z,k)\}-W^n_k(\hat z))+2xn^{1/3},
\\ G^n_k(z, y)&=&n^{1/6}(W^n_k(\hat z) + \{(\hat z,k )\LP \hat y\})-2yn^{1/3}-2n^{2/3}.
\end{eqnarray*}
The basic proof strategy for bounding $Z^n_k(x, y)$ is as follows. On the one hand,
\begin{equation}
\label{E:Hnxy}
F^n_k(x, Z^n_k(x, y)) + G^n_k(Z^n_k(x, y),y)=H_n(x, y)
\end{equation}
On the other hand, we can show that $ F^n_k(x, z) + G^n_k(z, y)<H_n(x, y)$ whenever $z$ is sufficiently far away from $-\sqrt{k/(2x)}$. 
To show this inequality, we use the bound on melon last passage values given in Proposition \ref{P:longbound} to control $F^n_k$, and Theorem \ref{T:top-bd} which implies that $H_n(x,y)$ is tight in $n$ for $x,y$ fixed.

\medskip

We will show that for every $\ep \in (0, 1)$ we have
\begin{equation}
\label{E:zzz}
\sup_{z : \;|z+\sqrt{k/(2x)}|>\epsilon \sqrt{k}} F^n_k(x, z)+G^n_k(z, y) \le -\epsilon^2\sqrt{kx}/2 + \oo(\sqrt{k}).
\end{equation}
By Lemma \ref{L:double-mono},
$F^n_k(x, \cdot)$ is monotonically increasing and $G^n_k(\cdot, y)$ is monotonically decreasing. We can use this monotonicity to bound the left hand side of \eqref{E:zzz} by a supremum over a finite set. 
 Let
$
A = (12\ep^2)^{-1}\Z \cap [1/4, 2],
$
and for $z \in [-n^{1/3} + x, y]$, define 
\begin{equation}
\label{E:floorznk}
\begin{split}
    \floor{z}_{n, k} &= \max \{w \in -\sqrt{k/x}A \cup \{x - n^{1/3}\} : w < z\} \qquad \mathand\\
\ceil{z}_{n, k} &= \min \{w \in -\sqrt{k/x}A \cup \{y\} : w > z\}.
\end{split}
\end{equation}
We also set $\floor{x - n^{1/3}}_{n, k} = x - n^{1/3}$ and $\ceil{y}_{n, k} = y$. The monotonicity of $F^n_k(x, \cdot)$ and  $G^n_k(\cdot, y)$ implies that the left hand side of \eqref{E:zzz} is bounded above by
\begin{equation}
\label{E:with-part}
    \sup_{z : \;|z+\sqrt{k/(2x)}|>\epsilon \sqrt{k}} F^n_k(x, \ceil{z}_{n, k})+G^n_k(\floor{z}_{n, k}, y).
\end{equation}
Therefore to show \eqref{E:zzz}, we just need to show that \eqref{E:with-part} is bounded above by $-\ep^2\sqrt{kx}/2 + \oo(\sqrt{k})$. The quantity in \eqref{E:with-part} is easier to work with since we are taking a supremum over only finitely many distinct terms. Moreover, the number of terms is uniformly bounded in $n$ and $k$, so it is enough to control the terms individually. Note that the bound on $F^n_k(x, z)+G^n_k(z, y)$ in \eqref{E:with-part} is much cruder when $z \notin -\sqrt{k/x} A$. 
This crude estimate will suffice for our purposes since the quantity inside the supremum in \eqref{E:zzz} is not close to the maximum for such $z$. 

\medskip

In particular, setting $z_{k, a} =  - a\sqrt{k/x}$, it is enough to show that
\begin{equation}
\label{E:with-part-2}
F^n_k(x, \ceil{z_{k, a}}_{n, k})+G^n_k(\floor{z_{k, a}}_{n, k}, y) \le -\ep^2\sqrt{kx}/2 + \oo(\sqrt{k})
\end{equation}
for every fixed $a \in [1/4, 1/\sqrt{2} - \ep] \cup [1/\sqrt{2} + \ep, 2]$.




\medskip

To prove \eqref{E:with-part-2}, we establish pointwise bounds on $F^n_k$ and $G^n_k$. Proposition \ref{P:longbound} gives that for a fixed $a > 0$ we have
\begin{equation}
    \label{E:Fnk-bd-2}
    F^n_k(x, z_{k, a}) = 2\sqrt{kx}(\sqrt{2} - a) +  \oo(\sqrt{k}).
\end{equation}
Proposition \ref{P:longbound} also yields the bound
\begin{equation}
    \label{E:Fnk-bd}
    F^n_k(x, y) = 2\sqrt{2kx}  + \oo(\sqrt{k}).
\end{equation}
The triangle inequality \eqref{E:triangle2} with   $x'=  x/(2a^2)$ gives
\begin{equation}
    \label{E:FnkGnk}
    G^n_k(z_{k, a}, y) \le H_n(x', y) - F^n_k(x', z_{k, a}).
\end{equation}
Now, $H_n(x', y)$ is equal to a rescaled and shifted Brownian last passage value by Proposition \ref{P:wm-equivalent}. Therefore by Theorem \ref{T:top-bd} and \eqref{E:Wf-one}, which together give bounds on single Brownian last passage values, it is tight in $n$ and hence $H_n(x', y)  = \oo(\sqrt{k})$. Moreover, Proposition \ref{P:longbound} gives that
$$
F^n_k(x', z_{k, a}) = 2 \sqrt{2 k x'} + 2 z_{k, a}  x' + \oo(\sqrt{k}) = \frac{\sqrt{kx}}{a}  + \oo(\sqrt{k})
$$
and so
\begin{equation}
    \label{E:FnkGnk-2}
    G^n_k(z_{k, a}, y) \le -\frac{\sqrt{kx}}a  + \oo(\sqrt{k}).
\end{equation}
We also have the bound
\begin{equation}
    \label{E:FnkGnk-3}
    G^n_k(x - n^{1/3}, y) \le H_n(0, y) - F^n_k(0, x - n^{1/3}) = H_n(0, y) = \oo(\sqrt{k}).
\end{equation}
The first equality here follows from the fact that  $F^n_k(0,\cdot)=0$, and the second equality again follows from Theorem \ref{T:top-bd}.

\medskip

Finally, the bound in \eqref{E:with-part-2} follows from \eqref{E:Fnk-bd-2} and \eqref{E:FnkGnk-3} when $a \le \inf A$, from \eqref{E:Fnk-bd-2} and \eqref{E:FnkGnk-2} when $a \in (\inf A, \sup A)$,  and from \eqref{E:Fnk-bd} and \eqref{E:FnkGnk-2} when $a \ge \sup A$.

\medskip

Now by Theorem \ref{T:top-bd},  $H^n(x, y)$ is tight in $n$. With \eqref{E:zzz} and \eqref{E:Hnxy} this implies that $Z^n_k(x,y)=\sqrt{k/(2x)}+ \oo(\sqrt{k})$ for $x,y$ fixed. Since $Z^n_k(x,y)$ is monotone in $x$ and $y$, the claim \eqref{E:xyKxyK} follows.

\medskip

For every $k, x, y$, the sequence $Z^n_k(x, y)$ is tight in $n$ since $Z^n_k(x, y) + \sqrt{k/(2x)} = \oo(\sqrt{k})$ and we have the monotonicity $y = Z^n_1(x, y) \ge Z^n_{2}(x, y) \ge \dots$. 
\end{proof}

Lemma \ref{L:zk-unif-bound} has an important consequence for disjointness of last passage paths.
Recall that two paths $\pi$ and $\tau$ are \textbf{disjoint} if either $\pi > \tau$ or $\tau > \pi$ on the intersection of the interiors of both their domains. Recall that $\pi\{x, y\}$ is the rightmost last passage path in the melon from $(x, n)$ to $(y, 1)$.

\begin{lemma}
	\label{L:close-not-dj}
	Fix $x > 0$ and $y_1 < y_2 \in \R$. Then
	\begin{equation}\label{E:close-not-dj}
	\lim_{\ep \to 0^+} \limsup_{n \to \infty} \p \big( \pi\{\bar{x} - \bar \ep, \hat{y}_1\} \;\;\; \mathand \;\;\; \pi\{\bar{x} + \bar{\ep}, \hat{y}_2\} \text{ are disjoint} \big) = 0.
	\end{equation}
\end{lemma}

\begin{proof}
	We will prove a stronger statement, with the leftmost last passage path $\pi^-\{\bar{x} - \bar \ep, \hat{y}_1\}$ replacing one  of the rightmost paths $\pi\{\bar{x} - \bar \ep, \hat{y}_1\}$. Disjointness of $\pi\{\bar{x} - \bar \ep, \hat{y}_1\}$ and $\pi\{\bar{x} + \bar{\ep}, \hat{y}_2\}$ implies disjointness of $\pi^-\{\bar{x} - \bar \ep, \hat{y}_1\}$ and $\pi\{\bar{x} + \bar{\ep}, \hat{y}_2\}$ by monotonicity. 
	
	\medskip
	
	By Lemma \ref{L:zk-tilde}, disjointness of the paths $\pi^-\{\bar{x} - \bar \ep, \hat{y}_1\}$ and $\pi\{\bar{x} + \bar{\ep}, \hat{y}_2\}$ is equivalent to disjointness of the original Brownian last passage paths $\pi^-[\bar{x} - \bar \ep, \hat{y}_1]$ and $\pi[\bar{x} + \bar{\ep}, \hat{y}_2]$. Here $\pi^-[\bar{x} - \bar \ep, \hat{y}_1]$ is the leftmost last passage path in $B^n$ from $\bar{x} - \bar \ep$ to $\hat{y}_1$. Hence the probability in \eqref{E:close-not-dj} is bounded above by 
	\begin{equation}
	\label{E:new53}
 \p \big( \pi^-[\bar{x} - \bar \ep, \hat{y}_1] \;\;\; \mathand \;\;\; \pi[\bar{x} + \bar{\ep}, \hat{y}_2] \text{ are disjoint} \big).
	\end{equation}
	By time-reversal symmetry of the increments of Brownian motion under the map $t\mapsto 1-t$
	 the probability in \eqref{E:new53} equals
\begin{equation}
	\label{E:brown-path}
 \p \big( \pi^-[-\bar{y}_2, 1-\bar{x} - \bar \ep] \;\;\; \mathand \;\;\; \pi[-\bar{y}_1, 1-\bar{x} + \bar{\ep}] \text{ are disjoint} \big).
\end{equation}
	By translation invariance and Brownian scaling, the probability \eqref{E:brown-path} remains unchanged if the points $-\bar{y}_2, 1-\bar{x} - \bar \ep, -\bar{y}_1, 1 - \bar{x} + \bar{\ep}$ are replaced by their images under any linear function $L(t) = at + b$ for some $a > 0$. In particular, for each $n$ we may choose the linear function $L = L_{n, \ep}$ sending $-\bar y_1 \mapsto \bar y:=2(\bar y_2 - \bar y_1)$ and $1 - \bar x + \bar \ep \mapsto 1$. For $t \in [-1, 2]$, we have
	$$
	L_{n, \ep}(t) = (1 - 2\bar y_2 + \bar y_1 + \bar x - \bar \ep)t + 2 \bar y_2 - \bar y_1 + O(n^{-2/3}).
	$$
	Therefore for all large enough $n$, we have
	\begin{equation}
	\label{E:Lnbd}
	L_{n, \ep} (-\bar y_2)  = \bar y_2 - \bar y_1 + O(n^{-2/3}) \ge 0, \qquad L_{n, \ep} (1 - \bar x - \bar \ep) = 1 - 2 \bar \ep + O(n^{-2/3}) \ge 1 - 3 \bar \ep.
	\end{equation}
	For such $n$, after translating back to melon paths we get that the probability in \eqref{E:brown-path} is equal to
	$$
	\p \big( \pi^-\{L_{n, \ep} (-\bar y_2), L_{n, \ep}(1 - \bar x - \bar \ep)\} \;\;\; \mathand \;\;\; \pi\{\bar y, \hat 0 \} \text{ are disjoint} \big).
	$$
By monotonicity of last passage paths, Lemma \ref{L:mono-path}, and \eqref{E:Lnbd}, this is bounded above by
	\begin{equation}
	\label{E:pbigg}
	    	\p \big( \pi^-\{0 , 1-3 \bar \ep \} \;\;\; \mathand \;\;\; \pi\{\bar y, \hat 0 \} \text{ are disjoint} \big).
	\end{equation}
	Now, the path $\pi^-\{0,1-3 \bar \ep \}$ starts at zero and therefore simply follows the top line in the melon, so the paths $\pi^-\{0 , 1-3 \bar \ep \}$ and $\pi\{\bar y, \hat 0 \}$ are disjoint if and only if $\pi\{\bar y, \hat 0 \}$ jumps up to line $1$ after time $1-3\bar \ep$. This jump time is $\hat Z^n_1(y, 0)$, so \eqref{E:pbigg} is equal to
	$$
	\p \big( Z^n_1(y, 0) \ge - 3\ep \big).
	 $$
	Hence to prove \eqref{E:close-not-dj} we just need to show that
	\begin{equation}
	\label{E:brown-path'}
	\lim_{\ep \to 0^+} \limsup_{n \to \infty}\p \big( Z^n_1(y, 0) \ge - 3\ep \big) = 0.
	\end{equation}
	To prove \eqref{E:brown-path'}, we just need to show that any subsequential limit $Z_1(x, 0)$ of the sequence of random variables $\{Z^n_1(x, 0) : n \in \N \}$ is strictly negative almost surely. Note that this sequence is tight by Lemma \ref{L:zk-unif-bound}.
	
	\medskip
	
	Let $A^n$ denote the rescaling of the melon $WB^n$ in Theorem \ref{T:airy-line}. By that theorem and Lemma \ref{L:zk-unif-bound}, the collection of random variables $(A^n, \{Z^n_k(x, 0) : k \in \N\})$ is tight. Let $(\scrA, \{Z_k(x, 0) : k \in \N\})$  denote a joint subsequential limit of these random variables. The asymptotics in Lemma \ref{L:zk-unif-bound} guarantee that
	\begin{equation}
	\label{E:kinfty}
	\lim_{k \to \infty} Z_k(x, 0) = -\infty.
	\end{equation}
	almost surely. Moreover, for any $k \in \N$ the points $\{(Z_i(x, 0), i) : i \le k\}$ are jump times along a last passage path from $(Z_k(x, 0), k)$ to $(1, 0)$ in $\scrA$ since the prelimiting points satisfied this property.
	
	\medskip
	
By \eqref{E:kinfty}, there exists a random $K \in \N$ such that $Z_1(x, 0)$ is a jump time on a last passage path in $\scrA$ from $(-1, K)$ to $(0, 1)$. Now, by Proposition \ref{P:brownian-airy}, the top $k$ lines of $\scrA$ restricted to the interval $[-1,0]$ are absolutely continuous with respect to $k$ independent Brownian motions. Therefore for every $k \in \N$, all jump times on any last passage path in $\scrA$ from $(-1, k)$ to $(0, 1)$ are contained in the open interval $(-1, 0)$. Hence all jump times on a last passage path in $\scrA$ from $(-1, K)$ to $(0, 1)$ are contained in the open interval $(-1, 0)$. In particular, $Z_1(x, 0) < 0$ almost surely, as desired.
	\end{proof}

\section{Constructing the Airy sheet}
\label{S:airy-sheet}

In this section, we construct the joint limit of last passage values at two times, known as the Airy sheet. We start by recalling the definition given in the introduction. Recall the notation $\langle \rangle$ for last passage values across $\scrA$ from Section \ref{SS:notation}.

\begin{definition}
\label{D:airy-sheet} The \textbf{Airy sheet} is a random continuous function $\scrS:\mathbb R^2\to \mathbb R$ so that
\begin{enumerate}[label=(\roman*)]
\item
$\scrS$ has the same law as  $\scrS(\cdot+t, \cdot+t)$ for all $t\in \mathbb R$.
\item $\scrS$ can be coupled with an Airy line ensemble $\scrA$ so that $\scrS(0,\cdot)=\scrA_1(\cdot)$ and for all $(x, y, z) \in \Q^+ \X \Q^2$ there exists a random variable $K_{x, y,z} \in \N$ such that for all $k \ge K_{x, y,z}$, almost surely
$$
\big\langle (-\sqrt{k/(2x)}, k) \LP z \big\rangle - \big \langle (-\sqrt{k/(2x)}, k) \LP y \big\rangle = \scrS(x, z) - \scrS(x, y).
$$
\end{enumerate}
\end{definition}

We leave the existence of the Airy sheet to Theorem \ref{T:airy-sheet}, and first show that it is unique.

\begin{prop}
	\label{P:uniq-sheet}
The Airy sheet is unique in law.
\end{prop}

\begin{proof}
By equation (5.15) in \cite{prahofer2002scale} $\scrA_1(z)+z^2$ is stationary and ergodic.  By Birkhoff's ergodic theorem for any fixed $y$ we have that almost surely,
\begin{equation*}
\scrS(0, y)  =  \lim_{m \to \infty} \frac{1}{2m} \int_{-m}^{m} \scrS(0, y)-(\scrS(0,z) + z^2) dz + \expt \scrA_1(0).
\end{equation*}
We can then use property (i) to translate the above formula (applied to $\scrS(0, y- x)$) to get an almost sure formula for any $\scrS(x, y)$:
\begin{equation*}
\scrS(x, y)  =  \lim_{m \to \infty} \frac{1}{2m} \int_{-m}^{m} \scrS(x, y)-(\scrS(x,z + x) + z^2) dz + \expt \scrA_1(0).
\end{equation*}
When $(x,y) \in \Q^+ \X \Q$, the integrand on right hand side above is determined by condition (ii) for rational values of $z$, and hence is determined by that condition for all values of $z$ by continuity. Therefore $\{\scrS(x, y) : (x,y) \in \Q^+ \X \Q\}$ is determined by the definition of $\scrS$. By stationarity and continuity, this implies that the distribution of $\scrS$ is uniquely determined by its definition.
\end{proof}

\begin{remark}
	\label{R:busemann-def}
	We can exchange condition (ii) in Definition \ref{D:airy-sheet} for the following Busemann function definition. Almost surely, for all $x > 0$ and $y, z \in \R$, we have that
	\begin{equation}
	\label{E:buse}
\lim_{k \to \infty} \big\langle (-\sqrt{k/(2x)}, k) \LP z \big\rangle - \big \langle (-\sqrt{k/(2x)}, k) \LP y \big\rangle = \scrS(x, z) - \scrS(x, y).
	\end{equation}
	The proof of Proposition \ref{P:uniq-sheet} implies that this definition gives rise to a unique object.
	 Moreover, condition (ii) of Definition \ref{D:airy-sheet} implies this definition, and so they must be the same. To see this, note that it clearly implies \eqref{E:buse} for rational triples. To extend to $x \in \Q^+$ and $y, z \in \R$, observe that
	$$
	\big\langle (-\sqrt{k/(2x)}, k) \LP z + \ep \big\rangle \ge \big\langle (-\sqrt{k/(2x)}, k) \LP z \big\rangle + \scrA_1(z + \ep) - \scrA_1(z).
	$$
	Therefore since the left hand side of \eqref{E:buse} is continuous when restricted to rational $x, y, z$, it is continuous for $x \in \Q^+$ and $y, z \in \R$. To extend to $x \in \R^+$, note that the left hand side of \eqref{E:buse} is monotone in $x$ by Proposition \ref{P:mono-inc}. Therefore it must again be continuous since it is continuous on rationals.
\end{remark}

The Airy sheet exists because it is the limit of Brownian last passage percolation.  More precisely, we have the following. For $n\in \N$ let $B^n$ be an $n$-tuple of independent two-sided Brownian motions and let $[x \LP y]_n$ be the last passage value there from $(x, n)$ to $(y, 1)$. Recall the scaling $\bar{x} = 2xn^{-1/3}$ and $\hat{y} = 1 + 2yn^{-1/3}.$ Define the sequence of prelimiting Airy sheets $\scrS_n(x, y)$ by the formula
$$
[\bar{x} \LP \hat{y}]_n = 2\sqrt{n} + (y - x)n^{1/6} + n^{-1/6} \scrS_n(x, y).
$$
\begin{theorem}
\label{T:airy-sheet} The Airy sheet $\scrS$ exists. Moreover, there exists a coupling so that $\scrS_n-\scrS$ is asymptotically small in the sense that
\begin{equation}\label{E:super-exp}
\text{for every compact }K\subset \R^2\text{ there exists }a>1\text{ with }\E a^{\sup_K |\scrS_n-\scrS|^{3/2}}\to 1.
\end{equation}
\end{theorem}
We first  show tightness and then prove  that all subsequential limits satisfy the
definition.
\begin{lemma}
\label{L:Sn-tight}
$\scrS_n$ is a tight sequence of random functions in $C(\R^2, \R)$. Moreover, if $\scrS$ is a limit of $\scrS_n$ along any subsequence, then there exists a coupling of $\scrS_n$ and $\scrS$ such that \eqref{E:super-exp} holds.
\end{lemma}

For the proof, $c$ and $d$ will be constants that may change from line to line and are independent of $n$. They will depend on an initial choice of a compact set.
\begin{proof}
It suffices to prove tightness and  \eqref{E:super-exp} for $\scrS$ restricted to compact sets of the form $K = [-b, b]^2$. First, by Theorem \ref{T:TW-airy}, we have
\begin{equation}
\label{E:Sn-m}
\p(|\scrS_n(0, 0)| > m) \le ce^{-dm^{3/2}}.
\end{equation}
Second, $\scrS_n(x, \cdot)$ and $\scrS_n(\cdot, y)$ are both given by the rescaled top line of a Brownian melon. Therefore tail bounds for the melon in Proposition \ref{P:dyson-tails} and the modulus of continuity of Lemma \ref{L:levy-est} imply that on $[-b, b]^2$ we have
$$
|\scrS_n(x, y) - \scrS_n(x', y')| \le C_n ||(x, y) - (x', y')||^{1/2} \log^{1/2} \lf( \frac{2b}{||(x, y) - (x', y')||} \rg)
$$
for a sequence of constants $C_n$ satisfying
\begin{equation}
\label{E:Cn-m}
\p(C_n > m) \le ce^{-dm^{3/2}}.
\end{equation}
 By the Kolmogorov-Chentsov criterion, see Corollary 16.9 in \cite{kallenberg2006foundations}, this uniform modulus of continuity bound coupled with the bound \eqref{E:Sn-m}  implies tightness of $\scrS_n|_K$. Moreover, if $\scrS_n|_K \to \scrS|_K$ in distribution along a subsequence, then by Skorokhod's representation theorem we can find a coupling so that $\scrS_n \to \scrS$ almost surely. In particular, this implies that $a^{\sup_K |\scrS_n-\scrS|^{3/2}}\to 1$ almost surely for every $a > 1$. The limit $\scrS$ satisfies the same modulus of continuity estimate as the sequence $\scrS_n$, with a random constant $C$ satisfying the tail bound in \eqref{E:Cn-m}. Therefore we have the bound
$$
a^{\sup_K |\scrS_n-\scrS|^{3/2}} \le a^{c(|\scrS_n(0, 0)|^{3/2} + |\scrS(0, 0)|^{3/2} + C_n^{3/2} + C^{3/2})}.
$$
All four of the random variables in the exponent above satisfy tail bounds of the form \eqref{E:Sn-m} or \eqref{E:Cn-m}, and so the above random variable is uniformly integrable for small enough $a > 1$. Hence we can conclude the desired convergence in expectation.
\end{proof}

Any subsequential limit $\scrS$ of $\scrS_n$ satisfies property (i) of the Airy sheet since the $\scrS_n$ are stationary by stationarity of Brownian increments. So it suffices to show that any limit restricted to $\Q^+ \X \Q$ satisfies property (ii) in Definition \ref{D:airy-sheet}.

\medskip

With this in mind, let $\scrS$ be any subsequential limit of $\scrS_n$ along a subsequence $Y'$. We first show that there is a further subsequence $Y \sset Y'$, and a coupling of the processes $B^n, n \in Y$ with limiting objects $\scrS, \scrA$ such that the following convergences hold on a set $\Omega$ of probability 1. All limits and claims about $n$ are for  $n\in Y$.
\begin{enumerate}
\item $\scrS_n\to \scrS$ uniformly on compact sets in $\R^2$.
\item Let $A_n$ denote the rescaling of the melon $WB^n$ in Theorem \ref{T:airy-line}. Then $A_n$ converges to the Airy line ensemble $\scrA$ uniformly on compact sets in $\Z \X \R$.
\item For every $(x, y) \in \Q^+ \X \Q$ and $n \in \N$ the sequence $Z_k^n(x,y)$ has some  limit $Z_k(x,y)$. Moreover, as $k\to\infty$,
\begin{equation}
\label{E:asym-Z}
Z_k(x, y)/\sqrt{k} \to -1/{\sqrt{2x}}.
\end{equation}
\item For every triple $(x, y, z) \in \Q^+ \X \Q^2$ with $y < z$, there exist random points $X_1 < x < X_2 \in \Q^+$ such that the melon paths
$$
\pi\{\bar{X}_1, \hat{y}\}_n \;\; \mathand \;\; \pi\{\bar{X}_2, \hat{z} \}_n
$$
are not disjoint for $n$ large enough.
\end{enumerate}

\begin{proof}[Proof of the existence of such a coupling]
Define indicator functions
$$
I_n(x, y, z, \ep) = \indic \Big\{\,\pi\{\bar x - \bar {\ep}, \hat y\}_n \; \mathand \; \pi\{\bar x + \bar {\ep}, \hat{z} \}_n \text{ are not disjoint}\Big\}.
$$
Each of the countably many sequences in $n$:
$$
\scrS_n, A_n, \{Z^n_k(x, y) : k \in \N, (x, y) \in \Q^+ \X \Q \}, \{I_n(x, y, z, m^{-1}) : (x,y,z,m)\in \Q^+\times  \Q^2 \times \N\},
$$
is tight. This uses Lemma \ref{L:Sn-tight} for $\scrS_n$, Theorem \ref{T:airy-line} for $A_n$, Lemma \ref{L:zk-unif-bound} for $Z^n_k(x, y)$, and the boundedness of $I_n(x, y, z, \ep)$. Thus they are jointly tight in the product of the appropriate topologies.

By Skorokhod's representation theorem, along any subsequence where $\scrS_n \cvgd \scrS$ we can find a further subsequence and a coupling of the environments such that all of these random variables converge almost surely. Property 1 clearly holds along this subsequence. The limit of $A_n$ is an Airy line ensemble $\scrA$ by Theorem \ref{T:airy-line}, giving property 2. The limits $Z_k(x, y)$ of $Z^n_k(x, y)$ satisfy \eqref{E:asym-Z} almost surely by the asymptotics in Lemma \ref{L:zk-unif-bound}, giving property 3. The limits of each $I_n(x, y, z, m^{-1})$ is an indicator function $I(x, y, z, m^{-1})$. Lemma \ref{L:close-not-dj} implies that for all $(x,y, z) \in \Q^+\times\Q^2$ we have
\begin{equation}
\label{E:cvge}
\p (I(x, y, z, m^{-1}) = 1) \to 1 \qquad \mathas m \to \infty.
\end{equation}
Monotonicity of last passage paths guarantees that $I_n(x, y, z, m^{-1})$ is nondecreasing in $m$, and this carries over to the limit $I(x, y,z, m^{-1})$. This monotonicity and \eqref{E:cvge} guarantees that there exists a random $M(x, y, z) \in \N$ such that
$
I(x, y, z, M(x, y, z)^{-1}) = 1
$
almost surely. Since
$$
I_n(x, y, z, M(x, y, z)^{-1}) \to I(x, y, z, M(x, y, z)^{-1})
$$
almost surely, $I_n(x, y, z, M(x, y, z)^{-1}) = 1$ for all large enough $n$, giving property 4 with $X_1 = x -M(x, y, z)^{-1}$ and $X_2 = x + M(x, y, z)^{-1}$.
\end{proof}

Theorem \ref{T:airy-sheet} then follows immediately from the following deterministic statement about the relationship between the subsequential limit $\scrS$ and the Airy line ensemble $\scrA$. Indeed, this next lemma shows that any distributional subsequential limit $\scrS$ of $\scrS_n$ must be an Airy sheet. Hence by the uniqueness of the Airy sheet law (Proposition \ref{P:uniq-sheet}) and the tightness of $\scrS_n$ (Lemma \ref{L:Sn-tight}), the sequence $\scrS_n$ converges in distribution to the Airy sheet.

\begin{lemma}
	\label{L:sheet-lem}
On the set $\Om$, the processes $\scrS$ and $\scrA$ satisfy condition (ii) of Definition \ref{D:airy-sheet}.
\end{lemma}

Lemma \ref{L:sheet-lem} puts together all the bounds that we have developed over the previous few sections. See Figure \ref{fig:lemma85} for a sketch of its proof.

\begin{figure}
	\centering
	\includegraphics[width=2.5in]{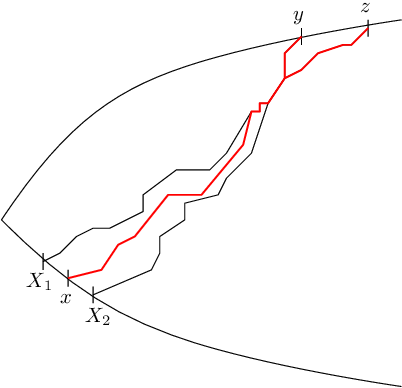}
	\caption{The idea of the proof of Lemma \ref{L:sheet-lem}. We start with points $\bar X_1$ and $\bar X_2$ on either side of a rational $\bar x$ whose melon last passage paths to $\hat y$ and $\hat z$ are not disjoint. These last passage paths both follow particular parabolas, so the place where they meet has to be in the part of the melon that converges to the Airy line ensemble. The last passage paths from $\bar x$ to $\hat y$ and $\hat z$ are squeezed by these outer paths, and hence the branch point between the paths from $\bar x$ must also occur in the part of the melon that converges to the Airy line ensemble.}
	\label{fig:lemma85}
\end{figure}

\medskip

To prove Lemma \ref{L:sheet-lem}, it will be convenient to translate condition 3 into a statement about limits of last passage paths. Fix $(x, y) \in \Q^+ \X \Q$. The function $k\mapsto \hat Z^n_k(x,y)$ records the time when the rightmost last passage path  $\pi\{\bar x,\hat y\}_n$ jumps from (possibly below) $k+1$ to (possibly above) $k$. Thus the function $k \mapsto \hat Z^n_k(x,y)$ acts as the inverse of the nonincreasing function $\pi\{\bar x,\hat y\}_n$. Now, for every $k$, $Z^n_k(x,y) \to Z_k(x,y)$, and $Z_k(x, y) \to -\infty$ as $k\to\infty$. It follows deterministically that there is a nonincreasing cadlag integer-valued function $\pi \langle x,y\rangle:(-\infty, y] \to \N$ so that
$$
\pi\{\bar x,\hat y\}_n(\hat t)\to \pi \langle x,y\rangle(t)
$$
for every $t$ not in the set of jump times $\{Z_k(x, y) : k \in \N\}$.
Moreover, even at these jump times we have
\begin{equation}
\label{E:pi-tt}
\pi \langle x,y\rangle(t_-)\ge \pi\{\bar x,\hat y\}_n(\hat t_-)\ge \pi\{\bar x,\hat y\}_n(\hat t)\ge \pi \langle x,y\rangle(t)
\end{equation}
for all large enough $n$.
This and \eqref{E:asym-Z} imply that
\begin{equation}\label{E:asym-pi}
\lim_{t\to -\infty} \frac{\pi \langle x,y\rangle (t)}{t^2}=2x.
\end{equation}

\begin{proof}[Proof of Lemma \ref{L:sheet-lem}] The proof proceeds deterministically for a fixed $\om \in \Om$.
We have $\scrS(0, y) = \scrA(y)$ for all $y \in \R$ because $\scrS_n(0, y) = A_n(y)$.
Now fix $(x, y, z) \in \Q^+ \X \Q^2$ with $y < z$, and let $X_1 < x < X_2$ be as in property 4 above.
By property 3 and \eqref{E:asym-pi} there exists a random $T<y$ so that for all $t<T$ we have
\begin{equation}\label{E:close}
\pi\langle X_1,y\rangle(t_-) < \pi \langle x,y\rangle(t) \qquad \mathand \qquad
\pi\langle X_1,y\rangle(t_-) < 2xt^2 < \pi\langle X_2,z\rangle(t).
\end{equation}
Let $k$ be arbitrary but large enough so that  $-\sqrt{k/(2x)}<T$. Let $t =-\sqrt{k/(2x)}$ and let $N$ be a large enough random integer so that \eqref{E:pi-tt} holds at the point $t$ for the three
paths $\pi\langle X_1,y\rangle, \pi \langle x,y\rangle,$ and $\pi\langle X_2,z\rangle$. Further, by property 4, we can require that $N$ is large enough so that for all $n\ge N$ there is a point lying along both $\pi\{\bar{X}_1, \hat{y}\}_n$ and  $\pi\{\bar{X}_2, \hat{z}\}_n$. We call this point $(R,W)$; it may depend on $n$. By monotonicity of last passage paths, $(R,W)$ must also lie along $\pi\{\bar{x}, \hat{y}\}_n$ and $\pi\{\bar{x}, \hat{z}\}_n$. Our next goal is to show that $(R, W)$ also lies along the melon last passage paths $\pi\{(\hat t,k),\hat y\}_n$ and $\pi\{(\hat t,k),\hat z\}_n$, which go from $(\hat t,k)$ to $(\hat y, 1)$ and from $(\hat t,k)$ to $(\hat z, 1)$.

\medskip

To this end, observe that \eqref{E:pi-tt} and \eqref{E:close} imply that
 \begin{equation}
\label{E:mono-mono}
\pi\{\bar{X}_1, \hat{y}\}_n(\hat t) \ne \pi\{\bar{x}, \hat{y}\}_n(
\hat t) \qquad \mathand \qquad \pi\{\bar{X}_1, \hat{y}\}_n(\hat t) < k < \pi\{\bar{X}_2, \hat{z}\}_n(\hat t)
\end{equation}
for $n \ge N$. Now, the set $I=\{s \in \R :\pi\{\bar{X}_1, \hat{y}\}_n(s)=\pi\{\bar{x}, \hat{y}\}_n(s)\}$ is an interval by Proposition \ref{P:tree-structure}. Since $R,\hat y\in I$ and $\hat t < \hat y$, the first condition in \eqref{E:mono-mono} guarantees that $\hat t<R$. The second condition in \eqref{E:mono-mono} and monotonicity of last passage paths guarantees that $\pi\{(\hat t,k),\hat y\}_n$ and $\pi\{(\hat t,k),\hat z\}_n$ are sandwiched between $\pi\{\bar{X}_1, \hat{y}\}_n$ and $\pi\{\bar{X}_2, \hat{z}\}_n$ restricted to the interval $[\hat t, \hat y]$. Since $R \in [\hat t, \hat y]$, the point $(R,W)$ must lie along these paths as well.

\medskip

Next, since metric composition holds at any point lying along  a last passage path, we have
\begin{align*}
\{\bar x\to \hat y\} &=  \{\bar x\to (R,W)\} + \{ (R,W)\to \hat y\},\\
\{\bar x\to \hat z\} &=  \{\bar x\to (R,W)\} + \{ (R,W)\to \hat z\},\\
\{(\hat t, k)\to \hat y\} &=  \{(\hat t, k)\to (R,W)\} + \{ (R,W)\to \hat y\},\\
\{(\hat t, k)\to \hat z\} &=  \{(\hat t, k)\to (R,W)\} + \{ (R,W)\to \hat z\}.
\end{align*}
Adding these equations with signs $-,+,+,-$ we get
\begin{align*}
\{\bar x\to \hat z\} -\{\bar x\to \hat y\} =
\{(\hat t,k) \to \hat z\} -\{(\hat t, k) \to \hat y\}.
\end{align*}
By this and the definition of $\scrS_n$ we have
\begin{align}\label{E:Sn}
\scrS_n&(x, z) - \scrS_n(x, y)
= n^{1/6}\lf[\{(\hat t, k)\to \hat z\} - \{(\hat t, k) \to \hat y\}\rg] - 2(z - y)n^{1/3}.
\end{align}
As $n\to \infty$ the uniform convergence of $A_n$ to $\scrA$ on $[t, z]\times \{1,\ldots, k\}$ implies the convergence of last passage values. Taking limits of \eqref{E:Sn} we get
\begin{equation}
\label{E:SS}
\scrS(x, z) - \scrS(x, y) = \langle (t, k) \to z\rangle - \langle(t, k) \to y\rangle.
\end{equation}
Since this holds for any $k$ large enough and  $t=-\sqrt{k/(2x)}$, the proof is complete.
\end{proof}

\begin{remark}
	\label{R:differences}
	One of the crucial ideas in Definition \ref{D:airy-sheet} is to look at \emph{differences} of last passage values, rather than just the last passage values themselves. Hopefully, the reader can now appreciate the importance of this. We showed that differences of last passage values are contained in $\scrA$ via coalescence arguments, which only required that last passage paths starting at distinct points $\bar x \ne \bar x'$ diverge away from the Airy line ensemble corner. Ultimately, this just required that the error estimate on the $F^n_k$ in Proposition \ref{P:longbound} was of lower order than the leading $O(\sqrt{k})$ terms. On the other hand, to prove convergence of last passage values directly, we would have needed a much finer error estimate: $\oo(1)$, rather than $\oo(\sqrt{k})$.
\end{remark}

\section{Properties of the Airy sheet}
\label{S:properties-Airy}

In this section we prove a few basic properties of the Airy sheet.

\begin{lemma}
\label{L:airy-facts} As a random continuous function in $\R^2$, for any $c \in \R$ the Airy sheet satisfies
$$
\scrS(x, y) \eqd \scrS(-y, -x) \quad \mathand \quad \scrS(x, y) \eqd \scrS(x, y + c) + 2c(y-x) + c^2.
$$
Moreover, for $x\le x',\; y\le y' $ we have
$$\scrS(x,y)+\scrS(x',y')\ge \scrS(x,y')+\scrS(x',y).
$$\end{lemma}

\begin{proof}
The first distributional equality in Lemma \ref{L:airy-facts} is inherited from the stationarity of Brownian increments under the map $t \mapsto 1 - t$. The second distributional equality follows from Brownian scaling. Indeed, letting $a = 1 + 2 c
n^{-1/3}$ we have that
\begin{align*}
&B^n[(2xn^{-1/3}, n) \to (1 + 2 y n^{-1/3}, 1)] \\
\eqd \;\; &a^{-1/2} B^n[(2xn^{-1/3} + O(n^{-2/3}), n) \to (1 + 2 (y + c) n^{-1/3} + O(n^{-2/3}), 1)], 
\end{align*}
jointly as functions in $x, y$, where on every compact set $K$, the error terms $O(n^{-2/3})$ are bounded above by $c_K n^{-2/3}$ for some constant $c_K$. Taking the limit of this equation after proper rescaling and centering yields the second distributional equality.
The quadrangle inequality is inherited from Proposition \ref{P:mono-inc}.
\end{proof}

The metric composition law is also inherited from Brownian last passage percolation. For the proof, we have to guarantee that the prelimiting optimal location is tight. Recall from the introduction that an \textbf{Airy sheet of scale $s$} is given by
$$
\scrS_s(x, y) = s \scrS(xs^{-2}, ys^{-2}).
$$
\begin{prop}[Metric composition  law]
\label{P:geod-eqn}
Let $\scrS_s,\, \scrS_t$ be independent Airy sheets of scale $s$ and $t$. For $(x, z) \in \R^2$, define
\begin{equation}
\label{E:stable-st}
Q(x, z) = \sup_{y \in \R} \scrS_s(x,y)+\scrS_t(y,z),
\end{equation}
The function $Q$ is an Airy sheet of scale $r$, where $r^3 = s^3 + t^3$.
Moreover, the largest value $Z^+(x,z)$ where the maximum in \eqref{E:stable-st} is achieved  is nondecreasing in both $x$ and $z$.  The same holds for $Z^-$ defined analogously.
\end{prop}

We have a true maximum, rather than a supremum. To prove Proposition \ref{P:geod-eqn}, we use the following lemma which gives tightness of the maximum location for two prelimiting Airy processes.

\begin{lemma}
\label{L:wm-max}
For $n \in \N$ and $t \in \{1/n, \dots, (n-1)/n\}$, let $k = nt$ and $m = n(1-t)$, and set $s = \min(t, 1-t)$. Let $W^k_1$ and $W^m_1$ be the top lines of two independent Brownian melons with $k$ lines and $m$ lines respectively. Define the melon sum
$$
A(z) = W^k_1(z) + W^m_1(1 - z).
$$
Let $S = \{z : A(z) = \max_{w \in [0, 1]} A(w) \}$. Then there exist constants $c$ and $d$ such that for all $r > 0$, we have
$$
\sup_{n \in \N} \sup_{t \in \{1/n, \dots, (n-1)/n\}} \p(S \not\sset [t - r s^{1/3}n^{-1/3}, t + r s^{1/3}n^{-1/3}]) \le ce^{-dr^{3/4}}.
$$
\end{lemma}

To understand Lemma \ref{L:wm-max}, think of each point $(s, k), s \in S$ as lying along a Brownian last passage path from $(0, n)$ to $(1,1)$. We expect such paths to closely follow the straight line between $(0, n)$ and $(1, 1)$, and hence hit line $k$ around time $t = k/n$. Lemma \ref{L:wm-max} quantifies the fluctuations from this guess.

To prove Lemma \ref{L:wm-max}, we need the following calculation.
\begin{lemma}
\label{L:calculation}
Let $b > 0$ be a fixed constant. Then there exists a constant $c$ such that for all $n \in \N, t \in \{1/n, 2/n, \dots, (n-1)/n\}, r>1$ and
$$
z \in [0, t - c(\min[t,1- t])^{1/3}r^2n^{-1/3}] \cup [t + c(\min[t,1- t])^{1/3}r^2n^{-1/3}, 1],
$$
we have that
\begin{align}
\label{E:main-term}
2&\lf(\sqrt{ntz} + \sqrt{n(1-t)(1 - z)}\rg)  \\
\label{E:err-0}
&+ r\lf(\sqrt{z}(nt)^{-1/6} + \sqrt{1-z}(n(1-t))^{-1/6}\rg) \\
\label{E:err-1}
&+ \sqrt{z}(nt)^{-1/6} b \log^{2/3} \lf((nt)^{1/3} \log\lf(\frac{t}{z} \vee \frac{z}t\rg) + 1\rg) \\
\label{E:err-2}
       &+ \sqrt{1-z}(n(1-t))^{-1/6} b \log^{2/3} \lf((n(1-t))^{1/3} \log\lf(\frac{1-t}{1-z} \vee \frac{1-z}{1-t}\rg) + 1\rg) \\
       \nonumber
       &\qquad \qquad \qquad \le 2\sqrt{n} - rn^{-1/6}.
\end{align}
Here the notation $a \vee b$ means the maximum of $a$ and $b$.
\end{lemma}

We leave the proof of Lemma \ref{L:calculation} to the end of the section.  Instead we proceed with the proof of Lemma \ref{L:wm-max}. Throughout the proof, $c$ and $d$ are universal constants that may change from line to line.
\begin{proof}[Proof of Lemma \ref{L:wm-max}]
Define $L_n = \max_{z \in [0, 1]} A(z).$ The value $W^k_1(z)$ can be thought of as a last passage value across $k$ independent Brownian motions $B_1, \dots, B_k$ in the interval $[0, z]$ and the value $W^m_1(1 - z)$ can be thought of as a last passage value across $m$ independent Brownian motions $B_{k+1}, \dots, B_n$ in the interval $[z, 1]$. In particular, by Lemma \ref{L:metric}, this means that $L_n$ is simply a Brownian last passage value, and so by Theorem \ref{T:top-bd} there exist constants $c, d > 0$ such that for all $n \in \N$ and $r > 0$, we have
\begin{equation}
\label{E:Ln-bd}
\p\lf(L_n \ge 2\sqrt{n} - r n^{-1/6} \rg) \ge 1-ce^{-dr^{3/2}}.
\end{equation}
Now by Proposition \ref{P:cross-prob}, there exist constants $b, c, d > 0$ such that for any $n \in \N$ and $r > 0$, the probability that
\begin{equation}
\label{E:Wn1-2n}
W^n_1(t) \le 2 \sqrt{n}\sqrt{t} + \sqrt{t} n^{-1/6}[r + b\log^{2/3}(n^{1/3} \log(t \vee t^{-1}) + 1)] \;\; \forall t \in [0, \infty)
\end{equation}
is bounded below by $1 - ce^{-dr^{3/2}}.$ The logarithmic error above is chosen to be minimized at $s = 1$. However, using the Brownian scaling $W^n_1(s \; \cdot) \eqd \sqrt{s} W^n_1(\cdot)$, we can get that $W^n_1(s t)$ is bounded above by $\sqrt{s}$ times the left side of \eqref{E:Wn1-2n} for all $t$ with probability at least $1 - ce^{-dr^{3/2}}.$ This minimizes the error for $W^n_1$ at $s$.

In particular, we will bound the sum $A(z) = W^k_1(z) + W^m_1(1 -z)$ by choosing to minimize the error term in $W^k_1(z)$ at $t$ and in $W^m_1(1 - z)$ at $1 - t$. This gives that with probability at least $1- c e^{-d r^{3/2}}$, we have
\begin{equation}
\label{E:Az}
\begin{split}
A(z) \le& \; 2\Big(\sqrt{kz} + \sqrt{m(1 - z)}\Big) + \sqrt{z}k^{-1/6} \lf(r + b \log^{2/3} \lf(k^{1/3} \log\lf(\frac{t}{z} \vee \frac{z}t\rg) + 1\rg)\rg) \\
       &+ \sqrt{1-z}m^{-1/6} \lf(r + b \log^{2/3} \lf(m^{1/3} \log\lf(\frac{1-t}{1-z} \vee \frac{1-z}{1-t}\rg) + 1\rg)\rg).
\end{split}
\end{equation}
Now, by Lemma \ref{L:calculation}, there is a constant $c$ such that for all $r > 1, n \in \N$, and $t \in \{1/n, \dots, (n-1)/n\}$, the right hand side above is bounded by $2\sqrt{n} - rn^{-1/6}$ for all
$$
z \in [0, t - c r^2 s^{1/3} n^{-1/3}] \cup [t + cr^2 s^{1/3} n^{-1/3}, 1].
$$
Combining the bound on the probability of the event in \eqref{E:Az} with the bound on $L_n$ in \eqref{E:Ln-bd} implies the lemma.
\end{proof}

\begin{proof}[Proof of Proposition \ref{P:geod-eqn}]
By rescaling, we can assume that $r = 1$. We set up a Brownian last passage percolation converging to an Airy sheet $\scrS$ as in the previous section. Then last passage across the first $s^3$ proportion of the Brownian motions
and last passage through the second $t^3$ proportion of the Brownian motions converge jointly in distribution to independent Airy sheets $\scrS_s,\scrS_t$. Set $\scrS^n, \scrS^n_t,$ and $\scrS^n_s$ to be the corresponding prelimiting Airy sheets, so that 
$$
\scrS^n(x, z) = \max_{w} \scrS^n_s(x, w) + \scrS^n_t(w, z)
$$
for all $x, z$ by Lemma \ref{L:metric}. To prove that the $Q$ in Proposition \ref{P:geod-eqn} is an Airy sheet, it is then enough to show that the right side above converges to $\max_{w} \scrS_s(x, w) + \scrS_t(w, z)$ in the uniform-on-compact topology on continuous functions from $\R^2$ to $\R$. For this, set
$$
M_n(x, z) = \{ y  : \max_{w} \scrS_s^n(x, w) + \scrS^n_t(w, z) = \scrS_s^n(x, y) + \scrS^n_t(y, z)\}.
$$
It is enough to show that for any compact set $K$, the random variables 
\begin{equation}
\label{E:IK}
I_n(K) := \inf_{(x, z) \in K} \inf M_n(x, z), \qquad S_n(K) :=\sup_{(x, z) \in K} \sup M_n(x, z)
\end{equation}
 are tight. When $K$ consists of a single point, this follows from Lemma \ref{L:wm-max}. For a general compact set $K$, note that points $y \in M_n(x, z)$ lie along last passage paths. In particular, monotonicity of last passage paths, Lemma \ref{L:mono-path}, guarantees that both $\inf M_n(x, z)$ and $\sup M_n(x, z)$ are increasing functions in both coordinates. Therefore for any compact set $K$, there are points $u, v\in \R^2$ such that $I_n(u) \le I_n(K) \le S_n(K) \le S_n(v)$ for all $n$, so tightness for general $K$ follows from tightness for single points. Finally, the fact that $Z^+$ and $Z^-$ are monotone in both coordinates follows from monotonicity in the finite case.
\end{proof}

The monotonicity of $Z^+$ and $Z^-$ can also be proved without any reference to the prelimiting Airy sheets by using the inequality in Lemma \ref{L:airy-facts}.

\medskip

Our next goal is to give a near optimal bound on the location of the maximum in Proposition \ref{P:geod-eqn}. This can be stated in terms of the parabolic Airy process $\scrA=\scrA_1$, the first line of the parabolic Airy line ensemble.
We define the {\bf parabolic Airy process of scale $\sigma$}
from the standard one by setting $\scrA_\sig(\cdot)=\sigma \scrA(\cdot/\sigma^2)$.  This is analogous to the scaling of Airy sheets. The GUE Tracy-Widom law of these processes is scaled by $\sigma$, and the Brownian component of an Airy process of any scale has variance $2$. The stationary version is
\begin{equation}\label{E:st-airy}
\scrA_\sigma(z) + \sigma^{-3}z^2.
\end{equation}
At small scales the parabola becomes more dominant.

\begin{lemma}
\label{L:airy-max}
Fix $0 < s \le t$, and let $\scrA_s$, $\scrA_t$ be two independent Airy processes of scale $s$ and $t$. Then
$\scrA = \scrA_s + \scrA_t
$
has a unique maximum $A$ at some location $S$ almost surely.
Moreover, for any $m > 0$, we have
$$
\p(S \not\in [-ms^2, ms^2]) \le ce^{-dm^3} \;\;  \mathand \;\; \p(A - \scrA(0) > ms) \le ce^{-dm^{3/2}}.
$$
Here $c, d > 0$ are constants that are independent of $s, t$.
\end{lemma}
In this proof, $c$ and $d$ are constants that may change from line to line.

\begin{proof}
By rescaling, we may also assume that $s + t = 1$. Consider the stationary versions $\scrR_s, \scrR_t$ of the Airy processes $\scrA_s, \scrA_t$, defined as in \eqref{E:st-airy}. Let $\scrR = \scrR_s + \scrR_t$. When $t \ge (28)^{1/3} s$, Lemma \ref{L:airy-tails} and rescaling implies that for all $z > 0$, we have
\begin{equation}
\label{E:B-eqn}
\p\lf(|\scrR_t(z) - \scrR_t(0)| > s^{-3}z^2/4\rg) \le ce^{-dz^3/s^6}.
\end{equation}
When $t \in [s,  (28)^{1/3} s]$, we can use the Tracy-Widom one-point bound on $t\scrR_t(z)$ and $t\scrR_t(0)$ (Theorem \ref{T:TW-airy}) to get the same bound. Using this same one-point bound on $s\scrR_s$ we also get
\begin{equation*}
\p\lf(|\scrR_s(z) - \scrR_s(0)| > s^{-3}z^2/4\rg) \le ce^{-dz^3/s^6}.
\end{equation*}
Combining these bounds gives
\begin{equation*}
\p\lf(\scrR(zs^2) > \scrR(0) + sz^2/2\rg) \le ce^{-dz^3},
\end{equation*}
which after translating back to $\scrA$ gives that
\begin{equation}
\label{E:on-Z}
\p\lf(\scrA(zs^2) > \scrA(0) - sz^2/2\rg) \le ce^{-dz^3},
\end{equation}
Now let $a > 0$. By Lemma \ref{L:airy-tails}, for any $z \in \R$ and $y \in [0, as^2]$, with $\scrR=\scrR_s+\scrR_t$ we have that
\begin{equation}
\label{E:B-my}
\p(|\scrR(z+y) - \scrR(z)| > \ell \sqrt{y}) \le ce^{d(49a^3 -\ell^2)}
\end{equation}
for $\ell > 0$.
By Lemma \ref{L:levy-est}, the process $\scrR$ satisfies the modulus of continuity bound \begin{equation}
\label{E:B-cb}
|\scrR(z+y) - \scrR(z)| \le C_{a, b} \sqrt{y} \log^{1/2}(2as^{2}/y)
\end{equation}
on any interval $[b, b + as^2]$ where $C_{a, b}$ is a random constant that satisfies
\begin{equation}
\nonumber
\p(C_{a, b} > \ell + c a^{3/2}) \le ce^{-d\ell^2}.
\end{equation}
In particular, we can use \eqref{E:B-cb} and a union bound over intervals of the form $[b, b + s^2]$ to get that there exists a constant $C$ satisfying the same tails as each of the constants $C_{1, b}$ (with possibly different $c$ and $d$) such that
$$
\max_{z \in [-m, m]} \scrR(zs^2) - \scrR(z_*s^2) \le C s \log^{1/2} m,
$$
where $z_*=\lfloor z \rfloor$ for $z\ge 0$ and $\lceil z \rceil$ otherwise. Here we have approximated $z$ with an integer closer to $0$ so that the same bound holds for the nonstationary process $\scrA$: this way the parabolic decay does not affect the bound. In particular, this bound combined with the bound \eqref{E:on-Z} applied to integers gives that
$$
\p(\scrA(z) > \scrA(0) \text{ for some } |z| > ms^{2}) \le ce^{-dm^3}.
$$
This proves the bound on $S$ in the lemma. The processes $\scrA_s(z)-\scrA_s(0),\scrA_t(z)-\scrA_t(0)$ are absolutely continuous with respect to Brownian motion of variance $2$ on a compact interval by Proposition \ref{P:brownian-airy}, and hence so is $\scrA(z) - \scrA(0)$. Hence it has a unique maximum on any compact interval almost surely. By the above bound on $S$ this implies that $\scrA$ has a unique maximum almost surely.

\medskip

Now we can use \eqref{E:B-cb} again to get that for any $m > 0$, we have
$$
\max_{z \in [-m, m]} \scrA(zs^{2}) - \scrA(0) \le (C + cm^{3/2})s\sqrt{m}
$$
for a constant $C$ satisfying
$$
\p(C > \ell) \le ce^{-d\ell^2}
$$
for $\ell > 0$.
We get the desired bound on $A -\scrA(0)$ by writing
\[
\p(A -\scrA(0) > ms) \le \p(|S| \ge \sqrt{m/3}s^2) + \p\lf(\max_{|z| \le \sqrt{m/3}} \scrA(zs^2) - \scrA(0) > ms\rg).
\qedhere \]
\end{proof}

Now we prove Lemma \ref{L:calculation}.
The proof of this lemma is essentially Taylor expansion. However, keeping track of what is happening with all the terms gets rather complex, and so we have included a proof here. Throughout the proof, $c_0, c_1,$ and $c_2$ are positive constants that may change from line to line and may depend on the constant $b$ in the statement of the lemma.
\begin{proof}
By symmetry, it is enough to prove the lemma for $t \le 1/2$. We will also write $z = t + \al$ where $\al \in [-t, 1 - t]$.

\medskip
\noindent \textbf{Step 1: Bounding the main term \eqref{E:main-term}.} \qquad We have the bounds
\begin{equation}
\label{E:sqrt-taylor}
\sqrt{1 + x} \le 1 + \frac{x}2 - \frac{x^2}{12}, \quad x \in [-1, 1] \quad  \mathand \quad \sqrt{1 + x} \le \sqrt{2} + \frac{x - 1}{2\sqrt{2}}, \quad x \ge 1.
\end{equation}
Using the first bound above, we can bound the term \eqref{E:main-term} on the interval $\al \in [-t, t]$:
\begin{align}
\nonumber
2\lf(\sqrt{ntz} + \sqrt{n(1-t)(1 - z)}\rg) &= 2\sqrt{n}\lf(t\sqrt{1 + \al/t} + (1-t)\sqrt{1 - \al/(1-t)} \rg) \\
\label{E:small-bd}
&\le 2\sqrt{n} \lf(1 - \frac{\al^2}{12t}\rg), \qquad \al \in [-t, t].
\end{align}
We can similarly bound \eqref{E:main-term} on the interval $\al \in [t, 1-t]$ by using the first bound in \eqref{E:sqrt-taylor} on the $t$ term and the second bound on the $1-t$ term. This gives
\begin{equation}
\label{E:large-bd}
2\lf(\sqrt{ntz} + \sqrt{n(1-t)(1 - z)}\rg) \le 2\sqrt{n}\lf(1 - c_0\al \rg), \qquad \al \in [t, 1 -t].
\end{equation}
\medskip
\noindent \textbf{Step 2: Bounding the error terms \eqref{E:err-0} + \eqref{E:err-1} + \eqref{E:err-2}.} \qquad
 We first consider the case $\al \in [-t, 0]$. Since $t \le 1/2$ and $z = t + \al \le t$, \eqref{E:err-0} is bounded by $3rn^{-1/6}$. Now by using that $\log (x) \le x - 1$ and that $\sqrt{t + \al} \le \sqrt{t}$, we can bound \eqref{E:err-1} above by
\begin{align}
\nonumber
n^{-1/6}&\lf( (t+ \al)^{1/3} b \log^{2/3} \lf((nt)^{1/3} \frac{|\al|}{t + \al} + 1\rg)\rg) \\
\nonumber
&= c_1 n^{-1/6} \lf((t+ \al)^{1/3} \log^{2/3} \lf((nt)^{1/3} |\al| + t + \al \rg) - (t+\al)^{1/3} \log^{2/3} (t + \al) \rg) \\
\label{E:err-1-bd}
&\le c_1 n^{-1/6} \lf( \log (n^{1/3} |\al| + 1) + 1\rg).
\end{align}
We can also use the bound $\log (x) \le x - 1$ and the fact that $1 - t$ and $1 - z$ are bounded away from zero to get that for $\al \in [-t, 0]$, that \eqref{E:err-2} is bounded above by
\begin{align*}
\sqrt{1-z}(n(1-t))^{-1/6} b \log^{2/3} \lf((n(1-t))^{1/3} \frac{|\al|}{1 - t} + 1\rg) \le c_1 n^{-1/6} \lf( \log \lf(|\al| n^{1/3} + 1 \rg) + 1 \rg).
\end{align*}
Combining this with the bound \eqref{E:err-1-bd} and the bound of $3rn^{-1/6}$ on \eqref{E:err-0} for $\al \in [-t, 0]$ gives that
\begin{align}
\nonumber
\text{\eqref{E:err-0}+\eqref{E:err-1}+ \eqref{E:err-2}} &\le c_1n^{-1/6} \lf( \log (n^{1/3} |\al| + 1) + 1 + 3r\rg) \\
\label{E:combined-small-bd}
&\le c_1n^{-1/6} \lf( \log (n^{1/3} |\al| + 1) + r\rg) \qquad \al \in [-t, 0].
\end{align}
We have folded the constant term into the $r$ term by using that $r \ge 1$. Now when $\al \in [0, 1-t]$, we can bound \eqref{E:err-0} by
\begin{equation}
\label{E:err-0-large-bd}
rcn^{-1/6}\lf(1 + t^{-1/6}\sqrt{\al}\rg).
\end{equation}
Again using that $\log (x) \le x - 1$, we can bound the term \eqref{E:err-1} for $\al \in [0, 1-t]$ by
\begin{align}
\nonumber
\sqrt{t+\al}(nt)^{-1/6} b \log^{2/3}& \lf((nt)^{1/3} \frac{\al}t + 1\rg) \\
\nonumber
&\le \lf(\sqrt{t} + \sqrt{\al}\rg)t^{-1/6}n^{-1/6} b \lf(\log^{2/3}\lf(n^{1/3}\al + t^{2/3}\rg) + \log^{2/3}(t^{-2/3}) \rg) \\
\label{E:large-bd'}
&\le c_1 n^{-1/6} \lf(1 + \sqrt{\al} t^{-1/6} \log(t^{-2/3}) + (\sqrt{\al} t^{-1/6} + 1)\log(n^{1/3} \al + 1)\rg).
\end{align}
We can also bound \eqref{E:err-2} for $\al \in [0, 1-t]$ by
\begin{align*}
 \sqrt{1-t-\al}(n(1-t))^{-1/6} b &\lf(\log^{2/3} \lf((n(1-t))^{1/3} \al + 1 - t - \al \rg) - \log^{2/3}(1 - t - \al) \rg) \\&\le
  c_1 n^{-1/6} \lf( 1 + \log(n^{1/3}\al + 1)\rg).
\end{align*}
Combining this with the bounds \eqref{E:err-0-large-bd} and \eqref{E:large-bd'} implies that
\begin{align}
\label{E:combined-l-bd}
\text{\eqref{E:err-0} + \eqref{E:err-1} + \eqref{E:err-2}} \le c_1n^{-1/6} \lf((\sqrt{\al} t^{-1/6} +1) \log (n^{1/3} \al + 1) + r + r\sqrt{\al} t^{-1/6} \log(t^{-2/3})\rg)
\end{align}
for $\al \in [0, 1-t]$. Note that on the interval $[0, t]$, this reduces to the bound \eqref{E:combined-small-bd}.
\medskip

\noindent \textbf{Step 3: The case $\al \in [-t, t]$.} \qquad By combining the inequalities \eqref{E:small-bd}, \eqref{E:combined-small-bd}, and \eqref{E:combined-l-bd}, we get that the inequality in the lemma holds whenever
\begin{align*}
c_1 \frac{\al^2 n^{2/3}}t - c_2 \log (n^{1/3} |\al| + 1) > c_3 r.
\end{align*}
Making the substitution $\al = \beta r^2 t^{1/3} n^{-1/3}$ and using that $t < 1$, the above statement is implied by the inequality
$$
c_1 r^4 \beta^2 - c_2 \log (|\beta| r + 1) > c_3 r.
$$
This holds for all $r > 1$ as long as $\beta$ is large enough.
\medskip

\noindent \textbf{Step 4: The case $\al \in [t, 1-t]$.} \qquad By combining the inequalities \eqref{E:large-bd} and \eqref{E:combined-l-bd} we get that the inequality in the lemma holds whenever
\begin{equation*}
c_0 n^{2/3} \al - c_1 \lf( (\sqrt{\al} t^{-1/6} +1) \log (n^{1/3} \al + 1) + r + r\sqrt{\al} t^{-1/6} \log(t^{-2/3})\rg) > 0
\end{equation*}
Again making the substitution $\al = \beta r^2 t^{1/3} n^{-1/3}$ the left hand side above is equal to
\begin{equation*}
 c_0 n^{1/3} \beta r^2 t^{1/3} - c_1 \lf( (1 + \sqrt{\beta} r n^{-1/6})\log (\beta r^2 t^{1/3} + 1) + r(1 + \sqrt{\beta} r n^{-1/6} \log(t^{-2/3}))\rg).
\end{equation*}
Since $t \ge n^{-1}$, this is bounded below by
$$
c_0 \beta r^2 - c_1 \lf( (1 + r\sqrt{\beta})\log (\beta r +1) + r(1 + \sqrt{\beta} r)\rg).
$$
For all large enough $\beta$, this is strictly greater than $0$ for all $r>1$.
\end{proof}

\section{The directed landscape}
\label{S:directed-landscape} The goal of this section is to construct the directed landscape, a scale-invariant, stationary, independent increment process with respect to metric composition. The following definition is based on analogies from last passage percolation. Recall that
$$
\R^4_\uparrow = \{(x, t; y, s) \in \R^4 : t < s\}.
$$
Let $C(\R^4_\uparrow, \R)$ be the space of continuous functions from $\R^4_\uparrow$ to $\R$ with the uniform-on-compact topology.
\label{S:landscape}
\begin{definition}
\label{D:directed-landscape}
A \textbf{directed landscape} is a random function $\mathcal{L}$ taking values in the space $C(\R^4_\uparrow, \R)$ that satisfies the following properties.
 \begin{enumerate}[label=\Roman*.]
 \item (Airy sheet marginals) For any $t\in \R$ and $s>0$ the increment over time interval $[t,t+s^3)$,
$$
(x,y) \mapsto \mathcal{\scrL}(x, t; y,t+s^3)
$$
is an Airy sheet of scale $s$.
\item (Independent increments) For any disjoint time intervals $\{(t_i, s_i) : i \in \{1, \dots k\}\}$, the random functions
$$
\scrL(\cdot, t_i ; \cdot, s_i), \quad  i \in \{1, \dots, k \}
$$
are independent.
\item (Metric composition law) Almost surely, for any $r<s<t$ and $x, y \in \R$ we have that
$$
\scrL(x,r;y,t)=\max_{z \in \mathbb R} [\scrL(x,r;z,s)+\scrL(z,s;y,t)] \quad \text{ for all } x, y \in \R.
$$

 \end{enumerate}
 \end{definition}

 We note here that the metric composition law implies a reverse \textbf{triangle inequality} for the directed landscape. For any $r<s<t$ and $x,y, z \in \R$ we have that
\begin{equation}
\label{E:triangle-ineqL}
  \scrL(x,r;y,t) \ge \scrL(x,r;z,s)+\scrL(z,s;y,t).
\end{equation}

Condition III could be weakened so that the metric composition law only holds at every fixed quintuple $(r,s,t, x, y)$ almost surely, rather than at all quintuples simultaneously.

\medskip

We need to show that an object satisfying the above properties exists and is unique. Before constructing the directed landscape, we note that if such an object $\scrL$ exists and is unique, then it must have the following symmetries.

\begin{lemma}
\label{L:land-props}
Let $\scrL$ denote the directed landscape. Assuming that $\scrL$ exists and is unique, we have the following equalities in distribution as functions in $C(\R^4_\uparrow, \R)$. Here $r, c \in \R$, and $q > 0$.
\begin{enumerate}

\item (Time stationarity)
$$
\scrL(x, t ; y, t + s) \eqd \scrL(x, t + r ; y, t + s + r).
$$
\item (Spatial stationarity)
$$
\scrL(x, t ; y, t + s) \eqd \scrL(x + c, t; y + c, t + s).
$$
\item (Flip symmetry)
$$
\scrL(x, t ; y, t + s) \eqd \scrL(-y, -s-t; -x, -t).
$$
\item (Skew stationarity)
$$
\scrL(x, t ; y, t + s) \eqd \scrL(x + ct, t; y + ct + sc, t + s) + s^{-1}[(x - y - sc)^2 - (x - y)^2].
$$
\item (Rescaling)
$$
\scrL(x, t ; y, t + s) \eqd  q \scrL(q^{-2} x, q^{-3}t; q^{-2} y, q^{-3}(t + s)).
$$
\end{enumerate}

\end{lemma}

\begin{proof}
All of these statements can be checked by straightforward computation. The only tools needed are symmetries and rescaling properties of the Airy sheet, along with the definition of the directed landscape.
\end{proof}

We will also work with the process
\begin{equation}
\label{E:stat-landscape}
\scrL(x, t; y, t+ s) + \frac{(x-y)^2}{s}.
\end{equation}
which satisfies the skew-stationarity statement in Lemma \ref{L:land-props} without the parabolic correction term. We first construct the directed landscape on an appropriate dense subset of $\R^4_\uparrow$. Define the sets
$$
S_k = \{ (x, t; y, s) \in \mathbb R_\uparrow^4 : t < s, 2^kt, 2^ks \in \mathbb Z\}.
$$
and let $S=\bigcup_k S_k$.

\begin{lemma}
\label{L:landscape-S}
There exist a random function $\scrL: S \to \mathbb R$ satisfying conditions I-III in the definition of the directed landscape, suitably modified so that all times are dyadic rationals. Its law is unique.
\end{lemma}

\begin{proof} Any process $\scrL:S \to \mathbb R$ can be viewed as a random function from the countable set
$$
D = \{s < t : 2^ks, 2^kt \in \mathbb Z \text{ for some $k \in \mathbb N$}\}
$$
to the space of functions from $\mathbb R^2 \to \mathbb R$. As conditions I-III above determine the joint distribution of $\scrL(d_1), \dots, \scrL(d_k)$ for any $d_1, \dots, d_k \in D$, the distribution of a process satisfying these conditions is unique. We now show that $\scrL$ exists.

 \medskip

First we define $\scrL$ for parameters in the set $S_k$. For this, pick independent Airy sheets $B_i$ for each $i \in \Z$. For $t=i/2^k$ and $s^3=1/2^k$ we define
$$
\scrL_k(x,t;y,t+s^3) = sB_i(x/s^2, y/s^2).
$$
Moreover, for $j=i+\ell$, $\ell \ge 2$ we define
$$
\scrL_k(x_0,i/2^k;x_{\ell}, j/2^k)=\max_{(x_1,\ldots, x_{\ell-1}) \in \mathbb R^{\ell-1}} \sum_{q=1}^\ell \scrL_k\left(x_{q-1},\frac{i+q-1}{2^k} ; x_q,\frac{i+q}{2^k} \right).
$$
By the metric composition law for Airy sheets, Proposition \ref{P:geod-eqn}, the following holds. First, $\scrL_k$ satisfies properties I-III on the set $S_k$. Also, for $k<k'$ the process $\scrL_{k'}$ restricted to $S_k$ has the same law as $\scrL_k$. Kolmogorov's extension theorem then provides a limiting process $\scrL$ on $S=\bigcup_k S_k$, which when restricted to $S_k$ has the same law as $\scrL_k$ for all $k$. Therefore the process $\scrL$ satisfies properties I-III.
\end{proof}

In order to show that $\scrL$ has a unique continuous extension to $\mathbb R^4_\uparrow$ it suffices to check that  $\scrL$ is uniformly continuous on a set of probability $1$ on $K\cap S \cap \Q^4$ for any compact set $K\subset \mathbb R^4_\uparrow$. This follows by proving an explicit tail bound on two-point differences for $\scrL$.

\begin{lemma}
	\label{L:mod-sk}  
	Let $\scrL$ be any random function defined on the dyadic set $S$ that satisfies conditions I-III in the definition of the directed landscape. Let $\scrK$ be the stationary version of $\scrL$ as defined in \eqref{E:stat-landscape}. Let $u_i = (x_i, t_i; y_i, s_i)$ for $i = 1, 2$ be points in $S$. Define $\hat u_2 = (\hat x_2, t_1; \hat y_2, s_1) \in S$ where the points $\hat x_2, \hat y_2$ are chosen to so that $(\hat x_2, t_1)$ and $(\hat y_2, s_1)$ lie on the line containing the two points 
	$
(x_2, t_2), (y_2, s_2).
	$ 
Define
	$$
	\chi = ||(x_1, y_1) - (\hat x_2, \hat y_2)||, \;\; \tau = ||(t_1, s_1) - (t_2, s_2)||.
	$$
	Then
	$$
	\p\lf(|\scrK(u_1) - \scrK(u_2)|  \ge m\tau^{1/3} + \ell \chi^{1/2}\rg) \le ce^{-dm^{3/2}} + ce^{-d\ell^2}.
	$$
	for universal constants $c$ and $d$.
\end{lemma}

In the statement of Lemma \ref{L:mod-sk}, using $\chi$ rather than the spatial difference $\xi = ||(x_1, y_1) - (x_2, y_2)||$ may seem a bit strange. The best way to see why this is the right choice is to observe that $\chi$ is invariant under the skew transformations in Lemma \ref{L:land-props}, whereas $\xi$ is not.

Throughout the proof, $c$ and $d$ are constants that may change from line to line.
\begin{proof}
	First, $\scrL$ and $\scrK$ satisfy all symmetries of Lemma \ref{L:land-props} for symmetries that preserve the set $S$, so we may use these to simplify the proof. Observe that the time stationarity, spatial stationarity and skew stationarity transformations in Lemma \ref{L:land-props} will not affect the quantities $\chi, \tau$. Because of this, we can take $t_2 = x_2 = 0$ (by time and spatial stationarity), and then $y_2 = 0$ (by skew stationarity). This implies that $\hat x_2 = \hat y_2 = 0$ as well. 
	
	To get the bound in this case, we will first change spatial coordinates and then change time coordinates. By the tail bound on two-point distributions of the Airy process, Lemma \ref{L:airy-tails}, applied twice, we have the bound
	$$
	\p(|\scrK(u_1) - \scrK(\hat u_2)| > \ell\chi^{1/2}) \le ce^{-d\ell^2}.
	$$
	This uses the fact that $\scrK$ has rescaled stationary Airy sheet marginals. We now bound the difference when we change time. By the triangle inequality on $\mathbb R$, it suffices to show that both
	\begin{align*}
	\p(|\scrK(\hat u_2)& - \scrK(0, t_1; 0, s_2)| > m\tau^{1/3}) \;\; \mathand \\
	\p(|\scrK(u_2)& - \scrK(0, t_1; 0, s_2)| > m\tau^{1/3})
	\end{align*}
	are bounded above by $ce^{-dm^{3/2}}$. We will just bound the first term, as the second term can be bounded similarly. For this, by symmetry we may assume $s_2 > s_1$. By the metric composition law for $\scrL$, we have that
	\begin{align}
	\label{E:L-to-bd}
	\scrL(0, t_1; 0, s_1)& + \scrL(0, s_1; 0, s_2) \le \scrL(0, t_1; 0, s_2) = \max_{z \in \R} [\scrL(0, t_1; z, s_1) + \scrL(z, s_1; 0, s_2)].
	\end{align}
	By Lemma \ref{L:airy-max}, we have that
	\begin{equation}
	\label{E:L-above}
	\max_{z \in \R} [\scrL(0, t_1; z, s_1) + \scrL(z, s_1; 0, s_2)] \le m|s_1 - s_2|^{1/3} + \scrL(0, t_1; 0, s_1) + \scrL(0, s_1; 0, s_2)
	\end{equation}
	with probability at least $1 - ce^{-dm^{3/2}}$. The random variable $\scrL(0, s_1; 0, s_2)$ has rescaled GUE Tracy-Widom distribution. By Theorem \ref{T:TW-airy} it satisfies
	\begin{equation*}
	\p(|\scrL(0, s_1; 0, s_2)| > m |s_1 - s_2|^{1/3}) \le ce^{-dm^{3/2}}.
	\end{equation*}
	Together with \eqref{E:L-to-bd} and \eqref{E:L-above}, this gives that
	\[
	\p(|\scrK(\hat u_2) - \scrK(0, t_1; 0, s_2)| > m\tau^{1/3}) = \p(|\scrL(0, t_1; 0, s_1) - \scrL(0, t_1; 0, s_2)| > m\tau^{1/3}) \le ce^{-dm^{3/2}}.\qedhere
	\]
\end{proof}

The process $\scrL$ constructed in Lemma \ref{L:landscape-S} thus has a unique continuous extension to $\R^4_\uparrow.$ We cannot quite conclude that it satisfies Definition \ref{D:directed-landscape} since we need to check that the metric composition law extends off of $S$. To prove this, we first prove various uniform bounds on $\scrL$. These will also be used later when constructing directed geodesics.

\medskip

The first result gives an essentially optimal modulus of continuity, up to constant factors. For this, we define
$$
K^\ep_b = [-b, b]^4 \cap \{(x, t; y, t + s) \in \R^4_\uparrow : s \ge \ep \}.
$$
\begin{prop}
\label{P:mod-land}
Let
$$
\scrK(x,t;y,t+s)=\scrL(x,t;y,t+s) + \frac{(x-y)^2}{s}
$$
denote the stationary version of the directed landscape. For two points $u_i = (x_i, t_i; y_i, t_i + s_i)$, $i = 1, 2$, let
$$
\xi = \xi(u_1, u_2) = ||(x_1, y_1) - (x_2, y_2)||, \;\; \tau = \tau(u_1, u_2) = ||(t_1, t_1 + s_1) - (t_2, t_2 + s_2)||.
$$
Then for any $b \ge 2, \ep \le 1$ and $u_1, u_2 \in K^\ep_b$ with $\tau \le \ep^3/b^3$, we have that
$$
|\scrK(u_1) - \scrK(u_2)| \le C \lf(\tau^{1/3}\log^{2/3}(\tau^{-1}) + \xi^{1/2}\log^{1/2}(4b\xi^{-1}) \rg),
$$
with a random constant $C$  satisfying
$
\p(C > m) \le c b^{10} \ep^{-6} e^{-dm^{3/2}},
$
where $c,d$ are universal constants.
\end{prop}

In the proof, the constants $c, d$ may change from line to line.

\begin{proof}
Lemma \ref{L:mod-sk} implies that for $u_1, u_2 \in K^\ep_{2b}$ with $s_1 = s_2, t_1 = t_2$, we have
\begin{equation}
\label{E:Ku}
\p(|\scrK(u_1) - \scrK(u_2)|  \ge \ell \xi^{1/2}) \le ce^{-d\ell^2}.
\end{equation}
Now let $u_1, u_2 \in K^\ep_{2b}$ with $x_1 = x_2, y_1 = y_2$ and $\tau \le \ep^3/b^3$. With $\chi$ as in Lemma \ref{L:mod-sk}, we have that 
$
\chi \le (4b /\ep) \tau \le 4 \tau^{2/3},
$
and hence for any $m > 0$,
$$
m \tau^{1/3} \ge \frac{m}{3} \tau^{1/3} + \frac{m}{3} \chi^{1/2}. 
$$
Using this lower bound and applying Lemma \ref{L:mod-sk} to the pair $u_1, u_2$ then gives
\begin{equation}
\label{E:Ku2}
\p(|\scrK(u_1) - \scrK(u_2)|  \ge m \tau^{1/3}) \le ce^{-d m^{3/2}}.
\end{equation}
Now, we can think of $K^\ep_b$ as a subset of the box $T = [-b, b]^3 \X [\ep, 2b]$ with coordinates $x, y, t, s$, which is in turn a subset of the set $K^\ep_{2b}$. 
 
\medskip

Therefore by \eqref{E:Ku} and \eqref{E:Ku2}, we can apply Lemma \ref{L:levy-est} to the box $T$ with $\al_x = \al_y = 1/2, \al_s = \al_t = 1/3, \beta_x = \beta_y = 2, \beta_s = \beta_t = 3/2, r_x = r_y = 2b, r_s = r_t = b^3/\ep^3$. This gives the desired result after simplification.
The conditions that $b \ge 2$ and $\ep \le 1$ are used only in the simplification. 
\end{proof}

 We now use the modulus of continuity for $\scrL$ to understand how the values in $\scrL$ blow up. The first proposition controls blowup as the time increment goes to $0$.
\begin{prop}
\label{P:zero-decay}
For any set $B = [-b, b]^4 \cap \R^4_\uparrow$ with $b \ge 2$, there is a random constant $C$ satisfying the tail bound
$$
\p(C > m) \le cb^{12}e^{-dm^{3/2}}
$$
such that for $(x, t; y, t + s) \in B$, we have
$$
\lf|\scrK(x, t; y, t + s) \rg|\le C s^{1/3} \log^{4/3}\lf(\frac{4b}{s}\rg).
$$
Here $c$ and $d$ are universal constants.
\end{prop}

Throughout the proof, $c_1, c$ and $d$ are constants that may change from line to line. 

\begin{proof}
For a large constant $m_0 \in \N$ define the discrete set
$$
D_k = \{(x, t; y, t + s) \in B : t, x, y, s \in \Z/(2^{3k + 1} b^3), s \in [2^{-k}, m_0 b 2^{-k}] \}, \text{ and let } D = \bigcup_{k \ge 1} D_k.
$$
How large we need to take $m_0$ will be made clear in the final step of the proof. In particular, the choice of $m_0$ is independent of $B$.
By the Tracy-Widom one-point bound for the function $\scrK$, Theorem \ref{T:TW-airy}, for any $(x, t; y, t + s) \in D_k$ and $m > c_1$ we have
$$
\p(|\scrK(x, t; y, t + s)| > m \log^{2/3}(4b/s) s^{1/3}) \le c e^{-d \log (4b/s) m^{3/2}} \le c \frac{s^{13}}{b^{13}} e^{-d m^{3/2}}.
$$
Now, each of the sets $D_k$ has at most $c b^{12} 2^{12 k}$ elements, so by a union bound, we get that
\begin{equation}
\label{E:L-bd-Z}
\lf|\scrK(x, t; y, t + s) \rg| \le C s^{1/3} \log^{2/3}(4b/s), \qquad \mathforall (x, t; y, t +s) \in D,
\end{equation}
where $C$ is a random constant satisfying
$$
\p(C > m) \le \sum_{k=1}^\infty c b^{12} 2^{12 k} m_0^{13} 2^{-13k} e^{-d m^{3/2}}\le  c b^{12} e^{-dm^{3/2}}
$$
for $m > c_1$. The same bound holds for all $m$ by possibly increasing $c$.
We can bound the values of $\scrK$ off of $D$ by the modulus of continuity bound in Proposition \ref{P:mod-land}. By Proposition \ref{P:mod-land} applied to the set $B_k = K^{2^{-k}}_b$ for some $k \ge 1$, for all $u_1, u_2 \in B_k$ with $||u_1 -u_2|| \le 2^{-3k}/b^3$, we have that
\begin{equation}
\label{E:Ku1}
|\scrK(u_1) - \scrK(u_2)| \le C_k ||u_1 - u_2||^{1/3} \log^{2/3}\lf(||u_1 - u_2||^{-1}\rg),
\end{equation}
where $C_k$ satisfies the following tail bound for $m > c_1$:
\begin{equation}
\label{E:Ck-bound}
\p(C_k > m \log^{2/3}(2^k b)) \le c b^{10} 2^{6k} e^{-d \log(2^k b) m^{3/2}} \le c 2^{-k} e^{-d m^{3/2}}
\end{equation}
Now for points $u_1 = (x_1, t_1; y_1, t_1 + s_1), u_2 = (x_2, t_2; y_2, t_2 + s_2) \in B$ with
\begin{equation}
\label{E:u1u2}
b||u_1 - u_2||^{1/3} \le \min(s_1, s_2, 1)/2
\end{equation}
the bound \eqref{E:Ku1} applies to the pair $u_1, u_2$ either for $k = 1$ or for some $k$ with $2^{-k} \le s_1 \wedge s_2 \le 2^{1-k}$. Therefore for all $u_1, u_2 \in B$ satisfying \eqref{E:u1u2} we have
\begin{equation}
\label{E:L-bd}
|\scrK(u_1) - \scrK(u_2)| \le C' ||u_1 - u_2||^{1/3} \log^{4/3}\lf(||u_1 - u_2||^{-1} \rg), \text{ where } \;\; C' = \sup_k C_k\log^{-2/3}(2^k b).
\end{equation}
The constant $C'$ satisfies the same tail bound as $C$ by \eqref{E:Ck-bound} and a union bound. Now, as long as $m_0$ was chosen large enough, for every point $v = (x, t; y, t +s) \in B$, there is a point $u = (x', t'; y', t' +s') \in D$ such that
$$
b||v - u||^{1/3} \le \min (s, s', 1)/2, \quad \mathand s' \ge s/2.
$$
The inequality in \eqref{E:L-bd} then applies with $v = u_1, u = u_2$. Combining this with the bound on $|\scrK(u)|$ in \eqref{E:L-bd-Z} bounds $|\scrK(v)|$, proving the lemma.
\end{proof}

As a corollary of Proposition \ref{P:zero-decay}, we can get uniform control over the whole directed landscape. This follows immediately from a union bound.
\begin{corollary}
\label{C:infty-blowup}
There exists a random constant $C$ satisfying
$$
\p(C > m) \le ce^{-dm^{3/2}}
$$
for universal constants $c,d$ and all $m > 0$, such that for all $u = (x, t; y, t + s) \in \R^4_\uparrow$, we have
$$
\lf|\scrL(x, t; y, t + s) + \frac{(x - y)^2}{s} \rg|\le C s^{1/3} \log^{4/3}\lf(\frac{2(||u|| + 2)}{s}\rg)\log^{2/3}(||u|| + 2).
$$
\end{corollary}

Finally, we can use the control established in Corollary \ref{C:infty-blowup} to conclude that the metric composition law holds for $\scrL$. The next lemma will also establish control over where the maximum in the metric composition law is attained.

\begin{lemma}
	\label{L:metric-strong}
	For $u = (x, r;y, t) \in \R^4_\uparrow$ and $s \in (r, t)$, define
	$$
	f_{u, s}(z) = \scrL(x,r;z,s)+\scrL(z,s;y,t).
	$$
	On a set of probability $1$, $\scrL$ satisfies the metric composition law
	\begin{equation}
	\label{E:L-met-strong}
	\scrL(u)=\max_{z \in \mathbb R} f_{u, s}(z)
	\end{equation}
	for every $u = (x, r;y, t) \in \R^4_\uparrow$ and $s \in (r, t)$. Moreover, for any compact set $K \sset \R^4_\uparrow$, there exists a random constant $B_K$ such that for all $u = (x, r;y, t) \in K$ and $s \in (r, t)$, the set where $f_{u, s}$ attains its maximum lies in the interval $[-B_K, B_K]$.
\end{lemma}

\begin{proof}
	By construction, \eqref{E:L-met-strong} holds almost surely at all dyadic rational $u, s$. Now for any $u = (x, r;y, t)$ and $s \in (r, t)$, let $u_n= (x_n, r_n;y_n, t_n)$ and $s_n \in (r_n, t_n)$ be a sequence of dyadic rational points converging to $u$ and $s$.  Corollary \ref{C:infty-blowup} implies that $\mathcal L$ decays like a parabola as the distance between $x$ and $y$ grows. Thus the sequence of rightmost maximizers $z_n$ of $f_{u_n, s_n}$ is bounded, and hence has a subsequential limit $z$. By continuity of $\scrL$, the point $z$ satisfies $f_{u, s}(z) = \scrL(u).$ Moreover, continuity and metric composition at dyadic rationals guarantees that $f_{u, s}(y) \le \scrL(u)$ for all $y \in \R$. Hence metric composition holds at $u$.
	
	\medskip
	
	Now, the left hand side of \eqref{E:L-met-strong} is uniformly bounded below on any compact set by continuity. Moreover, Corollary \ref{C:infty-blowup} guarantees that $f_{u, s}(z)$ converges uniformly to $- \infty$ as $|z| \to \infty$ for all $u = (x, r;y, t) \in K$ and $s \in (r, t)$. Hence the set where $f_{u, s}$ attains its maximum is uniformly bounded over $u = (x, r;y, t) \in K$ and $s \in (r, t)$.
\end{proof}

We can now conclude the existence and uniqueness of the directed landscape.

\begin{theorem}
	\label{T:landscape-r4}
	The directed landscape exists and is unique in law.
\end{theorem}

\section{Convergence to the directed landscape}
\label{S:cvg-land}
In this section, we show that the directed landscape is the distributional limit of last passage percolation. For each $n$, let  $B^n \in C^\Z$ be a sequence of independent two-sided Brownian motions. Let
$$
(x, t)_n = (t + 2xn^{-1/3}, -\floor{tn})
$$
denote the translation between coordinates before and after the limit. For $(x, t; y, s) \in \R^4_\uparrow$  we can define the \textbf{last passage percolation}
\begin{equation}
\label{E:last-pass}
\begin{split}
\scrL_n(x, t ; y, s) = n^{1/6} \Big(B[(x, t)_n & \LP (y, s)_n ] - 2(s - t)\sqrt{n} - 2(y - x)n^{1/6} \Big).
\end{split}
\end{equation}
 $\scrL_n$ is not defined at some points in $\R^4_\uparrow$; we may formally set its values to $-\infty$ at those points.  It is defined on any compact subset of $\R^4_\uparrow$ for all large enough $n$, so uniform compact convergence will not be affected. We choose the value $-\infty$ to preserve the metric composition law.
\medskip

Let $F^o(\R^4_\uparrow, \R \cup \{-\infty\})$ be the space of extended real-valued functions with domain $\R^4_\uparrow$ that arise as functions of the form \eqref{E:last-pass} for a continuous sequence of functions $f = (f_i)_{i \in \Z}$ in place of $B$. We let $F(\R^4_\uparrow, \R \cup \{-\infty\})$ be the closure of this space with the topology of uniform-on-compact convergence. The reason for defining the space this way (instead of setting it to be all functions on $\R^4_\uparrow$) is just to ensure that it is separable.

\begin{theorem}
\label{T:landscape-limit}
There exists a coupling of $\scrL_n$ and $\scrL$ so that
for every compact set $K\subset \R^4_{\uparrow}$ there exists $a>1$ with
\begin{equation}
\label{E:a-KK}
\E a^{\sup_K|\scrL-\scrL_n |^{3/4}}\to 1.
\end{equation}
\end{theorem}

The key step in the proof of Theorem \ref{T:landscape-limit} is the following lemma on tails of time increments in last passage percolation. Note that the spatial increments of $\scrL_n$ are bounded by Proposition \ref{P:dyson-tails}, as in the proof of tightness of the Airy sheet, Lemma \ref{L:Sn-tight}.

\begin{lemma}
\label{L:tail-lpp}
Fix $b > 0$, and let
$$
K_b = [-b, b]^4 \cap \{(x, t; y, t + s) \in \R^4 : s \ge b^{-1} \},
$$
and
$$
K_{b, n} = K_b \cap \{(x, t; y, t + s) \in \R^4 : s, t \in n^{-1} \Z \}.
$$
Then there exist constants $c, d > 0$ such that for every $n \in \N$, every $(x,t;y, s), (x,t + r_1;y, s + r_2)\in K_{b, n}$, and all $a > 0$ we have that
$$
\p \lf( |\scrL_n(x,t;y, s) - \scrL_n(x,t + r_1;y, s + r_2)| \ge a ||(r_1, r_2)||^{1/6} \rg) \le ce^{-da^{3/4}}.
$$
\end{lemma}

Throughout the proof, $c$ and $d$ will be positive constants that may change from line to line and depend only on the compact set $K_b$. In particular, they will not depend on the choice of points in the compact set or on $n$.
\begin{proof}
We will assume that $r_1 = 0$ and that $r_2 = r > 0$.
 Extending to the case of general $(r_1, r_2)$ follows by symmetry. Since time coordinates for points in $K_{b, n}$ are in $n^{-1}\Z$, we can assume that $r \ge 1/n$ since otherwise $r=0$. By the metric composition law for last passage percolation, Lemma \ref{L:metric}, we have that
\begin{align}
\nonumber
 \scrL_n(y, s; y, s + r)  &\le \scrL_n(x,t;y, s + r) - \scrL_n(x,t; y, s)\\
\label{E:Ln}
 &= \scrL_n(Z, s; y, s + r) + [ \scrL_n(x,t; Z, s) - \scrL_n(x,t; y, s)],
\end{align}
where $Z$ is the rightmost maximizer of the function $z \mapsto \scrL_n(x,t; z, s) + \scrL_n(z, s; y, s + r)$. By Theorem \ref{T:top-bd}, we have the lower bound
\begin{equation}
\label{E:Z-bdd}
\p(\scrL_n(y, s; y, s + r) < -ar^{1/3}) \le ce^{-da^{3/2}}
\end{equation}
on the left hand side of \eqref{E:Ln}. The same bound holds with possibly different $c, d$ when $r^{1/3}$ is replaced by $r^{1/6}$ since $r$ is bounded. To complete the proof, we show that the probability that the right hand side of \eqref{E:Ln} is larger than $ar^{1/6}$ is bounded above by $ce^{-d a^{3/4}}$ for all $a > 0$.

\medskip

We can bound the right side of \eqref{E:Ln} by using the bounds from Section \ref{S:prelim} on Brownian melons. Indeed, recall that the two functions $\scrL_n(\cdot, s; y, s + r)$ and $\scrL_n(x, t; \cdot, s)$ are simply rescaled and shifted top lines of Brownian melons by \eqref{E:Wf-one}. 
First, Proposition \ref{P:cross-prob} implies that
\begin{equation}
\label{E:ln-1}
\p(\scrL_n(Z, s; y, s + r) > ar^{1/3}) \le ce^{-da^{3/2}},
\end{equation}
which gives the desired bound on the first term on the right side of \eqref{E:Ln} since $r^{1/3} \le c r^{1/6}$ for the $r$-values we care about.
For the second term, first note that by applying Proposition \ref{P:cross-prob} to the term $\scrL_n(x,t; Z, s)$ and Theorem \ref{T:top-bd} to the term $\scrL_n(x,t; y, s)$, we have the bound
\begin{equation}
\label{E:Z-large}
\p(\scrL_n(x,t; Z, s) - \scrL_n(x,t; y, s)\ge b) \le ce^{-db^{3/2}}.
\end{equation}
Here we have used that $(t-s)^{1/3}$, which enters into the scaling of this difference, is uniformly bounded on $K_{b}$. We will use this bound when $b$ is large, but we will also need a $r$-dependent bound that will work for small $b$.

\medskip

When $Z$ is close to $y$, we can use the Gaussian tail bound on differences between points in Proposition \ref{P:dyson-tails} to bound the left hand side of \eqref{E:Z-large}. In particular, by combining the tail bounds in Proposition \ref{P:dyson-tails} with the modulus of continuity lemma for general tail-bounded processes (Lemma \ref{L:levy-est}), we have that for every $\ep \in [0, 1]$ and $b > 0$, that
\begin{equation}
\label{E:mod-bd-2}
\p \lf(\max_{z \in [-y - \ep, y + \ep]} \lf|\scrL_n(x,t; z, s) - \scrL_n(x,t; y, s) \rg| \ge b \rg) \le ce^{-db^{3/2}\ep^{-3/4}}.
\end{equation}
To get that $c$ and $d$ do not depend on any parameters we have used that the distance $|x - y|$ is bounded above and that the time increment $s - t$ is bounded below on $K_b$. Also, by Lemma \ref{L:wm-max} and Brownian scaling, we have control over how much $Z$ differs from its expected location at $f_r(x, y) = x + (y-x)(t + s)(t + s + r)^{-1}$. We get that
\begin{equation}
\label{E:Z-y}
\p(|Z - f_r(x, y)|> mr^{1/3}) \le ce^{-dm^{3/4}}.
\end{equation}
Again, the fact that $s -t$ is uniformly bounded below on $K_b$ is necessary to ensure that the constants $c$ and $d$ are independent of the points $(x, t; y, s)$ and $(x, t; y, s + r)$. Note that $|f_r(x, y) - y| \le c_1 r$ on the set $K_b$ for a $K_b$-dependent constant $c_1$. Since $r$ is bounded above on $K_b$, we can use the triangle inequality to rewrite \eqref{E:Z-y} as
\begin{equation}
\label{E:Z-y'}
\p(|Z - y|> mr^{1/3}) \le ce^{-dm^{3/4}}.
\end{equation}
Combining \eqref{E:mod-bd-2} and \eqref{E:Z-y'} by setting $\ep = br^{1/6}$ and $m = br^{-1/6}$, we get that
\begin{equation}
\label{E:mod-bd}
\p \lf(\scrL_n(x,t; Z, s) - \scrL_n(x,t; y, s) \ge b \rg) \le ce^{-db^{3/4}r^{-1/8}}
\end{equation}
whenever $b \le r^{-1/6}$. Also, when $b \ge r^{-1/6}$, then $b^{3/2} \ge b^{3/4}r^{-1/8}$, so we can combine \eqref{E:mod-bd} with \eqref{E:Z-large} to get that \eqref{E:mod-bd} holds for all $b, r$. Setting $ar^{1/6} = b$ in \eqref{E:mod-bd} gives that
\begin{equation*}
\p \lf(\scrL_n(x,t; Z, s) - \scrL_n(x,t; y, s) \ge ar^{1/6} \rg) \le ce^{-da^{3/4}},
\end{equation*}
 which is the desired bound on the second term on right hand side of \eqref{E:Ln}.
\end{proof}

\begin{proof}[Proof of Theorem \ref{T:landscape-limit}]
Replace $\scrL_n$ with a continuous interpolated version $\mathcal{J}_n$.  The version $\mathcal{J}_n$ will equal $\scrL_n$ whenever $ns, nt \in \Z$. We can do this in such a way that the sequence $\mathcal{J}_n$ satisfies the tail bound in Lemma \ref{L:tail-lpp} everywhere on $K_b$, not just on the sets $\scrK_{b, n}$. It suffices to prove convergence of $\mathcal{J}_n$ to $\scrL$.

\medskip

By the tail bound in Lemma \ref{L:tail-lpp} on the time increments of $\mathcal J_n$, the tail bound in Proposition \ref{P:dyson-tails} which applies to the spatial increments of $\mathcal J_n$, and the Kolmogorov-Chentsov criterion, we can conclude that the sequence $\mathcal{J}_n$ is tight. Moreover, if $\mathcal{J}_n$ converges in distribution to a subsequential limit $\mathcal{J}$, then there is a coupling of the $\mathcal{J}_n$ and $\mathcal{J}$ such that in this coupling $E a^{\sup_K|\scrL-\scrL_n |^{3/4}}\to 1$ for some $a > 1$, the equivalent of \eqref{E:a-KK}. This follows by the same reasoning as in Lemma \ref{L:Sn-tight}.

\medskip

Now let $\mathcal{J}$ be a subsequential limit of $\mathcal{J}_n$. By the definition of $\mathcal{J}_n$ and the Airy sheet convergence in Theorem \ref{T:airy-sheet}, the marginals $\mathcal{J}(\cdot, t, \cdot, s)$ are all rescaled Airy sheets. Also, $\mathcal{J}$ inherits the metric composition law, property II of the directed landscape, from the sequence $\scrL_n$. This follows from the tail bounds on maximizing locations of pre-limiting Airy sheets from Lemma \ref{L:wm-max} exactly as in the proof of Proposition \ref{P:geod-eqn}. 

\medskip

Moreover, $\scrJ$ has independent increments on any set of $k$ disjoint intervals since when $[s, t] \cap [s', t'] = \emptyset$, the processes $\scrJ_n(\cdot, t; \cdot, s)$ and $\scrJ_n(\cdot, t'; \cdot, s')$ are determined by different sets of Brownian motions  for all large enough $n$. This independence extends to intervals that share an endpoint by continuity. Thus $\scrJ$ satisfies the conditions of Definition \ref{D:directed-landscape}, and so $\scrJ$ must be the directed landscape.
\end{proof}

\section{Directed geodesics}
\label{S:directed-geodesics}

In this section, we construct geodesics in the directed landscape $\scrL$. First, for a continuous path $\pi:[t, s] \to \R$, we can define the \textbf{length} of $\pi$ by
\begin{equation}
\label{E:land-wt}
\int d \scrL \circ \pi = \inf_{k \in \N} \inf_{\{t = t_0 < t_1 < \dots < t_k = s\}} \sum_{i=1}^k \scrL(\pi(t_{i-1}), t_{i-1}; \pi(t_i), t_{i}).
\end{equation}
In other words, the length of $\pi$ is the infimum over all partitions of the sum of increments in $\scrL$ along $\pi$. This definition is the analogy of defining curve length in Euclidean space by piecewise linear approximation. By the triangle inequality \eqref{E:triangle-ineqL} for $\scrL$, we always have
$$
\int d \scrL \circ \pi \le \scrL(\pi(t), t;  \pi(s), s).
$$
We say that $\pi$ is a \textbf{directed geodesic} from $(\pi(t), t)$ to $(\pi(s), s)$ if equality holds. This is equivalent to saying that equality holds in \eqref{E:land-wt} for all subdivisions before taking any infima. A simple induction argument shows that if  for any $t_0=t<t_1<t_2<t_3=s$ we have 
\begin{equation}
\label{E:geodesic-3}
\int d \scrL \circ \pi = \sum_{i=1}^3 \scrL(\pi(t_{i-1}), t_{i-1}; \pi(t_i), t_{i}),
\end{equation}
then  $\pi$ is a directed geodesic.

\begin{theorem}
\label{T:directed-geodesic}
Let $\scrL$ be the directed landscape, and fix $u = (x, t; y, s) \in \R^4_\uparrow$. Then almost surely, there exists a unique directed geodesic $\Pi_u$ from $(x, t)$ to $(y, s)$.
\end{theorem}

\medskip
By the properties of the directed landscape in Lemma \ref{L:land-props}, the distribution of $\Pi_u$ is independent of the point $u \in \R^4_\uparrow$ up to symmetry. Let $u=(0,0;0,1)$. For $r>0$ the path $\Pi_{r^3u}(\cdot)$ has the same distribution as $r^2\Pi_{u}(\cdot/r^3)$. The distribution of $\Pi_v$ for any $v$ can be obtained as a translated shear of this family. So it suffices to prove Theorem \ref{T:directed-geodesic} for $u=(0,0;0,1)$.

\medskip

To obtain tightness on the path locations we will use Lemma \ref{L:airy-max}. This lemma can be used in conjunction with Proposition \ref{P:geod-eqn} to  give a bound over multiple maximum locations at once. For $(x, t; y, t + r^3) \in \R^4_\uparrow$ we define $Z^+_{t, r}(x, y)$ as the rightmost maximizer of the function
$$
f(z) = \scrL(x, t; z, t + r^3/2) +  \scrL(z, t + r^3/2; y, t+r^3),
$$
and similarly define $Z^-_{t, r}(x, y)$ as the leftmost maximizer.

\begin{lemma}
\label{L:path-loc-2} 
Fix $t \in \R$ and $r > 0$, and let $K = [a, a + b] \X [c, c + b] \sset \R^2$. Then
\begin{align*}
\p\big(\exists (x, y) \in K \text{ such that } |Z^+_{t, r}(x, y) - (x + y)/2|& > mr^2 \big) 
\le \lf(\frac{b}{mr^2}\rg)^2c_1e^{-c_2m^3}
\end{align*}
for constants $c_1, c_2 \in \R$. The same bound holds for $Z^-$.
\end{lemma}

\begin{proof}
The functions $Z^+_{t, r}$ and $Z^-_{t, r}$ are nondecreasing functions of $x$ and $y$ by Proposition \ref{P:geod-eqn}. Hence to prove the theorem, it suffices to take a union bound over $(x,y)$ in the grid
$$
 mr^2 \Z^2 \cap \{x \in \R^2: d(x, K) \le \sqrt{2} mr^2\}.
 $$
 We have enlarged the compact set $K$ to guarantee that we have an outer grid boundary containing $K$. The bound then follows if we can show that for every fixed $x, y$,
 $$
 \p\big(|Z^+_{t, r}(x, y) - (x + y)/2| > mr^2 \big) 
 \le c_1e^{-c_2m^3},
 $$ 
and similarly for $Z^-$. Now, the probability on the left side above is independent of $x, y$ by the spatial and skew stationarity of $\scrL$, Lemma \ref{L:land-props}. Moreover, for $x = y = 0$, the function $f$ is then a sum of independent Airy processes of scale $2^{-1/3}r $. Applying Lemma \ref{L:airy-max} then yields the desired bound.
\end{proof}

We are now ready to prove Theorem \ref{T:directed-geodesic}. As in L\'evy's construction of Brownian motion, our proof gives an essentially optimal modulus of continuity bound on $\Pi_u$.

\begin{prop}
\label{P:mod-dg}
There exists a random constant $C$ such that the path $\Pi_u$ in Theorem \ref{T:directed-geodesic} with $u = (0, 0 ; 0, 1)$ satisfies
$$
|\Pi_u(t_1) - \Pi_u(t_2)| \le C|t_1 - t_2|^{2/3}\log^{1/3} \lf(\frac{2}{|t_1 - t_2|} \right)
$$
for all $t_1, t_2 \in [0, 1]$. Moreover,
$$
\mathbb P(C > b) \le c_1 e^{-c_2b^3}
$$
for some constants $c_1, c_2$.
\end{prop}

\begin{proof}[Proof of Theorem \ref{T:directed-geodesic} and Proposition \ref{P:mod-dg}]
We first approximate the path $\Pi = \Pi_u$ on dyadic rationals.
Let $\Pi_1:[0, 1] \to \R$ be chosen so that $\Pi_1(0) = \Pi_1(1) = 0$ and
$$
\scrL(0, 0; \Pi_1(1/2), 1/2) + \scrL(\Pi_1(1/2), 1/2; 0, 1)  = \scrL(u),
$$
and so that $\Pi_1$ is linear at times in between. Now for $n \in \N$ and $i \in \{0, \dots 2^n\}$ define $t_{n, i} = i/2^n$. We then recursively define $\Pi_n:[0, 1] \to \R$ so that $\Pi_n(t_{n-1, i}) = \Pi_{n-1}(t_{n-1, i})$ for all $i \in \{0, 1, \dots, 2^{n - 1}\}$, and so that for odd $i \in \{0, 1, \dots, 2^n\}$, we have
\begin{align*}
\scrL(\Pi_n(t_{n, i-1}), t_{n, i-1}; \Pi_n(t_{n, i}), t_{n, i}) &+ \scrL(\Pi_n(t_{n, i}), t_{n, i}; \Pi_n(t_{n, i+1}), t_{n, i+1})  \\
&= \scrL(\Pi_n(t_{n, i-1}), t_{n, i-1}; \Pi_n(t_{n, i +1}), t_{n, i+1}).
\end{align*}
Define $\Pi_n$ to be linear at times in between. Such a sequence of paths $\Pi_n$ exists by property III of $\scrL$. 
We now show that $\Pi_n$ has a uniform limit satisfying the desired modulus of continuity. 

\medskip

Fix $b > 0$, and for $k \in \{1, 2 \dots\}$, define $h_k = b2^{-2k/3} k^{1/3}$. Also define $\ell_0 = 0$ and for $k \ge 1$, let $\ell_k=h_1+\ldots +h_k$. We recall the notation $Z^+_{t, r}$ and $Z^-_{t, r}$ from Lemma \ref{L:path-loc-2}.
Define
\begin{align*}
A_k & = \Bigg\{ \text{There exists } (x, y) \in [-\ell_{k}, \ell_{k}]^2, i \in \{1, \dots, 2^{k}\}, \text{ such that either } \\
&\lf|Z^+_{t_{k, i-1}, 2^{-k/3}}(x, y) -  (x + y)/2\rg| \ge h_{k+1} \;\text{ or }\; \lf|Z^-_{t_{k, i-1}, 2^{-k/3}}(x, y) -  (x + y)/2\rg| \ge h_{k+1} \Bigg\}.
\end{align*}
By Lemma \ref{L:path-loc-2} and a union bound,
\begin{equation}
\label{E:A-bd}
\p\lf(\bigcup_{k \ge 0} A_k \rg) \le c_1 e^{-c_2 b^3}
\end{equation}
for universal constants $c_1, c_2$. On the event $A_0^c$, we have $||\Pi_1||_\infty \le \ell_1$. Also, if $\|\Pi_{k-1}\|_\infty \le \ell_{k-1}$, and the event $A^c_{k-1}$ holds, then
\begin{equation}
\label{E:pi-k-approx}
\|\Pi_k-\Pi_{k-1}\|_\infty<h_{k} \quad \text{and the derivatives satisfy} \quad \|\Pi'_k-\Pi'_{k-1}\|_\infty<h_{k}2^k.
\end{equation}
As a consequence of the first bound in \eqref{E:pi-k-approx}, $\|\Pi_{k}\|_\infty \le \ell_{k}$. Hence on the complement of the event $A = \bigcup_{k \ge 0} A_k$, the bounds in \eqref{E:pi-k-approx} hold for all $k \ge 1$.
Since the sequence $h_k$ is summable, the first bound in \eqref{E:pi-k-approx} implies that $\Pi_k$ is a Cauchy sequence in the uniform norm, and hence has a continuous limit $\Pi$ on the event $A^c$. The modulus of continuity of $\Pi$ can be bounded as follows using \eqref{E:pi-k-approx} for any $r \in \N$.
$$
|\Pi(t_2)-\Pi(t_1)| \le |t_2 - t_1|\sum_{k=1}^{r}2^k h_k + 2\sum_{k=r}^{\infty} h_k.
$$
Setting $r=1-\lfloor \log_2|t_2-t_1| \rfloor$ and using the probability bound in \eqref{E:A-bd} then proves the modulus of continuity bound of Proposition \ref{P:mod-dg} for $\Pi$. Moreover, by construction, equality holds in \eqref{E:land-wt} for all dyadic time subdivisions. Continuity implies the same for all finite subdivisions. Thus $\Pi$ is a directed geodesic. Uniqueness follows from Lemma \ref{L:airy-max} applied at rational intermediate times. 
\end{proof}

\section{Joint limits of last passage paths}
\label{S:joint-limits-lp-paths}
In this section, we show that last passage paths in $\scrL_n$, defined in \eqref{E:last-pass}, converge jointly to directed geodesics in $\scrL$. The proof is essentially topological, once we establish a basic fact about the directed landscape. As a byproduct of the proof, we show that almost surely in $\scrL$, there exists a directed geodesic between every pair of points: Lemma \ref{L:rig-exist}.  This result is not implied by Theorem \ref{T:directed-geodesic}, which only concerns geodesics between single points. Lemma \ref{L:rig-exist} can also be used to give an alternate proof of existence and almost sure uniqueness of directed geodesics, but does not yield the strong modulus of continuity bound in Theorem \ref{T:directed-geodesic}.

\medskip

The proof of geodesic convergence is topological and follows from a deterministic convergence statement about particular functions $f: \R_\uparrow^4 \to \R \cup \{-\infty\}$. We say that a function $f:\R_\uparrow^4\to \R$ is a \textbf{landscape} if for every pair $x, y \in \R$ and $s < r < t \in \R$ we have that
\begin{equation}
\label{E:met-2}
f(x, s; y, t) = \max_{z \in \R} f(x, s; z, r) + f(z, r; y, t)
\end{equation}
We say that a landscape $f$ is \textbf{proper} if the following four conditions hold.
\smallskip
\begin{enumerate}[label=(\roman*), nosep]
\item The landscape $f$ is continuous.
\item For every bounded set $[-b, b]^4 \cap \R^4_\uparrow$, there exists a constant $c_b$ such that
$$
\lf|f(x, t; y, t + s) - \frac{(x -y)^2}{s} \rg| \le c_b
$$
for all $(x, t; y, t + s) \in [-b, b]^4$.
\item For every compact set $K \sset \R^4_\uparrow$, there exists a constant $b_K > 0$ such that the set where the function on the right hand side of \eqref{E:met-2} attains its maximum is contained in $[-b_K, b_K]$ for all $(x, s; y, t) \in K$ and $r \in (s, t)$.
\item For all $r < t$, $x \le x'$, and $y \le y'$, we have that
\begin{equation}
f(x, r; y, t) + f(x', r; y', t) \ge f(x, r; y', t) + f(x, r; y', t)
\end{equation}
\end{enumerate}

\smallskip

\noindent For any landscape $f$, we can define length $\int d f \circ \pi$ as in \eqref{E:land-wt} with $f$ in place of $\scrL$. Again, we say that a continuous function $\pi:[t, s] \to \R$ is a geodesic from $(\pi(t), t)$ to $(\pi(s), s)$ if
$$
\int d f \circ \pi = f(\pi(t), t; \pi(s), s).
$$
The directed landscape $\scrL$ is almost surely a proper landscape. This follows by Lemma \ref{L:metric-strong} (for (i) and (iii)), Proposition \ref{P:zero-decay} (for (ii)), and Lemma \ref{L:airy-facts} and continuity (for (iv)). Also, the last passage percolations $\scrL_n$ are landscapes. Our goal is to prove a deterministic  statement about convergence of geodesics. We need a notion of convergence that accommodates geodesics with different domains.

\medskip

For this, let $\scrG$ be the set of all continuous functions $\pi:[a, b] \to \R$ for any $a,b\in\R$. For $\pi:[a, b] \to \R \in \scrG$, let 
$$
\mathfrak{g}\pi = \{(\pi(r), r) : r \in [a, b] \}.
$$
The set $\mathfrak{g}\pi$ is the graph of $\pi$ with the coordinates switched. This is the natural order of coordinates in landscape notation. A sequence of functions $\pi_n \to \pi$ in $\scrG$ if $\mathfrak{g}\pi_n \to \mathfrak{g}\pi$ in the Hausdorff topology. When $\pi_n, \pi$ have the same domain, then this is equivalent to uniform convergence of functions.

\begin{theorem}
\label{T:geod-cvg}
Let $f_n$ be a sequence of landscapes converging to a proper landscape $f$ uniformly on compact subsets of $\R^4_\uparrow$. Suppose that for a fixed $(u, v) \in \R^4_\uparrow$, there is a unique geodesic $\pi$ from $u$ to $v$ in $f$. Suppose that $(u_n, v_n) \to (u, v)$, and 
that for all large enough $n$, there exists a geodesic $\pi_n$ from $u_n$ to $v_n$ in $f_n$. Then $\pi_n \to\pi$ in $\scrG$ as $n \to \infty$.
\end{theorem}

We first show that geodesics exist in any proper landscape.
Recall that $\pi$ is a \textbf{rightmost geodesic} from $(x, t)$ to $(y, s)$ if $\pi \ge \tau$ for any other geodesic $\tau$ from $(x, t)$ to $(y, s)$. We can similarly define leftmost geodesics.

\begin{lemma}
\label{L:rig-exist}
Let $f$ be a proper landscape. Then for any $(x, t; y, s) \in \R^4_\uparrow$, there exist rightmost and leftmost geodesics from $(x, t)$ to $(y, s)$.
\end{lemma}

\begin{proof}
We will show the existence of a rightmost geodesic $\pi$. The leftmost geodesic exists by a symmetric argument.
We define $\pi(t) = x$ and $\pi(s) = y$, and for every $r \in (s, t)$, set $\pi$ to be the rightmost maximizer of the function on the right hand side of \eqref{E:met-2}. The function $\pi$ is bounded by condition (iii) in the definition of proper landscape. We need to check that $\pi$ is continuous and that $\pi$ is a geodesic. 

\medskip

To check continuity, consider a sequence $r_n \to r \in [s, t]$. 
We will use the equality
\begin{equation}
\label{E:f-rn}
f(x,s ; y, t) = f(x, s; \pi(r_n), r_n) + f(\pi(r_n), r_n; y, t).
\end{equation}
First consider the case when $r_n \to s$. By continuity of $f$ and the boundedness of $\pi$, the second term on the right hand side of \eqref{E:f-rn} remains bounded as $n \to \infty$. Therefore the first term $f(x, s; \pi(r_n), r_n)$ also remains bounded, and so by condition (ii) in the definition of proper landscape we must have $\pi(r_n) \to \pi(r)$ since otherwise $f(x, s; \pi(r_n), r_n)$ would get arbitrarily large and negative as $n \to \infty$. The case when $r_n \to t$ follows a symmetric argument.

\medskip

Now suppose $r_n \to r \in (s, t)$ and let $w$ be a limit point of $\pi(r_n)$. By passing to a subsequence, we can assume that $\pi_n(r_n) \to w$, and that either $r_n < r$ for all $n$ or $r_n > r$ for all $n$. These two cases have symmetric arguments, so we will assume $r_n < r$ for all $n$.  By continuity of $f$, along this subsequence the right hand side of \eqref{E:f-rn} converges to
$$
f(x, s; w, r) + f(w, r; y, t).
$$
In particular, \eqref{E:f-rn} and the metric composition law for $f$ implies that $w$ maximizes the function
$$
z \mapsto f(x, s; z, r) + f(z, r; y, t).
$$
Since $\pi(r)$ is the rightmost maximizer of this function, $w \le \pi(r)$. We now prove the opposite inequality.

\medskip

Metric composition implies that there exists a point $z_n \in \R$ such that
\begin{equation}
\label{E:f-tritri}
   f(x,s ; y, t) = f(x, s; z_n, r_n) + f(z_n, r_n; \pi(r), r) + f(\pi(r), r; y, t).
\end{equation}
Equation \eqref{E:f-tritri} and the metric composition law for $f$ imply that
$$
f(x,s ; y, t) = f(x, s; z_n, r_n) + f(z_n, r_n; y, t).
$$
Therefore both $z_n$ and $\pi(r_n)$ maximize the function in the right hand side \eqref{E:met-2} at $r_n$, so $z_n \le \pi(r_n)$ since $\pi(r_n)$ is the rightmost maximizer. The sequence $z_n$ is bounded by condition (iii) in the definition of proper landscape.
Moreover, $f(z_n, r_n; \pi(r), r)$ must stay bounded as $r_n \to r$ in order to preserve the equality in \eqref{E:f-tritri}, since all other terms in \eqref{E:f-tritri} stay bounded by continuity of $f$. Therefore just as in the paragraph following \eqref{E:f-rn}, $z_n \to \pi(r)$ as $n \to \infty$. Since $z_n \le \pi(r_n)$ and $\pi(r_n) \to w$, we have that $\pi(r) \le w$ as desired.

\medskip

To check that $\pi$ is a geodesic, just as in  \eqref{E:geodesic-3} it is enough to show that for any $r_1 < r_2 \in(s, t)$ we have 
\begin{equation}
\label{E:want-spl-m}
f(x,s;y,t) = f(x, s; \pi(r_1), r_1)  + f(\pi(r_1), r_1; \pi(r_2), r_2) + f(\pi(r_2), r_2; y, t).
\end{equation}
By the metric composition law, we can find a point $z_1$ such that 
\begin{equation}
\label{E:fz1}
    f(x,s;\pi(r_2),r_2) = f(x, s; z_1, r_1)  + f(z_1, r_1; \pi(r_2), r_2).
\end{equation}
By the construction of $\pi(r_2)$, \eqref{E:fz1} implies that
\begin{equation}
\label{E:pir2}
    f(x, s; y, t)  =f(x, s; z_1, r_1)  + f(z_1, r_1 ; \pi(r_2), r_2) + f(\pi(r_2), r_2; y, t).
\end{equation}
By the triangle inequality for $f$ applied at the points $(z_1, r_1), (\pi(r_2), r_2)$, and $(y, t)$, \eqref{E:pir2} implies that
$$
f(x, s; y, t) \le f(x, s; z_1, r_1)  + f(z_1, r_1; y, t).
$$
Metric composition for $f$ implies that the opposite inequality must also hold, and hence that $z_1$ maximizes the function $z \mapsto f(x, s; z, r_1) + f(z, r_1; y, t)$.
In particular, $z_1 \le \pi(r_1)$ since $\pi(r_1)$ is the rightmost maximizer of this function.
By a similar argument, we can find a point $z_2$ such that $z_2 \le \pi(r_2)$ and
\begin{equation}
\label{E:pir1pir1}
    f(x, s; y, t) =f(x, s; \pi(r_1), r_1)  + f(\pi(r_1), r_1 ; z_2, r_2) + f(z_2, r_2; y, t).
\end{equation}
Now, we can sum the equalities \eqref{E:pir2} and \eqref{E:pir1pir1} 
and apply property (iv) of proper landscapes to the points $z_1 \le \pi(r_1)$ and $z_2 \le \pi(r_2)$ to get that
\begin{align}
\nonumber
&2f(x, s; y, t) \le\\
\label{E:r1}
&f(x, s; z_1, r_1)  + f(z_1, r_1; z_2, r_2) + f(z_2, r_2; y, t) \\
\label{E:r2}
+ &f(x, s; \pi(r_1), r_1)  + f(\pi(r_1), r_1 ; \pi(r_2), r_2) + f(\pi(r_2), r_2; y, t)
\end{align}
The triangle inequality for $f$ implies that each of \eqref{E:r1} and \eqref{E:r2} are less than or equal to $f(x, s; y, t)$. By the inequality above, this implies that \eqref{E:r2} is in fact equal to $f(x, s; y, t)$, yielding \eqref{E:want-spl-m}.
\end{proof}

We will also need the following lemma about landscape convergence.

\begin{lemma}
\label{L:limsup-princ}
Let $f_n$ be a sequence of landscapes converging uniformly on compact sets to a proper landscape $f$, and fix a bounded set $B = [-b, b]^4 \cap \R^4_\uparrow$. Then there exists a positive constant $c_b$ such that for all $\ep \in (0, 1)$ there exists $n_\ep \in \N$ such that for all $n \ge n_\ep$ and $(x, t; y; t + s) \in B$ we have
$$
f_n(x, t; y; t + s) \le c_b -\frac{(x -y)^2}{s + \ep}.
$$
\end{lemma}

\begin{proof}
Fix $(x, t ; y, t + s) \in B$ and $\ep \in (0,1)$. By the triangle inequality for $f_n$, we have that
$$
f_n(x, t; y, t+s) \le f_n(x, t; y, t+s + \ep) - f_n(y, t + s; y, t+s + \ep).
$$
In particular, the right hand side above converges uniformly to
$$
f(x, t; y, t+s + \ep) - f(y, t + s; y, t+s + \ep)
$$
for $[x, t; y; t + s] \in [-b, b]^4$ for any $b, \ep > 0$. Using condition (ii) in the definition of proper landscape for $f$ then implies the lemma.
\end{proof}

\begin{proof}[Proof of Theorem \ref{T:geod-cvg}]
Let $(u, v) = (x, t; y, s)$ and $(u_n, v_n) = (x_n, t_n; y_n, s_n)$. 
Suppose that $\pi_n$ does not converge to $\pi$ in $\scrG$. Then there exists a subsequence $Y$ and an $\ep > 0$ such that for every $n \in Y$, the Hausdorff distance between $\mathfrak{g}\pi_n$ and $\mathfrak{g} \pi$ is greater than $\ep$. Let 
$$
K_\ep = \{w \in \R^2 : \min_{w' \in \mathfrak{g} \pi} \|w - w'\| = \ep \}.
$$
Now for large enough $n$,  $\|u_n-u\|<\ep$ and  $\|v_n-v\|<\ep$. For such $n \in Y$, 
the continuity of the $\pi_n$ guarantees that $\mathfrak{g}\pi_n$ must intersect $K_\ep$ at  some point $w_n$. By passing to a further subsequence, we can guarantee that $w_n$ converges to a point $w = (z, r) \in K_\ep$. The time coordinate $r$ is in $[t, s]$. Since each $\pi_n$ is a geodesic, we have the metric composition law
 \begin{equation}
 \label{E:Ln'}
  f_n(u_n, w_n) +  f_n(w_n, v_n) = f_n(u_n, v_n).
 \end{equation}
 If $r \in (s, t)$, then both sides above must converge to the corresponding values of $f$ since all the points involved lie in a common compact subset of $\R^4_\uparrow$. This yields the equality
$$
f(u, w) +  f(w, v) = f(u, w).
$$
By Lemma \ref{L:rig-exist}, there are geodesics in $f$ from $u$ to $w$ and from $w$ to $v$. Concatenating these two geodesics gives a geodesic from $u$ to $v$. Since $w \notin \mathfrak{g} \pi$, this contradicts the uniqueness of $\pi$.

 \medskip

 Therefore either $r_n \to t$ or $r_n \to s$. We will only deal with the case $r_n \to t$, as the case where $r_n \to s$ is similar. In this case, all the points $(u_n, v_n), (w_n, v_n)$ are contained in a common compact set in $\R^4_\uparrow$, so
\begin{equation}
\label{E:ff}
f_n(u_n, v_n) - f_n(w_n, v_n) \to f(u, v) - f(w, v).
\end{equation}
On the other hand, all the points $(u_n, w_n)$ lie in a bounded subset of $\R^4_\uparrow$, and they converge to a point $(u, w) = (x, t, x', t)$ with $x \ne x'$. Therefore by Lemma \ref{L:limsup-princ}, we have that 
$$
\lim_{n \to \infty} f_n(u_n, w_n) = -\infty.
$$
Combining this with \eqref{E:ff} contradicts \eqref{E:Ln'} for large enough $n$.
\end{proof}

Next, we apply Theorem \ref{T:geod-cvg} to prove convergence of last passage paths in the prelimiting last passage percolations $\scrL_n$, Theorem \ref{T:intro-cvg-paths}. This is not immediate from Theorem \ref{T:geod-cvg} since the rescaled last passage paths in Theorem \ref{T:intro-cvg-paths} are not geodesics in $\scrL_n$. However, we can show that these rescaled paths will always be close to geodesics by tweaking the definitions.

Recall that $\scrL_n$ is the rescaled version of 
$$
B[(x, t)_n \LP (y, s)_n] = B[(t + 2xn^{-1/3}, -\floor{tn}) \LP (s + 2yn^{-1/3}, - \floor{sn})].
$$
Fix $(x, s; y, t) \in \R^4_\uparrow$ and let 
$$
\pi_n= \pi_n[(x, t)_n, (y, s)_n]
$$
be any last passage path in $B$ from $(x, t)_n$ to $(y, s)_n$. Let 
$
h_n
$ be the increasing affine function mapping $[t, s]$ onto $[t + 2xn^{-1/3}, s + 2yn^{-1/3}]$, and define the rescaled last passage path
$$
\tau_n(r) = \frac{\pi_n \circ h_n(r) + n h_n(r)}{2n^{2/3}}
$$
for $r \in [t,s]$. We would like to use Theorem \ref{T:geod-cvg} to show that $\tau_n$ converges to a directed geodesic in $\scrL$, but as written the path $\tau_n$ is not a geodesic in $\scrL_n$. In particular, it is not continuous. We work around this by constructing geodesics $\sig_n$ which are close to $\tau_n$. 
First define a function $\tilde \pi_n:[t + 2xn^{-1/3}, s + 2yn^{-1/3}] \to \R$ whose inverse will be continuous and will help us define $\sigma_n$. The function $\tilde \pi_n$ has the following specifications.
\begin{itemize}
    \item Let $z_1 < \dots < z_k$ be the set of points in $[t + 2xn^{-1/3}, s + 2yn^{-1/3}]$ when $$\lim_{z \to z_i^+} \pi_n(z) \ne \lim_{z \to z_i^-} \pi_n(z).
    $$
    Here recall the convention that 
    $$
    \lim_{z \to (t + 2xn^{-1/3})^-} \pi_n(z) = -\floor{tn} \quad \mathand \quad \lim_{z \to (s + 2yn^{-1/3})^+} \pi_n(z) = -\floor{sn}.
    $$
    Set $z_0 = t + 2xn^{-1/3}$ and $z_{k+1} = s + 2yn^{-1/3}$.
   On each of the intervals $[z_i, z_{i+1})$, let $\tilde \pi_n$ be equal to the linear function satisfying $\tilde \pi_n(z_i) = \pi_n(z_i)$ and $\tilde \pi_n(z_{i+1}) = \pi_n(z_i) - 1/2$. Informally, we are giving the flat parts of $\pi_n$ some slope so that there is a  continuous inverse. The constant $1/2$ yields better behavior at the right endpoint than the more natural $1$.
    \item Let $\tilde \pi_n(s + 2yn^{-1/3}) = -\floor{sn} -1/2$. 
\end{itemize}
The function $\tilde \pi_n$ is a strictly decreasing function from $[t + 2xn^{-1/3}, s + 2yn^{-1/3}]$ to $[-\floor{sn} - 1/2, -\floor{tn}]$. There is a unique surjective nonincreasing function 
$$
\tilde \pi_n^{-1}: [-\floor{sn} - 1/2, -\floor{tn}] \to [t + 2xn^{-1/3}, s + 2yn^{-1/3}]
$$
such that
$
\tilde \pi_n^{-1} \circ \tilde \pi_n
$
is the identity map on $[t + 2xn^{-1/3}, s + 2yn^{-1/3}]$. The construction of $\tilde \pi_n^{-1}$ guarantees that any point $(\tilde \pi_n^{-1}(r), \ceil{r})$ lies along the path $\pi$, see Section \ref{SS:lpp} for the definition of a point lying along a path. Therefore the geodesic property of $\pi$ guarantees that for any partition $r_0 < \dots < r_k$ of $[-\floor{sn} - 1/2, -\floor{tn}]$ we have
\begin{equation}
\label{E:pitilde}
B[(x, t)_n \to (y,s)_n] = \sum_{i=1}^k B[(\tilde \pi_n^{-1}(r_{i}), \ceil{r_{i}}) \to (\tilde \pi_n^{-1}(r_{i-1}), \ceil{r_{i-1}})].
\end{equation}
Now define the function $\sig_n:[\floor{tn}/n, (\floor{sn} + 1/2)/n] \to \R$ by
$$
\sig_n(r) = \frac{n(\tilde \pi_n^{-1}(-rn) - r)}{2n^{2/3}}.
$$
By \eqref{E:pitilde}, the function $\sig_n$ is a geodesic in $\scrL_n$ from $(x + O(n^{-2/3}), t + O(n^{-1}))$ to $(y + O(n^{-2/3}), s + O(n^{-1}))$.
Therefore by Theorem \ref{T:geod-cvg}, in a coupling where $\scrL_n \to \scrL$ almost surely and $\scrL$ has a unique directed geodesic $\Pi$ from $(x, t)$ to $(y, s)$, the functions $\sig_n$ will converge to $\Pi$ in $\mathcal{G}$. That is, the graphs of $\sig_n$ will converge to the graph of $\Pi$ in the Hausdorff topology. We want to show that in this case, the functions $\tau_n$ converge to $\Pi$ uniformly. This is accomplished by the following lemma.

\begin{lemma}
\label{L:det-cvg-pat}
With all functions defined as above, suppose that $\sig_n$ converges almost surely in $\mathcal{G}$ to a function $\sig:[t, s] \to \R$. Then $\tau_n \to \sig$ uniformly almost surely.
\end{lemma}

\begin{proof}
Fix a point $\om$ in the probability space such that $\sig_n(\om) \to \sig(\om)$ in $\scrG$. We show that $\tau_n(\om) \to \sig(\om)$ uniformly.
We will work with the set
$$
\Ga_n = \{(\tau_n(r), h_n(r) - 2\tau_n(r) n^{-1/3}) : r \in [t,s]\}.
$$
Let $\al_n$ be the linear map from $\R^2 \to \R^2$ given by $\al_n(x, t) = (t + 2xn^{-1/3}, -tn)$. We have
$$
\al_n \Ga_n = \{(h_n(r), \pi_n \circ h_n(r)) : r \in [t,s] \}.
$$
The set $\al_n\Ga_n$ is the graph of $\pi_n$.  Recall that 
$$
\mathfrak{g} \sig_n = \{(\sig_n(r), r) : r \in [t, s]\}.
$$
Letting $\beta(x, s) = (x, \ceil{s})$, we also have that 
$$
\beta\al_n \mathfrak{g} \sig_n = \{ (\tilde \pi_n^{-1}(r), \ceil{r}) : r \in [-\floor{sn} - 1/2, -\floor{tn}]\}. 
$$
The set $\beta\al_n \mathfrak{g} \sig_n$ consists of all points that lie along $\pi_n$.
In particular, $\al_n \Ga_n \sset \beta\al_n \mathfrak{g} \sig_n$, and so
\begin{equation}
\label{E:Gann}
    \Ga_n \sset \al_n^{-1} \beta \al_n \mathfrak{g} \sig_n.
\end{equation}
A straightforward computation yields that 
\begin{equation}
\label{E:unif-norm}
||\al_n^{-1}\beta \al_n(u) - u|| \le 2 n^{-1/3}  
\end{equation}
for all $u \in \R^2$. Thus all points in $\Gamma_n$ are in the $2n^{-1/3}$-neighborhood of $\mathfrak{g}\sigma_n$. The Hausdorff convergence of $\mathfrak{g}\sig_n$ to $\mathfrak{g} \sig$ then implies that the sequence of closures $\bar \Ga_n$ has subsequential limits and they are all subsets of $\mathfrak{g} \sig$.

\medskip

Since $\tau_n(r)$ is a coordinate of $\Ga_n$ for all $r, n$ this implies that there exists a constant $C(\om)$ such that $|\tau_n(r)| \le C$ for all $n \in \N, r \in [t, s]$. Also, since $h_n$ is the affine map from $[t + 2xn^{-1/3}, s + 2yn^{-1/3}] \to [t, s]$, we have that $h_n (r) = r + O(n^{-1/3})$. Therefore
\begin{equation}
\label{E:tau-h}
  (\tau_n(r), h_n(r) - \tau_n(r)n^{-1/3}) = (\tau_n(r), r + O(n^{-1/3})).
\end{equation}
Equation \eqref{E:tau-h} implies that any subsequential limit of $\bar \Ga_n$ must intersect each of the lines $\R \X \{r\}$ for every $r \in [t, s]$. Since $\mathfrak{g} \sig$ intersects each of these lines exactly once and all subsequential limits of $\bar \Ga_n$ are subsets of $\mathfrak{g} \sig$, we get that $\bar \Ga_n \to \mathfrak{g} \sig$. 

\medskip

Equation \eqref{E:tau-h} also implies that the Hausdorff distance between $\close{\Ga_n}$ and $\close{\mathfrak{g} \tau_n}$ is $O(n^{-1/3})$, and so $\close{\mathfrak{g} \tau_n} \to \mathfrak{g} \sig$ as well. By the equivalence of Hausdorff convergence of graphs and uniform convergence to continuous functions, this completes the proof.
\end{proof}

Theorem \ref{T:geod-cvg} combined with Lemma \ref{L:det-cvg-pat} immediately allows us to conclude joint convergence of last passage paths to directed geodesics. This next theorem is a restatement of Theorem \ref{T:intro-cvg-paths}.

\begin{theorem}
\label{T:blppgeod}
	Let the Brownian last passage percolations $\scrL_n$ and $\scrL$ be coupled so that $\scrL_n \to \scrL$ uniformly on compact sets almost surely. Then there exists an event $A$ of probability $1$ with the following property. For $u = (x, t; y, s) \in \R^4_\uparrow$, let $C_u$ be the set where the directed geodesic $\Pi_u$ is unique in $\scrL$, and let $h_{u, n}$ be the increasing affine function mapping $[t, s]$ onto $[t + 2xn^{-1/3}, s + 2yn^{-1/3}]$.
	
\medskip

For any $u \in \R^4_\uparrow$, and any sequence of last passage paths $\pi_{u, n}$ from $(x, t)_n$ to $(y, s)_n$ in $\scrL_n$, we have that
$$
\frac{\pi_{u, n} \circ h_{u, n} + n h_{u, n}}{2n^{2/3}} \to \Pi_u \quad \text{uniformly on the almost sure event }A\cap C_u.
$$
\end{theorem}

In Theorem \ref{T:blppgeod}, $A$ is the event where $\scrL$ is a proper landscape. This guarantees that we can apply Theorem \ref{T:geod-cvg} on $A \cap C_u$. Note that the set $C_u$ is not obviously a Borel measurable event in itself. However, when $\scrL$ is a proper landscape, uniqueness of the directed geodesic from $(x, t)$ to $(y, s)$ is equivalent to the statement that each of the continuous functions
$$
f_r(z) = \scrL(x, t; z, r) + \scrL(z, r; y, s)
$$
have a unique maximum, for $r \in (t, s) \cap \Q$. Therefore $C_u \cap A$ is Borel measurable since on the event $A$, uniqueness of geodesics can be checked with only countably many conditions.

\section{Open questions}

There are several natural open questions related to the directed landscape. We collect a few of these here. Some of these are geometric in nature, while others are about desired explicit formulas.

\medskip

Definition \ref{D:airy-sheet} constructs the Airy sheet on $\R_+ \X \R$ as a deterministic function of the Airy line ensemble $\scrA$. We believe that the whole Airy sheet should be a deterministic function of $\scrA$ in the following way. Define $\scrH:\R^2 \to \R$ so that $\scrH|_{(0,\infty) \X \R}$ is defined by \eqref{E:buse} and $\scrH(0, y) = \scrA_1(y)$. Now define $\scrH(-x, y)$ for $x > 0$ by applying the same formulas to the reflected version of $\scrA$: $\scrA_\cdot(-\,\cdot)$. This defines  $\scrH$ on $\mathbb R^2$ as a deterministic functional of $\scrA$.
\begin{conj}
	\label{C:H-Airy-sheet}
The function $\scrH$ is an Airy sheet.
\end{conj}

Since the first version of this paper appeared, we proved Conjecture \ref{C:H-Airy-sheet}, see Theorem 1.21 in \cite{dauvergne2021scaling}.

\medskip

Another way to clarify the Airy sheet definition would be to give the sheet values directly as a limit. The following would suffice.
\begin{conj}
	\label{P:det-correction}
There exists a deterministic function $a:\R^+\times \N\to \R$ so that for every $x>0$, almost surely as $k\to\infty$ we have
$$
\scrA[(-\sqrt{k/(2x)},k)\LP (0,1)]-a(x,k) \to \scrS(x,0).
$$
\end{conj}
Conjecture \ref{P:det-correction} would follow from an improved bound on the last passage problem across the Airy line ensemble that was analyzed in Section \ref{S:short-lp}. In that section, we show that 
$$
\scrA[(0, k) \to (x, 1)] = 2\sqrt{2kx} + o(\sqrt{k}).
$$
We believe that the correct error term here is $O(k^{-1/6})$, as in Brownian last passage percolation. This would give that $a(x, k) = \E \scrA_1(0)- \E \scrA_k(0) - \sqrt{2kx}.$ The first term is the expectation of a GUE Tracy-Widom random variable and the second term is equal to $-(3\pi k/2)^{2/3} + o(1)$ as $k \to \infty$.

\medskip

Consider the directed geodesics $\Pi_{x,y}$ from time $(x,0)$ to time $(y,1)$. By invariance, the following intersection question has only two parameters.
\begin{problem}
Find a formula for the probability that $\Pi_{0,0}$ and $\Pi_{x,y}$ intersect.
\end{problem}
%

 Even the most basic distributions related to the directed geodesic $\Pi = \Pi_{0,0}$ are not known.  The first question depends only on two Airy processes, so it should be doable using continuum statistics.
\begin{problem} \hfill

(a) Find the distribution of $\Pi(s)$ for all $s$.

(b) Find the distribution of $\max_s \Pi(s)$.

\end{problem}
Proposition \ref{P:mod-dg} gives tail bounds on these quantities of the form $e^{-x^3}$.

\medskip

We also believe that the beginning of geodesics are special.
\begin{conj} Consider the process $\eta_t:[0,1/2]\to \R$ defined by
$$\eta_t(s)=\Pi(t+s)-\Pi(t).$$
Let $0\le t<u<1/2$.  The law of $\eta_t$ and $\eta_u$ are mutually absolutely continuous if and only if $t>0$.
\end{conj}

Finally, there should be a stochastic calculus for the 1-2-3 scaling. A natural precise question for the development of this theory is the following.
\begin{question}
Let $h: \R^2 \to \R$ be a smooth function with compact support. Define the $h$-shift $\scrL'$ of the directed landscape $\scrL$ by the length formula:
$$
\int d \scrL' \circ \pi =\int d \scrL \circ \pi + \int h(\pi(t),t)\,dt
$$
Is the distribution of $\scrL'$ absolutely continuous with respect to the directed landscape? If so, what is the density?
\end{question}

\bigskip

\noindent {\bf Acknowledgments.} D.D. was supported by an NSERC CGS D scholarship. B.V. was supported by the Canada
Research Chair program, the NSERC Discovery Accelerator grant, the MTA Momentum
Random Spectra research group, and the ERC consolidator grant 648017 (Abert). We thank Robert Haslhofer for the reference \cite{perelman2002entropy} and Firas Rassoul-Agha and Chris Janjigian for pointing out that the equation \eqref{E:buse} holds for all triples $(x, y, z) \in \R^+ \X \R^2$ simultaneously, rather than just rational triples. We thank Ivan Corwin, Patrik Ferrari, Alan Hammond, Yun Li, Mihai Nica, Dmitry Panchenko, Jeremy Quastel, Timo Sepp\"al\"ainen, Xiao Shen, Benedek Valk\'o, and Peter Winkler for valuable comments about previous versions. B.V. thanks the organizers for the 2008 workshop {\it Random matrices: probabilistic aspects and applications} at the Hausdorff Institute,
Bonn, where these problems where highlighted.

\bibliographystyle{dcu}
\bibliography{ALEcitations2}

\bigskip\bigskip\noindent

\noindent Duncan Dauvergne, Department of Mathematics, Princeton University, United States,\\ {\tt dd18@math.princeton.edu}

\bigskip

\noindent Janosch Ortmann,
D\'epartement de management et technologie,
\'Ecole des sciences de gestion,
Universit\'e du Qu\'ebec \`a Montr\'eal, Canada,
{\tt ortmann.janosch@uqam.ca}

\bigskip

\noindent B\'alint Vir\'ag, Departments of Mathematics and Statistics, University of Toronto, Canada,\\ {\tt balint@math.toronto.edu}
\end{document}